\documentclass{article}

\usepackage{amsmath}
\usepackage{graphics}
\usepackage{graphicx}
\usepackage{amsmath,amsfonts,multirow}
\usepackage{amsthm}
\usepackage{hyperref}
\usepackage[capitalize]{cleveref}
\usepackage{url}
\usepackage{epstopdf}
\usepackage{amsopn}
\usepackage{subcaption}
\usepackage[algo2e,ruled,boxed,vlined,linesnumbered]{algorithm2e}
\SetKwRepeat{Do}{do}{while}
\usepackage{multirow}
\usepackage{enumerate}
\usepackage{amssymb}
\usepackage{placeins}
\Crefname{algocf}{Algorithm}{Algorithms}
\usepackage{float}
\usepackage[margin=1in]{geometry}

\newcommand{\R}{\mathbb{R}}
\newcommand{\C}{\mathbb{C}}

\newcommand{\eps}{\varepsilon}

\DeclareMathOperator*{\SRHT}{SRHT}
\DeclareMathOperator*{\SJLT}{SJLT}
\DeclareMathOperator*{\Gaussian}{Gaussian}
\DeclareMathOperator*{\trace}{trace}
\DeclareMathOperator*{\range}{range}
\DeclareMathOperator*{\diagmat}{diag}
\DeclareMathOperator*{\Binomial}{Binomial}
\newcommand{\abs}[1]{{\lvert #1\rvert}}
\newcommand{\norm}[1]{{\left\| #1\right\|}}
\newcommand{\normsq}[1]{\norm{#1}^2}
\newcommand*{\horzbar}{\rule[.5ex]{2.5ex}{0.5pt}}

\newcommand{\inc}{{\textsc{Incrementing}}}

\usepackage{pgfplots}
\pgfplotsset{width=10cm,compat=1.9}

\usepgfplotslibrary{external}

\newcommand{\epsa}[0]{\varepsilon_{\textup{abs}}}
\newcommand{\epsr}[0]{\varepsilon_{\textup{rel}}}

\newcommand{\bze}{\mathbf{O}}

\newtheorem{remark}{Remark}

\newtheorem{lemma}{Lemma}
\newtheorem{theorem}{Theorem}
\newtheorem{definition}{Definition}

\begin{document}

\title{Construction of Hierarchically Semi-Separable matrix Representation using Adaptive Johnson--Lindenstrauss Sketching}

\author{
Yotam Yaniv\footnote{Corresponding author, email: yotamy@lbl.gov} \footnotemark[2], Pieter Ghysels~\footnote{Lawrence Berkeley National Laboratory}, Osman Asif Malik~\footnotemark[\value{footnote}],
Henry A. Boateng~\thanks{San Francisco State University},
Xiaoye S. Li~\footnotemark[2]
}

\maketitle

\begin{abstract}
 We present an extension of an adaptive, partially matrix-free, Hierarchically Semi-Separable (HSS) matrix construction algorithm by Gorman et al.\ [SIAM J.\ Sci.\ Comput.\ 41(5), 2019]
 which uses Gaussian sketching operators to a broader class of Johnson--Lindenstrauss (JL) sketching operators. We develop theoretical work which justifies this extension.
 In particular, we extend the earlier concentration bounds
 to all JL sketching operators and examine this bound for specific classes of such operators including the original Gaussian sketching operators, 
 subsampled randomized Hadamard transform (SRHT)
 and the sparse Johnson--Lindenstrauss transform (SJLT). 
 We discuss the implementation details of applying SJLT and SRHT efficiently. Then we demonstrate experimentally that using SJLT or SRHT instead of Gaussian sketching operators leads to up to 2.5$\times$ speedups of the serial HSS construction implementation in the STRUMPACK C++ library. Additionally, we discuss the implementation of a parallel distributed HSS construction that leverages Gaussian or SJLT sketching operators. We observe a performance improvement of up to 35$\times$ when using SJLT sketching operators over Gaussian sketching operators. 
 The generalized algorithm allows users to select their own JL sketching operators with theoretical lower bounds on the size of the operators which may lead to faster run time with similar HSS construction accuracy. \end{abstract}


\section{Introduction}

Many large dense matrices in engineering and data sciences are \emph{data-sparse} in that 
the off-diagonal blocks can be well approximated as low-rank submatrices.
Some examples are matrices from discretized integral equations, boundary element methods, and
kernel matrices used in statistical and machine
learning~\cite{Bebendorf08,KRR-ipdps20}.
There are many types of matrix formats that can take advantage of the
off-diagonal low-rank structure; these include, to name a few,
Hierarchically Semi-Separable  matrices (HSS)~\cite{chandrasekaran2005fast,chandrasekaran2006fast},
Hierarchical matrices ($\mathcal{H}$) and Hierarchical Bases $\mathcal{H}$-matrices ($\mathcal{H}^{2}$)~\cite{hac00,hac02}.
This work focuses on HSS representation and, more specifically,
efficient HSS compression, i.e., construction of the HSS format.
Compression is the central component of the HSS framework, and
usually dominates the total cost.
Once a matrix is compressed into its HSS form, one can develop
asymptotically faster algorithms for multiplication, factorization
and solve based on the HSS structure.
One way to speed up the HSS compression algorithm is to use 
randomization~\cite{martinsson2011fast,halko2011},
in particular, randomized sketching.
The main advantage of randomization is that these methods
usually require fewer floating point operations and less communication
than their traditional deterministic counterparts.
Moreover, they are often easier to parallelize.

Consider a matrix $A \in \mathbb{C}^{n \times n}$ to be compressed
as an HSS matrix that approximates $A$. Randomized sketching can be
considered as a preprocessing step that helps compute the column
spaces of various off-diagonal submatrices throughout the compression algorithm.
This preprocessing step is done by post-multiplying $A$ by a tall-and-skinny random matrix $R$ of size $n \times (r+p)$: $S \leftarrow A R$.
If $A$ is nonsymmetric, the row space must be computed separately which requires an additional preprocessing step of the form $S' \leftarrow A^* R$.
The coefficient $r$ is an upper bound
on the numerical ranks of the off-diagonal blocks and $p$ is an oversampling parameter, a small integer on the order
of 10 or so.
The entries of the $n \times (r+p)$ matrix $R$ are drawn from a certain probability distribution.
A common choice is to draw the entries of $R$ independently from an appropriately scaled normal distribution.
The cost of matrix multiplication
$A R$ is $O(n^2 d)$, where $d = r+p$ while the remaining cost of the compression algorithm is $O(nr^2)$, therefore this upfront matrix multiplication is often the bottleneck in the entire compression algorithm.

This paper builds upon our previous 
work~\cite{ghysels2017-ipdps,gorman2019robust}.
The first motivation is to mitigate the $O(n^2 d)$ cost in the sketching step.
To this end, we study alternative random sketching operators,
with a focus on the sparse Johnson--Lindenstrauss transform (SJLT) 
and the subsampled randomized Hadamard
transform (SRHT) \cite{AilonChazelle06, KaneNelson14}.
SJLT and SRHT are asymptotically faster to apply than Gaussian sketching operators, but research is needed to understand whether they provide desired approximation quality, and what the time and accuracy trade offs are.
Secondly, one of the highlights of~\cite{gorman2019robust} is the development
of a new stopping criterion for adaptive sketching, which is needed because the numerical HSS rank $r$ is usually not known {\it a priori}.
The stopping criteria adaptivity ensures that we generate sufficient (for robustness),
yet not too many (for high performance),
random sketching operators (columns of $R$) until the range of $A$ is well approximated.
The stopping criterion in~\cite{gorman2019robust} is based on a probabilistic Frobenius norm estimation of $A$ by the sketch matrix $S = AR$ and concentration bounds when sketching with Gaussian sketching operators.
This analysis leads to a robust stopping criterion taking into account
both absolute and relative errors.
In this paper, we present theoretical analysis which justifies more general JL sketching operators.
We extend the concentration bounds discussed in~\cite{gorman2019robust} to
all real JL sketching operators and examine this bound for the original Gaussian sketching operators, SRHT operators
and SJLT operators.

\begin{remark}
In most literature on randomized sketching, the sketching operator $R$ is applied on the left of a vector or a matrix, such as $RA$. But in the HSS construction, we need to apply $R$ on the right of $A$ to probe its column space. Therefore, in the HSS context, we use the transpose of sketching operators described in existing JL theory.
\end{remark}

The contributions of this work are:
\begin{itemize}
    \item We generalize an adaptive HSS compression algorithm presented in Gorman et al. \cite{gorman2019robust} that required Gaussian sketching operators to any Johnson--Lindenstrauss (JL) sketching operators.
    \item We show that the Frobenius norm stopping criteria from Gorman et al. \cite{gorman2019robust} are still valid for JL sketching operators and prove Frobenius norm bounds for JL sketching operators and SJLT.
    \item We prove range-finder bounds for JL sketching operators and Sparse Johnson-Lindenstrauss Transforms (SJLT) which state that the sketch $S = A R$ for a low rank matrix $A$ contains relevant range information of the original matrix. This allows us to use the sketch instead of the original block when doing HSS compression.
    \item We implement our general HSS compression algorithm in the STRUMPACK C++ library \cite{strumpack_web} which allows the user to choose among sketching operators implemented in STRUMPACK or implement their own.
    We implement SJLT and SHRT as specific use cases and discuss the implementation details for SJLT in which we leverage a special data structure and multiplication routines for computing $A R$ and $A^* R$ and for SHRT which we develop an efficient multiplication routine.
    \item We compare our serial method using SJLT, SRHT and the existing Gaussian sketching operators and observe up to 2.5$\times$ speedups when using SJLT or SRHT while maintaining the similar compression accuracy. The number of flops for SJLT is reduced from $O(n^2 d)$ to $O(n\alpha d)$, where $\alpha \ll d$; usually $\alpha = 2$ to $4$ is sufficient.
    \item We implement and compare a distributed (Message Passing Interface) implementation for Gaussian and SJLT sketching operators. We observe that the sketching time may be improved by a factor of 40 in some cases when using SJLT over Gaussian sketching operators and overall compression is sped up by a factor of up to 35$\times$.
\end{itemize}

The rest of the paper is organized as follows. In the end of this section we outline the notation for the rest of the paper. In \cref{sec:prelim} we discuss the background on HSS matrices, our HSS compression algorithm, \cref{algo::HSScompressNodeAdaptive}, which we generalize from \cite{gorman2019robust} and the Johnson--Lindenstrauss sketching operators which we use in our generalization. Next,  in \cref{sec:adaptivestoppingcrit} we discuss the adaptive stopping criteria in \cref{algo::HSScompressNodeAdaptive} which leverage a Frobenius norm stopping criteria. Then in \cref{sec:frotheory} we prove that the Frobenius norm stopping criteria generalize to all Johnson--Lindenstrauss sketching operators. In \Cref{sec:rangefinderBounds} we prove range-finder bounds for JL sketching operators and SJLT sketching operators; these results enable us to use the sketch instead of the full low rank blocks in the compression.
\Cref{sec:cstheory} discusses the implementation details of using SJLT, followed by \cref{sec:srht_implement}, which outlines the implementation of SRHT. Afterwards, in \cref{sec:experiment} we conduct experiments comparing SJLT, SRHT and Gaussian sketching showing similar compression errors and faster compression when using SJLT or SRHT. Additionally, we discuss and experimentally compare the parallel distributed implementations for Gaussian and SJLT sketching.
Finally, in \cref{sec:conclusion} we state our concluding remarks.
\section{Preliminaries} \label{sec:prelim}
We begin this section by describing the HSS matrix format and the adaptive HSS construction algorithm. We then discuss the relevant background to incorporate a more general and possibly faster randomization via Johnson--Lindenstrauss sketching in our HSS construction algorithm. 

\subsection{Notation}
We denote a matrix as $A \in \C^{m \times n}$. 
We let a random \textit{sketching operator} be denoted as $R \in \R^{n \times d}$ and vectors $x \in \R^n$.
We refer to $S = AR$ as a \textit{sketch} of the matrix $A$. Sketching is the process of applying $R$ to $A$ on the right, computing $AR$. We use $\log$ to represent the logarithm with base $e$. We let $\norm{A},\; \norm{x}$ be the matrix and vector two-norm respectively. We let $\norm{A}_F$ represent the Frobenius norm of a matrix. We define $[n] = (1:n) = \{1,...n\}$ to be the set of integers from one to $n$. 
We use MATLAB notation to represent indexing a row, a column or a sub-block of our matrix,
where lower case $(i,\;j)$ represents individual entries and upper case $(I,\;J)$ represents index sets.
For example $A(i,j)$ is entry $(i,j)$ of matrix $A$, $A(i,:)$ is row $i$ of matrix $A$ and $A(I,J)$ is the sub-block of $A$ containing the rows in index set $I$ and columns in index set $J$. In the theory section to compress this notation we use $A_{i:}$ to represent the row $i$ of matrix $A$ and  $A_{:j}$ to represent column $j$ of matrix $A$.
When computing a QR factorization for a matrix $A$ we let $A = Q \Omega$ where $Q$ is an orthogonal matrix and $\Omega$ is upper triangular. An interpolative decomposition of a matrix $A$ with rank $r$ is computed as $A \approx A(:,J)U$ where $J$ is an index set of size $r$ and $U$ is an $r \times n$ matrix containing an $r \times r$ identity block. Finally, the projection operator onto a matrix $S$ is defined as $P_S = SS^\dagger$.

\subsection{Background on HSS Matrices} 

\usetikzlibrary{decorations.shapes}
\usetikzlibrary{decorations.pathreplacing}
\tikzstyle{cnode} = [draw, circle,scale=.75]
\tikzstyle{level 1} = [level distance=.2\textwidth, sibling distance=.45\textwidth]
\tikzstyle{level 2} = [level distance=.15\textwidth, sibling distance=.25\textwidth]
\tikzstyle{level 3} = [sibling distance=.1\textwidth]

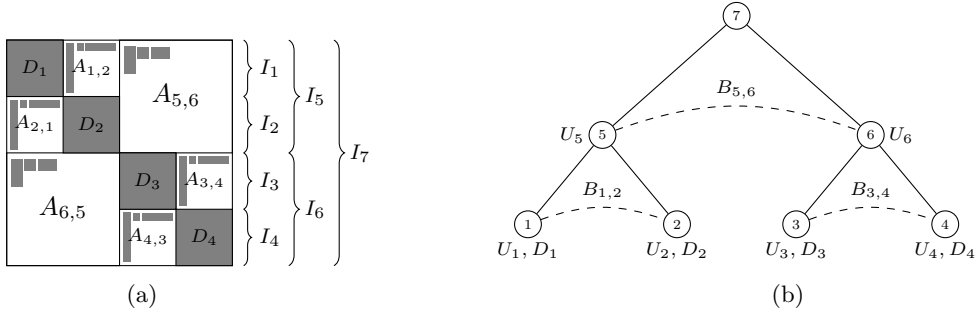
\begin{figure}
  \centering
  \begin{subfigure}[b]{.48\columnwidth}
    \centering
    \begin{tikzpicture}[scale=3]
      \path[use as bounding box] (0,0) rectangle (1.2,1); 
      \draw (0,0) rectangle (1,1);
      \draw (0,0) rectangle (0.5,0.5);
      \draw (1,1) rectangle (0.5,0.5);
      \draw (0,1) rectangle (0.25,0.75);
      \draw [fill=gray] (0,1) rectangle (0.25,0.75);
      \draw (0.5,0.5) rectangle (0.75,0.25);
      \draw [fill=gray] (0.5,0.5) rectangle (0.75,0.25);
      \draw [fill=gray] (0.25,0.75) rectangle (0.5,0.5);
      \draw [fill=gray] (0.75,0.25) rectangle (1,0);
      \path [fill=gray] (0.02,0.35) rectangle (0.07,0.47);
      \path [fill=gray] (0.08,0.42) rectangle (0.13,0.47);
      \path [fill=gray] (0.14,0.42) rectangle (0.22,0.47);

      % A6,3
      \node [label] at (0.25,0.25) {$A_{6,5}$};
      \node [label] at (0.75,0.75) {$A_{5,6}$};
      \path [fill=gray] (0.02+0.5,0.35+0.5) rectangle (0.07+0.5,0.47+0.5);
      \path [fill=gray] (0.08+0.5,0.42+0.5) rectangle (0.13+0.5,0.47+0.5);
      \path [fill=gray] (0.14+0.5,0.42+0.5) rectangle (0.22+0.5,0.47+0.5);

      % A2,1
      \path [fill=gray] (0.015,0.5+0.015) rectangle (0.05,0.235+0.5);
      \path [fill=gray] (0.06,0.2+0.5) rectangle (0.09,0.235+0.5);
      \path [fill=gray] (0.1,0.2+0.5) rectangle (0.25-0.015,0.235+0.5);

      % A1,2
      \path [fill=gray] (0.015+0.25,0.015+0.75) rectangle (0.05+0.25,0.235+0.5+0.25);
      \path [fill=gray] (0.06+0.25,0.2+0.75) rectangle (0.09+0.25,0.235+0.5+0.25);
      \path [fill=gray] (0.1+0.25,0.2+0.75) rectangle (0.5-0.015,0.235+0.5+0.25);

      % A4,5
      \path [fill=gray] (0.5+0.015,0.015) rectangle (0.5+0.05,0.235+0.5-0.5);
      \path [fill=gray] (0.5+0.06,0.2+0.5-0.5) rectangle (0.5+0.09,0.235+0.5-0.5);
      \path [fill=gray] (0.5+0.1,0.2) rectangle (0.75-0.015,0.235+0.5-0.5);

      % A5,4
      \path [fill=gray] (0.5+0.015+0.25,0.25+0.015) rectangle (0.5+0.05+0.25,0.235+0.5+0.25-0.5);
      \path [fill=gray] (0.5+0.06+0.25,0.2+0.25) rectangle (0.5+0.09+0.25,0.235+0.5+0.25-0.5);
      \path [fill=gray] (0.5+0.1+0.25,0.2+0.25) rectangle (1.0-0.015,0.235+0.5+0.25-0.5);

      \node [label] at (0.125/2*6,1-0.125/2*2) {\scriptsize $A_{1,2}$};
      \node [label] at (0.125/2*2,1-0.125/2*6) {\scriptsize $A_{2,1}$};
      \node [label] at (0.125/2*10,1-0.125/2*14) {\scriptsize $A_{4,3}$};
      \node [label] at (0.125/2*14,1-0.125/2*10) {\scriptsize $A_{3,4}$};
      \node [label] at (0.125,1-0.125) {\scriptsize $D_{1}$};
      \node [label] at (0.125*3,1-0.125*3) {\scriptsize $D_{2}$};
      \node [label] at (0.125*5,1-0.125*5) {\scriptsize $D_{3}$};
      \node [label] at (0.125*7,1-0.125*7) {\scriptsize $D_{4}$};

      \draw [decorate,decoration={brace,amplitude=3pt},yshift=0pt,xshift=-.1cm] (1.15,1-0.25*0) -- (1.15,1-0.25) node [black,midway,xshift=.35cm] {\footnotesize $I_1$};
      \draw [decorate,decoration={brace,amplitude=3pt},yshift=0pt,xshift=-.1cm] (1.15,1-0.25*1) -- (1.15,1-0.25*2) node [black,midway,xshift=.35cm] {\footnotesize $I_2$};
      \draw [decorate,decoration={brace,amplitude=3pt},yshift=0pt,xshift=-.1cm] (1.15,1-0.25*2) -- (1.15,1-0.25*3) node [black,midway,xshift=.35cm] {\footnotesize $I_3$};
      \draw [decorate,decoration={brace,amplitude=3pt},yshift=0pt,xshift=-.1cm] (1.15,1-0.25*3) -- (1.15,1-0.25*4) node [black,midway,xshift=.35cm] {\footnotesize $I_{4}$};
      \draw [decorate,decoration={brace,amplitude=4pt},yshift=0pt,xshift=-.1cm] (1.35,1-0.5*0) -- (1.35,1-0.5) node [black,midway,xshift=.35cm] {\footnotesize $I_5$};
      \draw [decorate,decoration={brace,amplitude=4pt},yshift=0pt,xshift=-.1cm] (1.35,1-0.5*1) -- (1.35,1-0.5*2) node [black,midway,xshift=.35cm] {\footnotesize $I_{6}$};
      \draw [decorate,decoration={brace,amplitude=4pt},yshift=0pt,xshift=-.1cm] (1.55,1) -- (1.55,0) node [black,midway,xshift=.35cm] {\footnotesize $I_{7}$};

    \end{tikzpicture}
    \caption{}
    \label{fig:HSS}
  \end{subfigure}
  \hfill
  \begin{subfigure}[b]{.48\textwidth}
    \begin{tikzpicture}\scriptsize
      \node[cnode](7){7}
      child{node[cnode](5){5} child{node[cnode](1){1}} child{node[cnode](2){2}} }
      child{node[cnode](6){6} child{node[cnode](3){3}} child{node[cnode](4){4}} }  ;
      % Level-arcs and their labels
      \foreach \ln / \rn in {1/2,3/4,5/6}{
        \draw[dashed] (\ln) to [out=20,in=160]
           node[above]{\(B_{\ln,\rn}\)} (\rn);
      }
      % Leaves labels
      \foreach \ln in {1,2,3,4}{ \node[below=.5em] at (\ln) {\(U_{\ln},D_{\ln}\)}; }
      \foreach \ln in {5}{  \node[left=.5em] at (\ln) {\(U_{\ln}\)}; }
      \foreach \ln in {6}{ \node[right=.5em] at (\ln) {\(U_{\ln}\)}; }
    \end{tikzpicture}
    \caption{}
    \label{fig:HSStree}
  \end{subfigure}
  \caption{(a) Illustration of a symmetric HSS matrix using $3$
  levels. Diagonal blocks are partitioned recursively. Gray blocks
  denote the basis matrices. (b) Tree for the HSS matrix from (a),
  using topological ordering. All nodes except the root store $U_i$ (and
  $V_i$ for the non-symmetric case). Leaves store $D_i$, non-leaves $B_{ij}$ (and $B_{ji}$ for the non-symmetric case).}
\end{figure}

Consider a square matrix $A \in \mathbb{C}^{n \times n}$ and index set
$I_{A} = \{1,\dots,n \}$. The HSS matrix representation is a hierarchical block
$2 \times 2$ partitioning of the matrix, where all off-diagonal blocks are compressed,
or approximated, using a low-rank product, see \cref{fig:HSS}. 
The hierarchical structure is succinctly described by a binary tree $\mathcal{T}$,
called {\em cluster tree}, as depicted in~\cref{fig:HSStree}. The recursive partitioning stops
at the leaf level, which corresponds to the smallest block size of the partition.
The leaves do not need to be of uniform size, because for certain input matrices
a non-uniform partition may be preferable for smaller numerical ranks. 

Each node $\tau\in \mathcal{T}$ is associated with a contiguous subset
$I_\tau \subset I_{\textrm{root}(\mathcal{T})} $.
We use $\#I_\tau$ to denote the cardinality of $I_{\tau}$.
For two children $\nu_1$ and $\nu_2$ of $\tau$, it
holds that $I_{\nu_1} \cup I_{\nu_2} = I_\tau$ and $I_{\nu_1} \cap
I_{\nu_2} = \emptyset$. It follows that $\cup_{\tau \in
  \textrm{leaves}(\mathcal{T})} I_\tau =
I_{\textrm{root}(\mathcal{T})} = I_A$. The same tree $\mathcal{T}$ is
used for the rows and the columns of $A$.
Commonly, the tree nodes are numbered in a {\em postorder},
and most of the HSS algorithms, such as construction, matrix-vector multiplication,
factorization and solve etc., can be described as traversing the
cluster tree following this postorder.
However, in the parallel implementation and throughout this paper, we traverse the
cluster tree following a bottom-up {\em topological order}, i.e., level by level from
the leaf level to the root, see~\cref{fig:HSStree}.

Each leaf node $\tau$ of $\mathcal{T}$ corresponds to a
diagonal blocks of $A$, denoted as $D_\tau$, and is stored as
a dense matrix : $D_{\tau} = A(I_{\tau},I_{\tau})$.
At each node $\tau$, the off-diagonal block $A(I_\tau, I_A\setminus I_\tau)$
is called a row {\em Hankel block}, and the off-diagonal block
$A(I_A\setminus I_\tau, I_\tau)$ is a column Hankel block.
The compression algorithm sweeps through the tree bottom-up. At each tree node,
it computes the column basis for the row Hankel block and row basis for
the column Hankel block. Note that all the blocks within a row (column)
Hankel block share the same column (row) basis.
The HSS algorithm goes further to reduce complexity: each internal node
recycles the bases computed at the two children nodes. Thus, the
basis at each internal node has the nested structure
(see Equation~\cref{eq:nestedUV}), called {\em nested basis} property, which we describe now.
For a node $\tau$ with two children $\nu_1$ and $\nu_2$, 
the off-diagonal block $A_{\nu_1,\nu_2} = A(I_{\nu_1},I_{\nu_2})$
is factored (approximately) as
\begin{equation}
  A_{\nu_1,\nu_2} \approx U^{\mathrm{big}}_{\nu_1} B_{\nu_1,\nu_2} \left( V_{\nu_2}^{\mathrm{big}}\right)^*  \, ,
\end{equation}
where $U^{\mathrm{big}}_{\nu_1}$ has dimensions $\#I_{\nu_1} \times r^r_{\nu_1}$, $B_{\nu_1,\nu_2}$ is a submatrix of $A_{\nu_1,\nu_2}$ with dimensions  $\#I_{\nu_1} \, \times \, \#I_{\nu_2}$
and $V_{\nu_2}^{\mathrm{big}}$ has dimensions $\#I_{\nu_2} \, \times \,
r^c_{\nu_2}$\footnote{Superscripts $r$ and $c$ are used to denote that
  $U^{\mathrm{big}}$/$V^{\mathrm{big}}$ are column/row bases for the
  row/column Hankel blocks of $A$.}. The {\em HSS-rank} $r$ is a numerical rank defined as
the maximum of $r^r_\tau$ and $r^c_\tau$ over all off-diagonal blocks, where typically $r \ll
N$. $B_{\nu_1,\nu_2}$ and $B_{\nu_2,\nu_1}$ are stored at the parent
node. For a node $\tau$ with children $\nu_1$ and $\nu_2$,
$U^{\mathrm{big}}_\tau$ and $V^{\mathrm{big}}_\tau$ are represented
hierarchically as
\begin{equation}\label{eq:nestedUV}
  U^{\mathrm{big}}_{\tau} = \begin{bmatrix} U^{\mathrm{big}}_{\nu_1} & 0 \\ 0 & U^{\mathrm{big}}_{\nu_2} \end{bmatrix} U_\tau
  \quad \textrm{and} \quad
  V^{\mathrm{big}}_{\tau} = \begin{bmatrix} V^{\mathrm{big}}_{\nu_1} & 0 \\ 0 & V^{\mathrm{big}}_{\nu_2} \end{bmatrix} V_\tau \, .
\end{equation}
Note that for a leaf node $U^{\mathrm{big}}_\tau = U_\tau$ and
$V^{\mathrm{big}}_\tau = V_\tau$. Additionally, every node $\tau$, except the
root, keeps matrices $U_\tau$ and $V_{\tau}$. The top two levels of
the example shown in Figure~\ref{fig:HSS} can be written out
explicitly as

\begin{equation}
  A = \begin{bmatrix}
    D_1 & U_1 B_{1,2} V_2^* & \multicolumn{2}{c}{\multirow{2}{*}{$\begin{bmatrix} U_1 & 0 \\ 0 & U_2 \end{bmatrix} U_5 B_{5,6} V_6^* \begin{bmatrix} V_3^* & 0 \\ 0 & V_4^* \end{bmatrix}$}} \\
    U_2 B_{2,1} V_1^* & D_2 &  \\
    \multicolumn{2}{c}{\multirow{2}{*}{$\begin{bmatrix} U_3 & 0 \\ 0 & U_4 \end{bmatrix} U_6 B_{6,5} V_5^* \begin{bmatrix} V_1^* & 0 \\ 0 & V_2^* \end{bmatrix}$}}
    & D_3 & U_3 B_{3,4} V_4^* \\
      & & U_4 B_{4,3} V_3^* & D_4
    \end{bmatrix} \, .
    \label{eq:HSSexamplematrix}
\end{equation}
Only at the leaf nodes, where $U_\tau^\mathrm{big} \equiv U_\tau$, is the $U_\tau^\mathrm{big}$ stored explicitly.
A similar relation holds for the $V_\tau$ basis matrices.
For symmetric matrices, $U_i \equiv V_i$ and $B_{ij} \equiv B_{ji}$.

HSS matrix construction based on randomized
sampling techniques has attracted a lot of attention in recent years.
Compared to standard HSS construction techniques~\cite{xia2010fast,
wang2013efficient} which assume that an explicit matrix is given on
input, randomized techniques allow the design of \emph{matrix-free}
construction algorithms. A \emph{fully matrix-free} construction
algorithm relies solely on the availability of a matrix-vector product
routine~\cite{levitt2022linear}. 

A \emph{partially matrix-free} algorithm relies on a
matrix-vector product routine and additionally requires access to some entries of the matrix~\cite{martinsson2011fast,gorman2019robust}.
For certain applications, for example
Toeplitz systems, where fast (e.g., linear time) matrix-vector products
exist, a randomized algorithm typically has linear or log-linear
complexity instead of quadratic complexity with the standard construction algorithms \cite{xia2012superfast}.

This paper is based on a partially matrix-free algorithm and its adaptive version.
Our implementation in STRUMPACK~\cite{strumpack_web} is designed for nonsymmetric matrices and is parallelized to leverage shared and distributed memory architectures.
Other works have investigated parallel HSS constructions~\cite{ghysels2016efficient, wang2013efficient,fernando2017scalable} and even GPU implementations \cite{chen2022solving}.

\subsection{Sketching Based Adaptive HSS Construction Algorithm} 

We extend the HSS construction algorithm described in \cite{gorman2019robust} which is partially matrix-free and leverages sketching.
The algorithm needs a matrix-vector multiplication routine and access to $O(nr)$
entries of $A$.
Instead of compressing the Hankel block itself at each node, we compress a sketch of the Hankel block from which we can recover the compressed version of the off diagonal block~\cite{martinsson2011fast}. 
Then, as we traverse up the tree we combine local sketches from both of the children Hankel blocks, and subtract off the already compressed low rank blocks to recover a local sketch for the parent Hankel block that is written in the basis of the children blocks. 
Finally, this local sketch can be compressed, exploiting the nested basis property.
This procedure is described in equations (2.5)-(2.9) of~\cite{gorman2019robust} and in detail in~\cref{appendix:algodescription}~\cite{martinsson2011fast}. 

We use an interpolative decomposition to compress the off diagonal Hankel blocks \cite{woolfe2008fast}. Given a matrix $A$ with dimensions $m \times m$ with numerical rank $r \ll m$. We can write an interpolative decomposition of $A$ as $A = U A(J,:) + O(\eps)$. Where $U$ has dimensions $m \times r$ and $J$ is an index set of $r$ rows. This interpolative decomposition can be computed using a rank revealing QR factorization~\cite{gu1996efficient}, detailed in equation (2.4) of \cite{gorman2019robust}.

\begin{remark}
    In practice, the interpolative decomposition is computed using a rank revealing QR factorization as $A \approx A(:,J) V$ which computes a column basis. To compute a row basis, we compute the interpolative decomposition of $A^* \approx A^*(:,J) V$ and apply the conjugate transpose so  $A \approx V^* A(J,:)$, then we can rename $V^* = U$ so  $A \approx U A(J,:)$ resulting in a row basis.
\end{remark}

We can represent a numerically low rank Hankel block as a basis matrix $U$ and a sampling of the rows. 
To compress our low numerical rank off diagonal matrices in HSS we first compute an interpolative decomposition for both row blocks and column blocks.  Then, we combine the bases and query the matrix $A$ for the selected row indices $I$ and column indices $J$ resulting in the representation: $U A(I,J) V$ ~\cite{xia2012superfast}. 
The $r$ rows of the sketch correspond to $r$ rows of the original matrix $A$, allowing us to only use our sketch to compress the Hankel blocks as long as the sketch of the Hankel block is representative of the original Hankel block.

In most practical problems, the numerical rank of the low dimensional off diagonal blocks
is not known {\it a priori}, therefore, the size of the sketching operator needs to be chosen adaptively. Previously, Gorman et al.~\cite{gorman2019robust} developed
a {\it blocked incrementing strategy} which fully reuses 
the already-computed basis set in two ways:
(1) at each HSS tree node $\tau$, if the initial samples are not sufficient, we increase a block of samples $\Delta d$, and augment $\tau$'s orthogonal basis by this amount; (2) This augmented basis will 
cause basis sets of the ancestor nodes to have sizes
at least as large as that of $\tau$, while the 
basis sets of the descendant nodes are not affected.
\Cref{algo::HSScompressNodeAdaptive} illustrates the HSS compression procedure with adaptation built in. The details can also be found in~\cite{ghysels2017-ipdps}.

In the original adaptive compression algorithm from \cite{gorman2019robust} the global sketch of the matrix $A$ was computed using a Gaussian sketching operator. This sketching operator is dense so it requires $O(n^2)$ time to compute an additional column when trying to expand the sketch. 
Now we extend the algorithm to any Johnson--Lindenstrauss sketching operator, and in
particular, SJLT, to speed up the sketching operation.

\subsection{Background on Johnson--Lindenstrauss Sketching} 
We begin this section by stating the classical Johnson--Lindenstrauss (JL) lemma \cite{johnson1984extensions}.
The particular version below is from \cite{dasgupta2003elementary}.

\begin{lemma}[Johnson--Lindenstrauss Lemma \cite{johnson1984extensions}] \label{def:JL}
Given $\varepsilon \in (0,1)$, let $m$ and $d$ be positive integers such that $d \geq 4 (\varepsilon^2/2 - \varepsilon^3/3)^{-1} \log m$. 
For any set $P$ of $m$ points in $\mathbb{R}^{n}$ there exists $f: \mathbb{R}^{n} \rightarrow \mathbb{R}^{d}$ such that for all $u, v \in P$
\begin{equation} \label{eq:JL-lemma}
    (1-\varepsilon) \|u-v\|^2 
    \leq \|f(u)-f(v)\|^2 
    \leq (1+\varepsilon) \|u-v\|^2.
\end{equation}
\end{lemma}

\Cref{def:JL} does not say anything about \emph{how} to construct $f$ and what form it might take. 
In practice, $f$ is usually chosen to be a linear map in the form of a matrix which is drawn randomly from an appropriate distribution.
The following definition captures this idea \cite{krahmer2011new}.

\begin{definition}[JL Sketching Operator] \label{def:distJL}
Suppose $\mathcal{D}$ is a distribution over matrices of size $d \times n$.
We say that a matrix $R \sim \mathcal{D}$ is a $(n, d, \delta, \eps)$-JL sketching operator if for any vector $x \in \R^n$ it satisfies 
\begin{equation*}
       \Pr_{R \sim \mathcal{D}}\left[\left|\|Rx\|^2-\|x\|^2 \right|>\eps \|x\|^2\right]<\delta.
\end{equation*}
\end{definition}

The condition in \cref{def:distJL} considers length preservation of a single vector.
A standard union bound argument can be used to show that a JL matrix with probability $1 - \delta$ satisfies \eqref{eq:JL-lemma} for all $u,v \in P$ where $P$ contains $m$ points, provided that $d$ is chosen to be sufficiently large; see Remark 2.2
of \cite{bamberger2021johnson} for a discussion about this. This non-constructive definition for a JL sketching operator allows us to develop a unified theory for HSS construction using general JL sketching operators. From a practical standpoint we consider Gaussian sketching operators, SJLT and SRHT as specific realizations of JL sketches with tighter bounds. 

In the following subsections, we introduce three popular JL sketching operator distributions.
All three satisfy the condition in \cref{def:distJL} provided that $d$ is large enough. 
Details on theoretical guarantees for each distribution appear in \cref{sec:rangefinderBounds,sec:frotheory}.

\subsubsection{Gaussian Sketching Operator}\label{const:gaussian}
A \emph{Gaussian sketching operator} $R$ of size $n \times d$ has entries which are drawn independently from a normal distribution with mean zero and variance $1/d$ \cite{avron2011RandomizedAlgorithms, dasgupta2003elementary}.  
We indicate that $R$ is drawn from such a distribution by writing $R \sim \Gaussian(n,d)$.
Gaussian sketching operators are JL sketching operators if the dimension $d$ is sufficiently large \cite{dasgupta2003elementary}. 
Key advantages of Gaussian sketching operators are ease of construction and that they lend themselves to simple and clean theoretical analysis \cite[Remark~8.2]{martinsson_tropp_2020}.
The main downside of the Gaussian sketching operator is that it is relatively slow to apply since it has no particular structure and is dense.
The sketching operators in the two subsections below address this issue by using fast structured or sparse operators, respectively.

\subsubsection{Subsampled Randomized Hadamard Transform (SRHT)}\label{const:fjlt}
A \emph{subsampled randomized Hadamard transform} (SRHT) of size $n \times d$ takes the form $R=DHP$ \cite{AilonChazelle06}. The matrix $D \in \R^{n \times n}$ is diagonal with the diagonal entries drawn independently from the Rademacher distribution, i.e., each entry is $+1$ with probability $1/2$ and $-1$ with probability $1/2$.
The matrix $H \in \R^{n \times n}$ is the normalized Hadamard matrix, a deterministic unitary matrix which can be applied to a vector in $O(n\log(n))$ time instead of $O(n^2)$. 
The normalized Hadamard matrix can be defined recursively via $H_{0} = [1]$ and $H_{2n}=[H_n,  H_n ; H_n, -H_n]$.
Finally, $P \in \R^{n \times d}$ is a sparse random sampling matrix whose columns are chosen independently and uniformly at random from the set $\{ \sqrt{n/d} \cdot e_j^T \}_{j=1}^n$ where $e_j \in \R^n$ is the $j$th canonical basis vector.
We indicate that $R$ is drawn in this fashion by writing $R \sim \SRHT(n,d)$.
An early version of the SRHT appeared in \cite{AilonChazelle06} where each entry of $P$ was independently chosen to be either zero or nonzero, with the nonzero entries drawn from an appropriately scaled normal distribution.

\subsubsection{Sparse Johnson--Lindenstrauss Transform (SJLT)}\label{const:sjlt}
The \emph{sparse
Johnson--Lindenstrauss transform} (SJLT) was first introduced in \cite{KaneNelson14} with subsequent further analysis in \cite{nelson2013osnap,cohen2018simple}.
An SJLT matrix $R$ of size $n \times d$ has a fixed number $\alpha \in [d]$ of nonzero entries per row.
The nonzero entries are drawn independently from a scaled Rademacher distribution, taking values in $\{1/\sqrt{\alpha}, -1/\sqrt{\alpha}\}$ uniformly at random.
The paper \cite{KaneNelson14} proposes two different methods for randomly drawing the position of the nonzero entries in $R$.
The first method draws the $\alpha$ nonzero positions for each row of $R$ uniformly at random from $[d]$.
The second method divides the length-$d$ rows of $R$ into $d / \alpha$ chunks, and for each chunk a single entry is selected uniformly at random to be nonzero. 
This method requires $d/\alpha$ to be an integer.
For both methods, sampling is done for each row independently of the nonzero positions in the other rows.
The two approaches to constructing an SJLT are referred to as the \emph{graph construction} and \emph{block construction}, respectively. 
Throughout the paper, we will denote an SJLT drawn using either construction by $R \sim \SJLT(n,d,\alpha)$.
We implement both approaches in our software and allow the user to select which one to use. 
We test our implementation with the block construction since it is easier to construct and performs better experimentally than the graph construction.
\section{Stopping Criteria for Adaptive HSS Algorithm} \label{sec:adaptivestoppingcrit}
For any adaptive algorithm, it is critical to develop robust stopping criteria,
which allow sufficiently large sketches (enough columns of $S = AR$) to ensure accuracy
but not too large to hurt performance. 
The goal is to find $d$ columns of $R$ to approximate the numerical HSS rank $r$,
where $r < d \ll n$.
In an earlier work~\cite{gorman2019robust}, Gorman et al. developed a block incrementing strategy,
which begins with $d_0$ columns and adds $\Delta d$ columns iteratively.
The algorithm terminates when the last $\Delta d$ columns does not contain new range information. 
One of their primary contributions is the development of the Frobenius norm stopping criteria. They showed that when the sketching operator $R$ has i.i.d. standard Gaussian entries with mean zero and variance one, 
$\mathbb{E}[\|\frac{1}{\sqrt{d}}S\|_F^2]= \|A\|_F^2$.
Moreover, a concentration bound was established detailing that when $R$ has
more columns the Frobenius norm of the sketch matrix is closer to the Frobenius original matrix with high probability [Theorem 3.3]~\cite{gorman2019robust}.
The significance of this theoretical result is that we can use the projection error
based on the sketch $S$ to stop the iteration instead of the original matrix $A$.
The Frobenius norm stopping criteria are: 
\begin{equation} \label{eq:stopping_criteria-1}
    \frac{\|\widehat{S}\|_{F}}{\|\Tilde{S}\|_{F}} < \epsr \, , \qquad
   \|\widehat{S}\|_{F} < \epsa.
\end{equation}
Where $\Tilde{S} = A_{\nu_i,\nu_j} \Bar{R}$ is a matrix of the $\Delta d$ new sketch for
the Hankel block and $\widehat{S} = (I - Q_\tau Q_\tau^*) \Tilde{S}$ is the projection of the new sketch onto the orthogonal complement of the current sketch
($Q_\tau$ is constructed from $A_{\nu_i,\nu_j} R$). If $\|\widehat{S}\|_F$ is small either relative to the first $d$ columns of the sketch or absolutely then we do not need more columns. For the block incremental adaptation, we need to employ an additional rank deficiency
test as part of the stopping criteria, see~\cite{gorman2019robust}[Section 3.5] for details.

\begin{remark}
    We have updated the stopping condition $\frac{1}{\sqrt{d}}\|\widehat{S}\|_{F} < \epsa$ in \cite{gorman2019robust} to $\|\widehat{S}\|_{F} < \epsa $ because now we scale the sketching operator $R$ so that it satisfies the JL sketching operator definition, removing the need for $\frac{1}{\sqrt{d}}$ scaling.
\end{remark}

\begin{remark}
    In the implementation we set $\widehat{S} = (I - Q_\tau Q_\tau^*)^2 \Tilde{S}$, which applies two steps of block Gram-Schmidt for projection to ensure orthogonality under roundoff errors \cite{stewart2008block}.
\end{remark}

In the next section, we extend the theory necessary to justify the Frobenius norm stopping criteria. That is, the more columns added to our sketching operator $R$ the closer our sketch $S$ will be to $A$ in terms of Frobenius norm.

\section{Frobenius Norm Bounds} \label{sec:frotheory}

In this section, we present the mathematical theory to support the use of the Frobenius norm bound as one of the stopping criteria discussed in~\Cref{sec:adaptivestoppingcrit}.
The new result in this Section is \cref{thm:froJL}, which is a unified, foundational theorem about the concentration bound for general JL sketching operators. We will then make the connection of this theorem with the existing theory in the literature, sharpening the general bound for Gaussian sketching operators, SJLT and SRHT. 
The unified framework provides theoretical lower bounds on the number of samples, columns of the sketching operator, $d$ needed in each case to achieve the approximation guarantee in a probabilistic sense. 

While these guarantees provide conservative lower bounds on the number of samples, in practice, many fewer samples are needed. In our experiments we observe that the number of samples needed is on the order of the HSS rank. Although, the theoretical bounds are hard to sharpen without additional assumptions, our experimental results highlight the practical efficiency of the method, even when the theoretical lower bounds are pessimistic. The first result provides a Frobenius norm concentration result which holds for any real JL sketching operator.
\begin{theorem} \label{thm:froJL}
Let $A \in \C^{m \times n}$ and $\eps, \delta \in (0,1)$. If $R \in \R^{n \times d}$ is a $(n, d, \delta', \eps)$-JL matrix where $\delta' = \delta/m$ when $A$ is real and $\delta' = \delta/(2m)$ when $A$ is complex, then the following holds with probability at least $1-\delta$:
	\begin{equation}
		(1-\eps) \| A \|_F^2 \leq \| A R \|_F^2 \leq (1+\eps) \| A \|_F^2.
	\end{equation}
\end{theorem}
\begin{proof}
    Consider first the case when $A$ is real.
Since $R$ is a $(n, d, \frac{\delta}{m}, \eps)$-JL matrix it satisfies 
\begin{equation} \label{eq:JL-matrix-prop}
     \operatorname{Pr}\left[\left|\|x^TR\|^2-\|x^T\|^2 \right|>\eps \|x^T\|^2\right]<\frac{\delta}{m}
    \end{equation}
for any $x \in \R^{n}$. 
Let $A_{j:}$ denote the $j$th row of $A$.
By the triangle inequality,
\begin{equation} \label{eq:triangle-inequality}
\left| \normsq{AR}_F - \normsq{A}_F \right|
= \left |\sum_{j=1}^m \left(\normsq{A_{j:} R} -\normsq{A_{j:}} \right) \right |
\leq \sum_{j=1}^m \left | \normsq{A_{j:} R} -\normsq{A_{j:}} \right |.
\end{equation}
Consequently,
\begin{equation} \label{eq:some-inequalities}
\begin{aligned}
    \Pr \left[ \left| \normsq{AR}_F - \normsq{A}_F \right|  > \eps \normsq{A}_F  \right] 
    & \leq \Pr \left[ \sum_{j=1}^m \left | \normsq{A_{j:} R} -\normsq{A_{j:}} \right |  > \eps \sum_{k=1}^m \normsq{A_{j:}}  \right] \\
    & \leq \Pr \left[  \bigcup_{j = 1}^{m} \left( \left | \normsq{A_{j:} R} -\normsq{A_{j:}} \right |  > \eps \normsq{A_{j:}} \right) \right] \\
    & \leq \sum_{j=1}^m  \Pr \left[  \left | \normsq{A_{j:} R} -\normsq{A_{j:}} \right |  > \eps \normsq{A_{j:}} \right]  \\
    & < m \frac{\delta}{m} = \delta,
\end{aligned}
\end{equation}
where the third inequality is a union bound and the final inequality follows from \cref{eq:JL-matrix-prop}.
This proves the result for the real case.

For the complex case, we may write $A = B + \hat{\imath} C$ where $B, C \in \R^{m \times n}$.
Since
\begin{equation} \label{eq:complex-as-real-matrix}
    \|A\|_F^2 
    = \left\| \begin{bmatrix} B \\ C \end{bmatrix} \right\|_F^2, \qquad 
    \| AR \|_F^2 
    = \left\| \begin{bmatrix} B \\ C \end{bmatrix} R \right\|_F^2,
\end{equation}
and $R$ is a $(n,d,\delta/(2m),\varepsilon)$-JL matrix, the complex case follows from the result when $A$ is real (proved above).
\end{proof}

The statement in \cref{thm:froJL} can be strengthened when specific sketching operators are considered.
We state known bounds for the Gaussian sketching operators (\cref{thm:gaussian-estimator}), SJLT (\cref{thm:SJLT-F-norm}), and SRHT (\cref{thm:srtt-estimator}).
The statement in \cref{thm:gaussian-estimator} follows directly from Theorem~5.2 in \cite{avron2011RandomizedAlgorithms}; see \cref{sec:comment-on-avron-results} for details.

\begin{theorem}[Theorem~5.2 in \cite{avron2011RandomizedAlgorithms}]
\label{thm:gaussian-estimator}
    Let $A \in \C^{m \times n}$ and suppose $R \sim \Gaussian(n,d)$.
    If $d \geq 20 \eps^{-2} \log(2/\delta)$, then the following holds with probability at least $1-\delta$:
    \begin{equation}
	(1-\eps) \| A \|_F^2 \leq \| A R \|_F^2 \leq (1+\eps) \| A \|_F^2.
    \end{equation}
\end{theorem}

The result in \cref{thm:SJLT-F-norm} below is a matrix variant of the main result in \cite{KaneNelson14}.
It can be proven with a slight modification to a proof in \cite{cohen2018simple} which provides a simplified analysis of the result in \cite{KaneNelson14}. Our proof for the matrix version is new but since it is long we omit it in the main text.
For completeness, we provide the novel proof in \Cref{appendix:sjlt-fro-norm}.

\begin{theorem}[Matrix version of result in \cite{KaneNelson14}] \label{thm:SJLT-F-norm}
    Let $A \in \C^{m \times n}$ and suppose $R \in \R^{n \times d}$ is an SJLT constructed using either the graph or block construction (see \cref{sec:prelim}), and suppose $\eps \in (0,1)$ and $\delta \in (0,1/2)$. If $d \geq C \eps^{-2} \log(1/\delta)$ and $\alpha = \lceil \eps d \rceil$ where $C$ is an absolute constant, then the following holds with probability at least $1-\delta$:
    \begin{equation}
        (1-\eps) \| A \|_F^2 \leq \| A R \|_F^2 \leq (1+\eps)\| A \|_F^2.
    \end{equation}
 \end{theorem}
\begin{proof}
    See \Cref{appendix:sjlt-fro-norm}.
\end{proof}
Finally, we present a concentration bound for SRHT matrices from \cite{avron2011RandomizedAlgorithms}.
\begin{theorem}[Theorem~8.4 in \cite{avron2011RandomizedAlgorithms}] \label{thm:srtt-estimator}
    Let $A \in \C^{m \times n}$ and suppose $R \sim \SRHT(n,d)$.
    \newline If $d \geq 2 \eps^{-2} \log^2(4 n^2 / \delta) \log(4/\delta)$, then the following holds with probability at least $1-\delta$:
    \begin{equation}
        (1-\eps) \| A \|_F^2 \leq \| A R \|_F^2 \leq (1+\eps) \| A \|_F^2.
    \end{equation}
\end{theorem}
 
These bounds are conservative -- in practice, we find that fewer samples are sufficient for good compression. From a theoretical standpoint, Gaussian sketching operators require fewer samples than SRHT and SJLT. However, SJLT and 
SRHT can be applied faster, leading to a trade-off between speed and accuracy. 
The bounds above present a unifying theory that allows us to extend our HSS construction method, via the Frobenius norm stopping criteria, \cref{eq:stopping_criteria-1}, to all JL sketching operators.

  \Cref{tab:frobenius_table} summarizes the known theoretical results
  in which a lower bound on $d$ -- the number of columns of $R$ -- is provided
  such that the following holds with probability at least $1-\delta$:
	\begin{equation*}
		(1-\eps) \| M \|_F^2 \leq \| M R \|_F^2 \leq (1+\eps) \| M \|_F^2.
	\end{equation*}

\begin{table}[ht]
    \centering
    \begin{tabular}{ |p{3.5cm}||p{8.5cm}| }
        \hline
         Sketching Operator & Frobenius Norm Bound \\
         \hline
         JL Sketch & $(n,d,\delta/(2m),\varepsilon)$-JL matrix (\Cref{thm:froJL}, {\bf new result}) \\
         Gaussian&  $d \geq 20\eps^{-2}\log(2/\delta)$   (\Cref{thm:gaussian-estimator}) \\
         SJLT& $d \geq C \eps^{-2}\log(1/\delta)$ (\Cref{thm:SJLT-F-norm}, {\bf new matrix version}) \\
         SRHT&   $d \geq 2\eps^{-2}\log^2(4n^2/\delta)\log(4/\delta)$ (\Cref{thm:srtt-estimator}) \\     
        \hline
    \end{tabular}
    \caption{Convergence guarantees for Frobenius norm stopping criterion.}
    \label{tab:frobenius_table}
\end{table}
These bounds are known to be conservative, requiring $d$ to be quite large. For example, if we use a Gaussian sketching operator and set our failure probability $\delta = 0.01$ and $\eps = 0.5$ then we have the bound $d \geq 424$ for $(0.5) \| A \|_F^2 \leq \| A R \|_F^2 \leq (1.5) \| A \|_F^2$ to hold with probability at least $0.99$. In practice, it works well to choose $d_0 = 128$ and $\Delta d = 64$ (STRUMPACK library default values). 

Next, we build on our unified framework for JL sketching operators by proving general range-finder bounds. These bounds extend our theoretical foundation by showing that sketches preserve the relevant range information of low-rank blocks, a necessary property for accurate and efficient HSS compression. 
\section{Range-finder Bounds} \label{sec:rangefinderBounds}
In this section, we establish novel bounds for distributional JL sketching operators (\Cref{thm:rangefinderJL}) and SJLT sketching operators (\Cref{range:sjlt}). Additionally, we state existing results for Gaussian sketching operators (\Cref{range:gauss}) and SRHT (\Cref{range:fjlt}). These bounds demonstrate that the sketch $S$ of a matrix $A$ preserves its approximate range, a necessary property for HSS compression. Notably, our results show that JL sketching operators share the same range-preserving property as Gaussian sketching, as established in Theorem 10.8 of~\cite{halko2011}.

Specifically, We prove bounds of the form $\normsq{A - QQ^*A} = \normsq{(I - P_S)A} \leq c_n \sigma_{r+1}$ where $c_n$ is a constant dependent on $n$, $r$ and $d$ such that $0< r \leq d$. Here, $S = AR = Q \Omega$ and $P_S = QQ^*$, we refer to these bounds as \textit{range-finder bounds}. 
While \cite{halko2011} prove range-finder bounds for Gaussian sketching operators and SRHT, we extend these results to sketching operators drawn from a distributional JL family and SJLT. We leverage many of the same tools as \cite{halko2011} to prove our results and restate the existing bounds and present our novel bounds in \cref{thm:rangefinderJL} and \cref{range:sjlt}.

The extension of range-finder theory is necessary for our HSS compression algorithm, \cref{algo::HSScompressNodeAdaptive}, where an interpolative decomposition is computed for the small sketch $S$ of a low rank block which represents the range of the original large low rank block. 

We use the same setup as \cite{halko2011} where we let $A \in \C^{m \times n}$ with SVD $A = U \Sigma V^*$, where $U \in \C^{m \times n}$ and $V \in \C^{n \times n}$ are orthogonal matrices and $\Sigma \in \R^{n \times n}$ is a diagonal matrix of singular values. Let $R \in \R^{ (r+p) \times n}$ with $d = r+p$ where $r$ is our target rank and $p$ is our oversampling parameter, usually set to around 10, and consider the following decomposition:

\begin{equation}\label{def:svd}
A=U
\begin{array}{cccc}
\begin{bmatrix}
\Sigma_1 & \\
& \Sigma_2
\end{bmatrix}
 &
\begin{bmatrix}
V_1^* \\
V_2^*    
\end{bmatrix}
\end{array}.
\end{equation} 
Where $\Sigma_1 \in \C^{r \times r}$ and $\Sigma_2 \in \C^{(n-r) \times (n-r)}$ are diagonal matrices.
Let 
\begin{equation}\label{eq:orthTimesR}
    R_1 := V_1^* R \in \C^{r \times d}\,\,\, , R_2 := V_2^* R \in \C^{(n-r) \times d}.
\end{equation}
The error bound for the range-finder algorithm is dependent on properties of $R_1$ and $R_2$. 

To prove a range-finder bound for distributional JL sketching operators, \Cref{thm:rangefinderJL}, we leverage Theorem 9.1 from \cite{halko2011} and two intermediate lemmas which we state and prove in \Cref{appendix:jl-rangefinder}. The first lemma,
\Cref{lem:sketchnorm}, provides an upper bound for the 2-norm of any JL sketching operator and the second lemma, \Cref{lem:subspace-embed}, provides a lower bound on the smallest singular value  of our JL matrix times a tall-and-skinny full-rank matrix $V$. With these two lemmas and Theorem 9.1 from \cite{halko2011} we now prove our general rangefinder bound.

\begin{theorem}[Distributional JL implies Range-finder Bound] \label{thm:rangefinderJL}
    Suppose $A \in \C^{m \times n}$ is a matrix and let $0 < r < \min(m,n)$ be the target rank. 
    If $R$ is a $(n, d, \frac{\delta}{2 \max(5^{2r}, n)}, \frac{\eps}{12})$-JL sketching operator with $\eps/12, \delta \in (0,1)$ and $d = r+p$ with $p \geq 0$, then the following holds with probability at least $1-\delta$:   
    \begin{equation} \label{eq:range-finder-bound-dist-JL}
    \norm{(I - P_Y)A} \leq  \left(\sqrt{1+ \frac{n(1+\eps)}{(1-\eps)}}\right)\sigma_{r+1}(A),
    \end{equation}
	where $Y = AR = Q \Omega$ with $P_Y = Q Q^\dagger$.
\end{theorem}

\begin{proof}
	From \cref{lem:sketchnorm}, \cref{lem:subspace-embed} and \cref{rem:subspace-embedding-smallest-sv} we have that the following two events happen simultaneously with probability at least $1-\delta$:
	\begin{equation} \label{eq:good-event}
		\|R\| \leq \sqrt{n (1+\eps)} \qquad \text{ and } \qquad \sigma_{\min}^2(RV) \geq (1-\eps) \sigma_{\min}^2(V).
	\end{equation}
	We proceed under the assumption that the events in \eqref{eq:good-event} occur.
	
	Due to \eqref{eq:good-event}, $R_1$ is full rank, and \cref{thm:halkobd} therefore yields
	\begin{equation} 
		\normsq{(I - P_Y)A} \leq  \|\Sigma_2\|^2 + \|\Sigma_2R_2 R_1^\dagger\|^2.
	\end{equation}
	Taking the square root of both sides and using the sub-multiplicativity of the two norm we have
	\begin{equation} \label{eq:projected-norm}
		\norm{(I - P_Y)A} \leq  \sqrt{\|\Sigma_2\|^2 + \|\Sigma_2\|^2 \|R_2\|^2 \|R_1^\dagger\|^2} = \sqrt{\|\Sigma_2\|^2 (1+\|R_2\|^2 \|R_1^\dagger\|^2)}.
	\end{equation}
	To bound $\|R_2\|^2$, note that
	\begin{equation} \label{eq:R_2-bound}
		\|R_2\|^2 = \|V_2^* R\|^2 = \|R\|^2 \leq n (1+\eps),
	\end{equation}
	where the second equality follows from unitary invariance of the two norm, and inequality follows from \eqref{eq:good-event}.
	To bound $\|R_1^\dagger\|^2$, note that
	\begin{equation} \label{eq:R_1-dagger-bound}
		\|R_1^\dagger\|^2 = \frac{1}{\sigma_{\min}^2(R_1)} \leq \frac{1}{(1-\eps)\sigma_{\min}^2(V_1)} = \frac{1}{1-\eps}
	\end{equation}
	where the inequality follows from \eqref{eq:good-event}.
	Combining \eqref{eq:projected-norm}, \eqref{eq:R_2-bound} and \eqref{eq:R_1-dagger-bound} and the fact that $\|\Sigma_2\|=\sigma_{r+1}(A)$ results in the bound \eqref{eq:range-finder-bound-dist-JL}.
\end{proof}

Next we restate a range-finder bound for Gaussian sketching operators from \cite{halko2011}.

\begin{theorem}[Corollary 10.9 from \cite{halko2011}, simplified deviation bounds of Theorem 10.8] \label{range:gauss}
    Suppose that $A \in \C^{m \times n}$ has singular values $\sigma_1 \geq \sigma_2 \geq \sigma_3 \geq \dots$. Choose oversampling parameter $p\geq 4$ and target rank $ r \geq 2$, where $r+p \leq \min(m,n)$. Draw an $R \in \R^{n \times (r+p)}$ with standard Gaussian entries, construct the sketch matrix $Y = A R = Q \Omega$, and let $P_Y = QQ^\dagger$. Then the norm squared approximation error is
    \begin{equation*}
        \norm{(I - P_Y)A} \leq \bigg(1 + 16\sqrt{1 + \frac{r}{p+1}}\bigg) \sigma_{r+1}(A) + \frac{8\sqrt{r+p}}{p+1}\bigg(\sum_{j > r} \sigma_j^2(A) \bigg)^{1/2},
    \end{equation*}
    with probability at least $1 - 3e^{-p}$.
\end{theorem}

The above theorem states that $R$ has standard Gaussian entries. However, we consider a Gaussian sketching operator where the variance of the Gaussian entries is $1/d$ corresponding to scaling all of the standard Gaussian entries by $1/\sqrt{d}$. 
Since the sketch $Y=AR$ is used to construct a projection operator, this scaling cancels out, leaving the projection operator unchanged. Therefore, the result also holds for our scaled Gaussian sketching operators.

Next, we state and prove range-finder bound for SJLT. The proof follows the steps of the proof of \cref{thm:rangefinderJL} but with stronger guarantees since it is restricted to SJLT matrices.

\begin{theorem}\label{range:sjlt} 
    Given matrix $A \in \C^{m \times n}$ and a rank $r < \min(m,n)$. Fix $\eps,\delta \in (0,1)$. If $R \sim \SJLT(n,d,\alpha)$ with $\alpha = \Theta(\log^3(r/\delta)/\eps)$, $d = \Omega(r\log^6(r/\delta)/\eps^2)$, $Y = A R = Q \Omega$ and $P_Y = Q Q^*$ then   
    \begin{equation} \label{eq:rf-bound-sjlt}
       \norm{(I - P_Y)A} 
       \leq  \sigma_{r+1}(A) \sqrt{1 + \frac{1}{(1-\eps)} \max\Big(\frac{e^2 n \alpha}{d}, \log\Big(\frac{2d}{\delta}\Big) - \frac{n\alpha}{d}\Big)}.
    \end{equation}
    with probability $1 - \delta$.
\end{theorem}

To prove this theorem, we leverage Theorem 9.1 from \cite{halko2011} and two lemmas which we state and prove in \Cref{appendix:sjlt-rangefinder}. The first lemma (\Cref{lem:sjlt-norm-bound}) provides an upper bound on the 2-norm of the SJLT sketching operator and the second lemma (\Cref{thm:sjlt-subspace-embedding}) provides a lower bound on the smallest singular value of our SJLT matrix times a tall-and-skinny full-rank matrix $V$. These lemmas, \Cref{lem:sjlt-norm-bound,thm:sjlt-subspace-embedding}, are akin to \cref{lem:sketchnorm,lem:subspace-embed} but with stronger guarantees since they are restricted to SJLT matrices. We can now combine these two results and follow the steps of the proof of \Cref{thm:rangefinderJL} to prove a range-finder bound for SJLT matrices.

\begin{proof}
    We consider the SVD of the matrix $A$ defined in \cref{def:svd} and let $\mu$ be defined as in \cref{lem:sjlt-norm-bound}. 
    From \cref{lem:sjlt-norm-bound}, \cref{thm:sjlt-subspace-embedding} and \cref{rem:subspace-embedding-smallest-sv} we have that the following two events happen simultaneously with probability at least $1-\delta$:
    \begin{equation} \label{eq:sjlt-events}
        \| R \|^2 \leq \max(e^2 \mu, \log(2 d / \delta) - \mu), \qquad \sigma_{\min}^2(R V_1) \geq 1-\varepsilon.
    \end{equation}
    We proceed under the assumption that the events in \eqref{eq:sjlt-events} occur.

    Following steps similar to those in the proof of \cref{thm:rangefinderJL}, we have
    \begin{equation} \label{eq:sjlt-1}
        \norm{(I - P_Y)A} 
        \leq \sqrt{\|\Sigma_2\|^2 (1+\|R_2\|^2 \|R_1^\dagger\|^2)},
    \end{equation}
    where
    \begin{equation} \label{eq:sjlt-2}
        \|R_2\|^2 
        \leq \max(e^2 \mu, \log(2 d / \delta) - \mu) \quad \text{and} \quad \|R_1^\dagger\|^2 \leq \frac{1}{1-\varepsilon}.
    \end{equation}
    Combining \cref{eq:sjlt-1}, \cref{eq:sjlt-2} and the fact that $\|\Sigma_2\|=\sigma_{r+1}(A)$ results in the bound \cref{eq:rf-bound-sjlt}.
\end{proof}    

Finally, we state a range-finder bound for SRHT from \cite{halko2011}.

\begin{theorem}[Theorem 11.2 from \cite{halko2011}] \label{range:fjlt}
Suppose that $A \in \C^{m \times n}$ has singular values $\sigma_1(A) \geq \sigma_2(A) \geq \sigma_3(A) \geq \dots$. Choose oversampling parameter $p\geq 1$ and target rank $ r \geq 1$ such that $r+p \leq \min\{m,n\}$ and 
\begin{equation*}
    4\left[\sqrt{r} + \sqrt{8\log(rn)} \right]^2 \log(r) \leq (r+p) \leq n.
\end{equation*}
Draw an $R \in \R^{n \times (r+p)}$ SRHT, construct the sketch matrix $Y = A R = Q \Omega$, and let $P_Y = Q Q^*$. Then the norm squared approximation error is 
\begin{equation*}
\norm{(I - P_Y)A} \leq \sqrt{1 + 7n/(r+p)} \cdot \sigma_{r+1}(A), 
\end{equation*}
    with failure probability at most $O(r^{-1})$.
\end{theorem}

\begin{remark}
The above result in \cite{halko2011} is stated when the fast transform is a discrete Fourier transform but in this paper we apply the Hadamard transform. The identical result holds for the Hadamard transform by combining the result in \cite[Theorem 1.3]{tropp2011improved} and following the identical steps of the proof with Fourier transform in \cite{halko2011}. 
\end{remark}

In summary, the new foundational theory in this section is \cref{thm:rangefinderJL}, which shows that a projection based on a distributional JL sketching operator achieves good approximation of the range of the original matrix. With similar proof techniques, we prove that the SJLT sketching achieves good range approximation as well (\cref{range:sjlt}). These two new results augment the existing range-finder bounds for the Gaussian sketching operators and SRHT matrices justifying our use of a more general class of sketching operators in our HSS compression algorithm.

In the following sections we discuss our efficient implementation of an SJLT and SRHT sketching routine for HSS construction. We also compare SJLT and SRHT sketching to the existing Gaussian sketching routine. 
We observe that we can achieve faster compression time with similar accuracy when applying SJLT or SRHT sketching over Gaussian sketching.   

\section{Implementation Details of SJLT Sketching} \label{sec:cstheory}

The SJLT matrix is a highly structured random matrix. To leverage this structure we have created an SJLT data structure and custom sketching routines that use the SJLT data structure. Our specialized data structure and sketching routines speed up the HSS compression algorithm by leveraging matrix sparsity and bypassing multiplications.

\subsection{SJLT Data Structure}
An SJLT matrix is a structured sparse matrix whose entries have two possible nonzero values. $R \in \R^{n \times d}$ is an SJLT matrix with $\alpha$ nonzeros in each row with each nonzero drawn from $\{1/\sqrt{\alpha}, -1/\sqrt{\alpha}\}$ with equal probability. We factor out and store the scaling of $1/\sqrt{\alpha}$ and split our matrix into positive and negative components, resulting in $R = 1/\sqrt{\alpha}(B_+ - B_-)$, where the matrices $B_+$ and $B_-$ only have entries in $\{0,1\}$.
Since $B_+$ and $B_-$ are sparse binary matrices we store them in compressed form. We use both compressed row storage (CRS) and compressed column storage (CCS) \cite{barrett1994templates} where we store pointers to the start of each row (CRS) or column (CCS), and the column or row indices of the nonzero entries respectively. Since our matrices are binary the nonzero values are always one so we do not need to store the values at these nonzero positions. We store the binary matrices in both compressed row and column storage to optimize the caching efficiency when computing $A R$ and $A^* R$. Below we provide an example of our data structure and decomposition.

\begin{equation*}
    R =\frac{1}{\sqrt{2}}\begin{bmatrix}
        1 & 0 & -1  \\ 
        0 & -1& -1  \\
        1 & -1 & 0 \\
        1 & 0 & 1 \\
    \end{bmatrix} = \frac{1}{\sqrt{2}}\left(\begin{bmatrix}
        1 & 0 & 0  \\ 
        0 & 0& 0  \\
        1 & 0 & 0 \\
        1 & 0 & 1 \\
    \end{bmatrix} - \begin{bmatrix}
        0 & 0 & 1  \\ 
        0 & 1& 1  \\
        0 & 1 & 0 \\
        0 & 0 & 0 \\
    \end{bmatrix} \right) = \frac{1}{\sqrt{2}}(B_+ - B_-)
\end{equation*}

\begin{equation*}
    R : \begin{cases}
    s = \frac{1}{\sqrt{2}} \\
    B_+  \text{ stored in CRS and CCS without value arrays}\\
    B_- \text{ stored in CRS and CCS without value arrays}
    \end{cases}
\end{equation*}

This specialized SJLT data structure for binary matrices allows us to avoid doing any multiplications in our algorithm because all multiplications would be by the number one. Instead, we only need to index and sum relevant values. Then after our matrix multiplication is complete we can scale all entries in our resulting sketch.
Additionally, storing the SJLT as a sum of two binary compressed matrices requires less space than as a single compressed matrix which additionally includes the value at each nonzero position when the number of nonzero entries per row is strictly greater than one. Finally, the SJLT data structure is well integrated in the HSS compression algorithm allowing for fast and efficient sketching operator adaptivity. 

\subsection{Adaptive SJLT Sketching}
In the HSS compression algorithm we use adaptive SJLT sketching where the user inputs the number of non-zeros per row for each sketching operator. For example if the user selects S(4) then initially an SJLT matrix with 4 nonzeros per row and $d0$ columns is constructed. If the sketch of $A$ is insufficient for the HSS compression to succeed then we must extend the SJLT matrix to produce a more accurate sketch. We append an additional $\Delta d$ columns with $4$ nonzeros per row until our sketch is accurate enough for the HSS compression to succeed.
we efficiently update our SJLT data structure by adjusting the scaling factor and appending binary columns to the existing SJLT matrix. 

\subsection{\texorpdfstring{Efficiently Computing $AR$ and $A^*R$ For Dense $A$}{Computing A times R}}
In the C++ STRUMPACK library a dense matrix $A$ is stored in column major ordering, so to leverage caching we would like to access our large dense matrix $A$ column by column. We implement the sketching of $A$, $AR$ by considering the outer product formulation. 
\begin{equation*}
A  R = 
\begin{bmatrix}
    \vert & \vert & & \vert\\
    A_{:1}   & A_{:2} & \dots &  A_{:n}\\
    \vert & \vert & & \vert
\end{bmatrix}
\begin{bmatrix}
    \horzbar & R_{1:} & \horzbar \\
    \horzbar & R_{2:} & \horzbar \\
             & \vdots    &          \\
    \horzbar & R_{n:} & \horzbar
  \end{bmatrix} = \sum_{i = 1}^n A_{:i} R_{i:}.
\end{equation*}

First, we initialize a zero matrix which will store our solution and factor out the scaling factor $s$ from our matrix $R$. We iterate over each row of $R$ in compressed row storage. For each row $R_{i}$ if entry $ij$ is $1$, corresponding to a nonzero entry in $B_+$, then we add column $A_i$ to column $j$ of our solution matrix. If entry $ij$ is $-1$, corresponding to a nonzero entry in $B_-$, then we subtract column $A_i$ from column $j$ of our solution matrix. This algorithm accesses each column of $A$ exactly once and uses the row $R_{i:}$ to add or subtract it at different positions in our solution matrix. Since our solution matrix is much smaller than the matrix $A$ this trade-off of leveraging caching of $A$ while accessing many entries in our solution matrix is advantageous. Finally, we scale the resulting matrix to complete our sketching routine. 
 
In the HSS compression algorithm, we compute the sketch for both the rows and the columns of our input dense matrix $A$. 
This means that in our STRUMPACK implementation in addition to computing $A R$ we must also compute  $A^* R$. 
Since we only store $A$ and it is stored in column major format we leverage an inner product formulation for this sketching routine. Where

\begin{equation*}
A^*  R = \left(
\begin{bmatrix}
    \vert & \vert & & \vert\\
    A_{:1}  & A_{:2} & \dots &  A_{:n}\\
    \vert & \vert & & \vert
\end{bmatrix}\right )^*
\begin{bmatrix}
    \vert & \vert & & \vert\\
    R_{:1}   & R_{:2} & \dots &  R_{:k}\\
    \vert & \vert & & \vert
\end{bmatrix}.
\end{equation*}

So to compute this sketch we iterate over each column of $A$ which allows us to leverage caching. Then we take an inner product between the complex conjugate of this column of $A$ and each column of $R$ which we do by using compressed column storage, ignoring the scaling factor. This corresponds to entries in our resulting matrix. Each entry in the resulting matrix is a scaled sum of either $+1,-1$ or $0$ times each entry of the column of $A$ so no multiplication is necessary in this computation. Finally, we can scale the entire result matrix afterwards. 

\subsection{Distributed Parallel Implementation For Dense A}

In addition to providing a shared parallel implementation in STRUMPACK we also provide a distributed memory parallel implementation of the SJLT sketching operators for symmetric matrices. Since the SJLT sketching operators are efficient to store we are able to duplicate the entire sketching operator on each MPI process with low memory overhead. Once we have duplicated the sketch on each process we can use the serial SJLT multiplication routines described in the previous section. We store the operator $A$ in 1D block row format allowing us to efficiently parallelize the multiplication. This storage is in contrast to the Gaussian case which leverages a 2D block cyclic format for both the dense operator and the Gaussian random matrix. We observe a much greater speedup over the Gaussian sketching operators in the distributed parallel setting. 
\section{Implementation Details of SRHT Sketching} \label{sec:srht_implement}
Recall, the sketch matrix $R \sim \SRHT(n,d)$, is given by $R = DHP$.
HSS compression of a matrix  $A \in \mathbb{C}^{m \times n}$ using an SRHT sketch raises two main issues:
\begin{enumerate}
\item An efficient sketch of $A$ when the number of columns $n$, is not a power of $2$.
\item Efficient sketches of the diagonal blocks in lines 18 and 20, of Algorithm \ref{algo::HSScompressNodeAdaptive}.
\end{enumerate}
Matrices $D$ and $P$ are stored as vectors and $H$, the normalized Hadamard transform, is not stored. 

\subsection{Efficient Sketch Of $A$}\label{sec:Hadamard-sketch}
Let $A \in \mathbb{C}^{m \times n}$. The cost of the sketch $S = ADHP$ is dominated by the Hadamard transform. When $n$ is not a power of $2$, we break the Hadamard transform into two smaller Hadamard transforms. Let
\[
k = 2^{\lfloor{\log_2 n}\rfloor},\quad \mbox{with} \quad n = k + r,
\]
and $\bze$ a zero matrix of size $m \times (k-r)$. Then,
\begin{eqnarray}    \label{eqn:had1}
    AH &=& \begin{bmatrix} A_{m,k} & A_{m, r} & \bze_{m,k-r}\end{bmatrix}H_{2k},  \nonumber \\ 
    &=&\begin{bmatrix}A_{m,k} & \Tilde{A}\end{bmatrix}\begin{bmatrix}
        H_k & H_k\\H_k & -H_k
    \end{bmatrix}, \nonumber \\ 
    &=& \begin{bmatrix}
        A_{m,k}H_k + \Tilde{A}H_k & A_{m,k}H_k - \Tilde{A}H_k
    \end{bmatrix},
\end{eqnarray}
where $\Tilde{A} = \begin{bmatrix}
    A_{m, r} & \bze_{m,k-r}
\end{bmatrix} \in \mathbb{R}^{m,k}$. Let 
\[
p = 2^{\lceil{\log_2 r}\rceil},\quad \mbox{and} \quad q = \frac{k}{p},
\]
and $\hat{A}_{m,p} = \begin{bmatrix}
    A_{m,r} & \bze_{m,p-r}
\end{bmatrix}$. Then
\begin{eqnarray}
    \Tilde{A}H_k &=& \begin{bmatrix}
        A_{m,r} & \bze_{m,p-r} & \bze_{m,k-p}\end{bmatrix}\begin{bmatrix}
            H_p & H_p & \cdots & H_p \\
            \cdot & \cdot & \cdots & \cdot \\
            \vdots & \vdots & \cdots & \vdots \\
            H_p & \cdot & \cdots & \cdot
        \end{bmatrix}\\
        &=&\left[\hat{A}H_p \quad \hat{A}H_p \quad \cdots \quad \hat{A}H_p\right]_{m \times q} \\
        \label{eqn:hadiden}
        &=&\hat{A}H_p\left[I \quad I \quad \cdots \quad I\right]_{m \times q}.
\end{eqnarray}
Then from equations \eqref{eqn:had1} and \eqref{eqn:hadiden}, we have
\begin{equation}
    AH = \begin{bmatrix}
        A_{m,k}H_k + \hat{A}H_p\begin{bmatrix}
            I & \cdots & I
        \end{bmatrix} & A_{m,k}H_k - \hat{A}H_p\begin{bmatrix}
            I & \cdots & I
        \end{bmatrix}
    \end{bmatrix}.
\end{equation}
Thus, the Hadamard transform of dimension $2k$ is replaced by two transforms of dimensions $k$ and $p$.

\subsection{Sketching Diagonal Blocks}
In lines 18 and 20 of Algorithm \ref{algo::HSScompressNodeAdaptive}, access to portions of the sketch matrix $R$ is required to compute the sketch of the diagonal blocks at level $\tau$. The parts of $R$ required can be computed (i) as needed (just in time), or (ii) all of $R$ can be computed beforehand. 

Here, we derive a formula for computing $R = DHP$, element-wise. The cost of this computation is $O(md)$, for a matrix $A_{m \times n}$ and sketch dimension $d$. Let
\begin{equation}
\label{eqn:srhtDHPmats}
    D = \begin{bmatrix} d_1 & & & \\  & d_2 & & \\ & & \ddots & \\ & & & d_n \end{bmatrix}, \quad H = H_{\nu}, \quad P = c  \begin{bmatrix} | & | & & |\\ P_{:1} & P_{:2} & \cdots & P_{:d}\\  | & | &  & | \end{bmatrix}
\end{equation}
where $d_i = \pm 1$,  $\nu = 2^{\lceil \log_2 n\rceil}$,
 $c = \sqrt{ \frac{\nu}{d} }$  and $P_{:i} = {\bf e}_{\mu}$, i.e. the $\mu$th column of $I_{\mu}$.
Define 
\[
\Tilde{D} = \begin{bmatrix}D_{n, n} & \Tilde{\bze}_{n, \nu - n}\end{bmatrix} \quad \mbox{and} \quad H_{\nu} = \begin{bmatrix} \Hat{H}_{n, \nu}\\ \Tilde{H}_{\nu - n, \nu} \end{bmatrix}.
\]
Then,
\begin{equation}
    \label{eqn:DtildeH}
    DH \equiv \Tilde{D}H = \begin{bmatrix}D_{n, n} & \bze_{n,\nu - n}\end{bmatrix}\begin{bmatrix} \Hat{H}_{n,\nu}\\ \Tilde{H}_{\nu - n, \nu} \end{bmatrix} = D\Hat{H}.
\end{equation}
Hence, 
\begin{eqnarray}
    \label{eqn:DHP}
    \nonumber
    R = DHP &=& c D\Hat{H}\begin{bmatrix} | & | & & |\\ P_{:1} & P_{:2} & \cdots & P_{:d}\\  | & | &  & | \end{bmatrix}\\
    &=& c \begin{bmatrix} | & | & & |\\ {\bf dv}\odot \Hat{H}P_{:1} & {\bf dv}\odot \Hat{H}P_{:2} & \cdots & {\bf dv}\odot \Hat{H}P_{:d}\\  | & | &  & | \end{bmatrix},
\end{eqnarray}
where ${\bf dv} = \begin{bmatrix}d_1\\d_1\\ \vdots \\d_n\end{bmatrix}$ and $\odot$ is the Hadamard product. The $j$th column of $R$,
\[
R_{:j} = c ~{\bf dv} \odot \Hat{H} {\bf e}_{\mu} = c~{\bf dv} \odot \Hat{H}_{:\mu},
\]
and 
\[
\left(H_{\nu}\right)_{i,j} = \frac{(-1)^{i\cdot j}}{\sqrt{\nu}},
\]
is an element-wise definition of the Hadamard transform, where $i \cdot j$ is the dot-product of the base 2 representations of $i$ and $j$.  Then,
\begin{equation}
    \label{eqn:RsrHt}
    R_{i,j} = \frac{d_i}{\sqrt{d}} (-1)^{i\cdot \mu}.
\end{equation}
\section{Experimental Results} \label{sec:experiment}

\subsection{Test Problems}
\label{subsec:setup}
In this section, we compare our HSS construction algorithm in both the serial and parallel settings. 
In the serial setting, we use Gaussian sketching operators, SJLT sketching operators with different numbers of nonzero entries per row, and SRHT sketching operators. 
In the parallel distributed memory setting we only use Gaussian sketching operators and SJLT sketching operators with variable nonzeros. 
We did not implement a parallel distributed version of SRHT due to complexity of handling non-power of two dimension for $H$, and because SRHT was less competitive compared to SJLT.
We observe that the accuracy of the construction is comparable between Gaussian, SJLT with $\alpha >  1$ and SRHT sketching operators for most problems while SJLT and SRHT sketching can often be computed faster. 

We consider the following test cases:
\begin{enumerate}
    \item A covariance matrix (Cov.), using an exponential kernel 
    \begin{equation}
        G_{ij} = \exp{\left( - \frac{\| x_i - x_j \|_2}{\lambda} \right) }
    \end{equation}
    with $x_i, x_j \in [0, 1]^3$ and $\lambda = .2$ the correlation length. We use a structured hexahedral finite element mesh, discretized using the MFEM finite element library. The matrix is reordered using recursive bisection, which also defines the HSS cluster tree. 
    \item A Toeplitz matrix describing a 1D kinetic energy quantum chemistry problem~\cite{jones2016efficient} (QChem Toeplitz), given by 
    \begin{equation}
        T_{ij} = \begin{cases}
          \pi^2 / \left( 6 d^2 \right) & \text{if} \,\, i = j \\
          \left(-1\right)^{i-j} / \left( d^2\left(i-j\right)^2\right) & \text{else}
        \end{cases}
    \end{equation}
    where $d = 0.1$ is a discretization parameter (grid spacing). 
    The matrix $T$ is fairly ill-conditioned and has small HSS ranks which grow slowly with the dimension of $T$.
    \item 
        The impedance matrix $Z$~\cite{liu2016hss} (Scatt. wave):
        \begin{equation}
            Z_{ij} = \frac{k \eta_0}{4} \int_S{t_i(\rho) \int_S{b_j(\rho') H_0^{(2)}(k|\rho - \rho'|) ds' }ds }
    \end{equation}
    where $k = 2 \pi / \lambda_0$ is the wave number, $\lambda_0$ denotes the free-space wavelength, $\eta_0$ is the intrinsic impedance of free space, and $H_0^{(2)}$ is the zeroth-order Hankel function of the second kind. The surface $S$ is a perfectly electrically conducting circle (2D) residing in free space. This circle is discretized using $n$ line segments, and we use delta functions located at the center of each line segment for $t_i$, and constant functions supported on the line segments for $b_j$. The inner integral is evaluated with a simple quadrature rule with $4$ quadrature points. For the experiments, we vary $n$ and adjust $\lambda_0$ accordingly such that the number of points per wavelength is approximately $24$.
    \item The root front from a sparse multifrontal solver~\cite{ghysels2017-ipdps} (3D Poisson front). The multifrontal solver is applied to a linear system resulting from the second order central finite difference discretization of the 3D Poisson equation on a $k^3$ grid, with zero Dirichlet boundary conditions. The sparse solver uses a nested dissection ordering, and the root vertex separator, a $k \times k$ plane in the grid, corresponds to the dense $k^2 \times k^2$ root frontal matrix.
\end{enumerate}

\subsection{Test Machine and Software}

All experiments are run on the Perlmutter system at NERSC, LBNL. Each Perlmutter (CPU) node has 2 AMD EPYC 7763 CPUs with 64 cores each and 512GB of DDR4 memory.
The code is compiled with GCC 12.3.0, and the BLAS/LAPACK routines are from OpenBLAS 0.3.26. In the distributed parallel setting we test with 8, 16 and 32 MPI ranks on 1, 2 and 4 Perlmutter nodes respectively. 

The HSS construction algorithm with different sketching options, and the test cases are implemented in the STRUMPACK library, and are available at \url{https://github.com/pghysels/STRUMPACK/}.

\subsection{Results}

\subsubsection{Sequential Results with SJLT and SRHT}

\begin{table}[h]
    \centering
    \resizebox{\textwidth}{!}{
    \begin{tabular}{|l|r|r|rrrrrr|rrrrrr|r|}
          \hline
                &     &            & \multicolumn{6}{c|}{HSS sketching time (sec)} & \multicolumn{6}{c|}{Total HSS construction time (sec)} & comp \\ 
         Matrix & $\epsr$ & $n$ & G & S(1) & S(2) & S(4) & S(8) & H & G & S(1) & S(2) & S(4) & S(8) & H & (\%)\\
         \hline
         \hline
         \multirow{9}{*}{Cov.}
         & \multirow{3}{*}{$10^{-2}$}
& $10^3$ & 0.00867 &   0 & 0.000667 & 0.000667 & 0.00167 & 0.0153 & 0.045 & 0.034 & 0.035 & 0.037 & 0.038 & 0.049 &  46.1 \\ 
& & $20^3$ & 0.866 & 0.051 & 0.0927 & 0.168 & 0.317 & 2.58 & 1.43 & 0.569 & 0.611 & 0.662 & 0.836 & 3.1 &   7.4 \\ 
& & $30^3$ & 14.4 & 1.15 & 2.11 & 3.61 & 6.8 & 13.7 &  17 & 3.37 & 4.28 & 5.93 & 9.13 & 15.9 &   2.2 \\ 
        \cline{2-16}
            & \multirow{3}{*}{$10^{-4}$} 
& $10^3$ & 0.011 &   0 & 0.000667 & 0.001 & 0.003 & 0.017 & 0.079 & 0.042 & 0.063 & 0.066 & 0.068 & 0.078 &  58.0 \\ 
& & $20^3$ & 2.35 & 0.23 & 0.381 & 0.71 & 1.28 & 2.74 & 5.87 & 3.92 & 3.99 & 4.49 & 4.88 & 6.3 &  19.2 \\ 
& & $30^3$ & 72.1 & 6.81 & 11.9 & 21.5 & 40.1 & 15.5 & 121 & 58.8 & 64.4 &  73 & 91.1 & 67.8 &  11.2 \\ 
        \cline{2-16}
            & \multirow{3}{*}{$10^{-6}$}
& $10^3$ & 0.014 &   0 & 0.000667 & 0.00233 & 0.00533 & 0.0187 & 0.111 & 0.052 & 0.089 & 0.113 & 0.116 & 0.105 &  73.7 \\ 
& & $20^3$ & 3.46 & 0.346 & 0.559 & 1.02 & 1.88 & 2.84 &  11 & 8.82 & 8.6 & 9.16 & 9.82 & 10.5 &  30.5 \\ 
& & $30^3$ & 95.9 & 9.02 & 15.4 & 28.5 & 54.3 & 16.5 & 201 & 116 & 123 & 137 & 162 & 133 &  18.0 \\ 
        \hline
         \hline
         \multirow{9}{*}{\shortstack[l]{QChem\\ Toeplitz}}
                 & \multirow{3}{*}{$10^{-2}$} 
& 10K & 0.783 & 0.0307 & 0.045 & 0.0793 & 0.162 & 1.85 & 0.9 & 0.113 & 0.126 & 0.16 & 0.249 & 1.93 &   1.7 \\ 
& & 20K & 3.16 & 0.142 & 0.229 & 0.452 & 0.73 & 8.51 & 3.39 & 0.307 & 0.392 & 0.613 & 0.883 & 8.67 &   0.9 \\ 
& & 40K & 12.7 & 0.769 & 1.49 & 2.61 & 4.7 & 77.6 & 13.2 & 1.1 & 1.82 & 2.92 & 5.02 & 77.9 &   0.4 \\ 
        \cline{2-16}
            & \multirow{3}{*}{$10^{-4}$} 
& 10K & 0.788 & 0.029 & 0.045 & 0.086 & 0.155 & 1.86 & 0.913 & 0.121 & 0.135 & 0.179 & 0.257 & 1.95 &   1.9 \\ 
& & 20K & 3.17 & 0.145 & 0.23 & 0.44 & 0.729 & 8.51 & 3.42 & 0.331 & 0.412 & 0.627 & 0.903 & 8.68 &   0.9 \\ 
& & 40K & 12.6 & 0.774 & 1.5 & 2.57 & 4.72 & 77.8 & 13.1 & 1.15 & 1.88 & 2.95 & 5.06 & 78.2 &   0.5 \\ 
        \cline{2-16}
            & \multirow{3}{*}{$10^{-6}$} 
& 10K & 0.788 & 0.029 & 0.0463 & 0.083 & 0.162 & 1.86 & 0.931 & 0.138 & 0.155 & 0.195 & 0.275 & 1.96 &   2.0 \\ 
& & 20K & 3.18 & 0.141 & 0.238 & 0.432 & 0.753 & 8.52 & 3.46 & 0.357 & 0.458 & 0.665 & 1.01 & 8.74 &   1.0 \\ 
& & 40K & 12.6 & 0.77 & 1.49 & 2.61 & 4.71 & 77.6 & 13.2 & 1.21 & 1.94 & 3.06 & 5.16 &  78 &   0.5 \\ 
         \hline
         \hline
         \multirow{9}{*}{\shortstack[l]{Scatt.\\ wave}}
         & \multirow{3}{*}{$10^{-2}$} 
& 5K & 1.98 & 0.228 & 0.265 & 0.409 & 0.687 & 1.31 & 2.21 & 0.436 & 0.472 & 0.617 & 0.903 & 1.89 &   4.7 \\ 
& & 10K &  12 & 1.7 & 2.11 & 3.29 & 5.54 & 6.05 & 12.8 & 2.45 & 2.86 & 4.06 & 6.34 & 7.41 &   2.7 \\ 
& & 20K & 81.7 & 15.7 & 18.6 & 27.7 & 47.6 & 44.2 & 85.1 & 18.6 &  22 & 31.1 & 49.2 & 47.5 &   1.6 \\ 
        \cline{2-16}
       & \multirow{3}{*}{$10^{-4}$}  
& 5K & 1.97 & 0.233 & 0.263 & 0.406 & 0.677 & 1.31 & 2.23 & 0.464 & 0.493 & 0.636 & 0.906 & 1.97 &   5.1 \\ 
& & 10K & 12.1 & 1.71 & 2.12 & 3.29 & 5.53 & 6.05 &  13 & 2.52 & 2.94 & 4.11 & 6.34 & 7.55 &   2.9 \\ 
& & 20K & 89.9 &  17 & 20.9 & 31.1 & 51.6 & 44.2 &  94 & 20.9 & 24.8 &  35 & 55.4 & 47.8 &   1.8 \\ 
        \cline{2-16}
       & \multirow{3}{*}{$10^{-6}$}
& 5K & 1.96 & 0.228 & 0.262 & 0.402 & 0.682 & 1.31 & 2.28 & 0.518 & 0.558 & 0.702 & 0.988 & 2.1 &   5.4 \\ 
& & 10K &  12 & 1.71 & 2.12 & 3.3 & 5.53 & 6.04 &  13 & 2.62 & 3.06 & 4.26 & 6.52 & 7.58 &   3.1 \\ 
& & 20K & 81.6 &  17 & 20.9 & 31.3 & 51.3 & 44.3 & 85.7 & 21.1 &  25 & 35.4 & 55.6 & 47.9 &   1.9 \\ 
         \hline
         \hline
         \multirow{9}{*}{\shortstack[l]{3D \\Poisson\\ front}}
         & \multirow{3}{*}{$10^{-2}$} 
& $100^2$ & 1.078 & 0.1547 & 0.092 & 0.17 & 0.323 & 2.516 & 1.467 & 0.839 & 0.438 & 0.519 & 0.674 & 5.381 &   3.9 \\ 
& & $150^2$ & 9.863 & 2.132 & 1.536 & 2.302 & 4.31 & 9.792 & 11.65 & 7.177 & 3.382 & 3.92 & 5.939 & 17.81 &   2.4 \\ 
& & $200^2$ & 35.95 & 10.23 & 9.584 & 11.65 & 20.53 & 79.7 & 39.61 & 24.96 & 13.75 & 15.25 & 24.17 & 93.89 &   1.3 \\ 
                 \cline{2-16}
       & \multirow{3}{*}{$10^{-4}$}  
& $100^2$ & 1.944 & 0.2873 & 0.2643 & 0.4893 & 0.9377 & 2.51 & 3.038 & 2.43 & 1.417 & 1.613 & 2.076 & 6.71 &   6.5 \\ 
& & $150^2$ & 18.72 & 3.259 & 2.885 & 5.237 & 9.977 & 9.794 & 24.48 & 18.3 & 8.963 & 11.14 & 15.92 & 21.68 &   4.2 \\ 
& & $200^2$ & 82.36 & 16.82 & 18.95 & 28.15 & 49.74 & 79.91 & 96.88 & 66.97 & 35.28 & 43.57 & 65.36 & 102.3 &   2.5 \\ 
                 \cline{2-16}
       & \multirow{3}{*}{$10^{-6}$}
 & $100^2$ & 2.806 & 0.3503 & 0.4007 & 0.6767 & 1.401 & 2.516 & 4.949 & 3.822 & 2.613 & 2.703 & 3.595 & 8.178 &   9.3 \\ 
& & $150^2$ & 26.08 & 3.825 & 4.158 & 7.483 & 14.02 & 9.788 & 37.58 & 27.98 & 16.74 & 19.18 & 25.91 & 26.47 &   6.2 \\
& & $200^2$ & 115 & 18.73 & 26.15 & 38.74 & 70.21 & 80.11 & 145.6 & 96.87 & 60.15 & 71.11 & 103.5 & 112.7 &   3.9 \\ 
         \hline
    \end{tabular}
    }
    \caption{Serial times for HSS compression, and sketching time. $G$ refers to sketching with a Gaussian sketching operator, $S(\alpha)$ refers to sketching with an SJLT matrix (block construction) with $\alpha$ nonzeros per row, and $H$ refers to SRHT sketching.}
    \label{tab:experiments}
\end{table}

All the experiments for the Gaussian and SJLT begin with $d_0 = 128$ and $\Delta d = 64$ for the adaptive HSS construction. These are the default values set in STRUMPACK.
In the case of SRHT sketching, we found that our incremental adaptive strategy may not guarantee the desired accuracy. This may be due to the following reason:
Recall in Equation (\ref{eqn:had1})
we extend the dimension $n$ to the next power-of-two in order to use fast Hadamard transform. Yet, the sampled columns using $P$ are not of the original matrix $ADH$, but are the selected sums of certain columns.
In our experiments, we observed that for the covariance and QChem Toeplitz matrices, the default setting $\{d_0=128, \Delta d = 64\}$ delivers good accuracy. However, for the 
scattering wave and the 3D Poisson front problems, we cannot use the adaptive scheme. In each case, we manually tried to increase $d$ to perform the one-shot sampling and empirically found that $d=576$ suffices for the scattering wave problem, and $d=1856$ suffices for the 3D Poisson problem.
It remains an open problem to handle the non-power-of-two case, both theoretically and practically.
The covariance matrix, Toeplitz matrix and Poisson front are symmetric, so for these cases we only sample $AR$ and not $A^*R$.
The HSS leaf size is set to $256$. In the experiments, we vary the relative HSS compression tolerance $\epsr$, and keep the absolute compression tolerance at $\epsa = 10^{-8}$. Random numbers are generated using the C++11 \texttt{std::minstd\_rand} linear congruential engine.

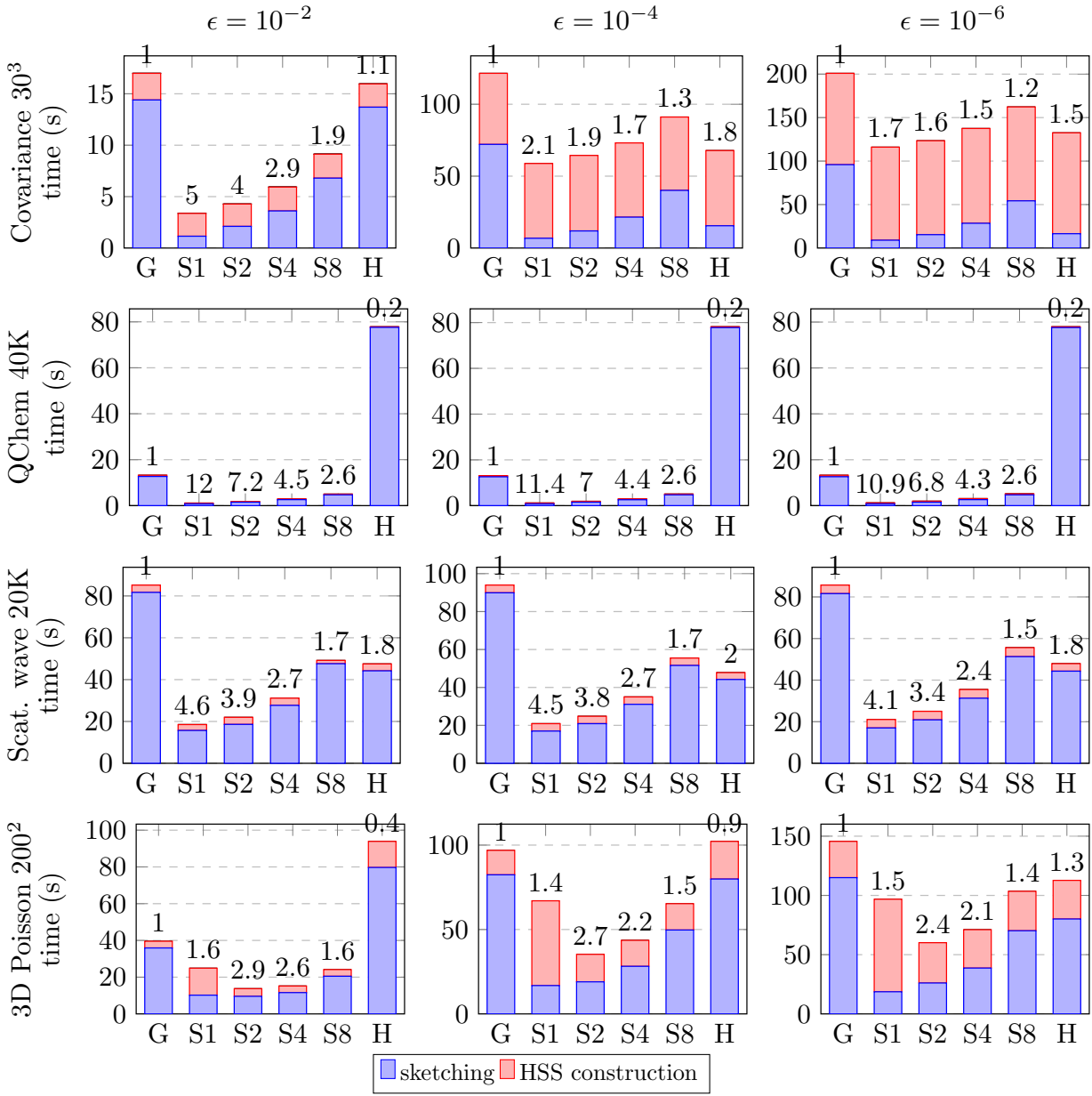
\begin{figure}
  \centering
  \resizebox{\textwidth}{!}{
  \begin{tikzpicture}
  \begin{axis}[
    width=5cm,height=4cm,
    ybar stacked, ymin=0,
    title={$\epsilon=10^{-2}$},
    ylabel={\shortstack[c]{Covariance $30^3$\\ time (s)}}, 
    symbolic x coords={G,S1,S2,S4,S8,H}, xtick=data, 
    ymajorgrids=true, grid style=dashed,
    legend entries={sketching, HSS construction},
    legend to name=times_legend,legend columns=3]
    \addplot+ [ybar] coordinates {
(G, 14.4)
 (S1, 1.15)
 (S2, 2.11)
 (S4, 3.61)
 (S8, 6.8)
 (H, 13.7)
    };
    \addplot+ [ybar] coordinates {
(G, 2.59)
 (S1, 2.22)
 (S2, 2.17)
 (S4, 2.32)
 (S8, 2.33)
 (H, 2.26)
    };
    \addplot+ [ybar] coordinates {
        (G,  0)  (S1, 0) (S2, 0) (S4, 0) (S8, 0) (H, 0)
    };
    \addplot+ [ybar, nodes near coords, point meta=17 / y, nodes near coords align={vertical}, /pgf/number format/.cd,fixed,precision=1] coordinates {
        (G,  0.00001)
        (S1, 0.00001)
        (S2, 0.00001)
        (S4, 0.00001)
        (S8, 0.00001)
        (H,  0.00001)
    };
  \end{axis}
  \end{tikzpicture}
  \begin{tikzpicture}
  \begin{axis}[
    ybar stacked, ymin=0,
    width=5cm,height=4cm,
    title={$\epsilon=10^{-4}$},
    symbolic x coords={G,S1,S2,S4,S8,H}, xtick=data, ymajorgrids=true, grid style=dashed]
    \addplot+ [ybar] coordinates {
(G, 72.1)
 (S1, 6.81)
 (S2, 11.9)
 (S4, 21.5)
 (S8, 40.1)
 (H, 15.5)
    };
    \addplot+ [ybar] coordinates {
(G, 49.4)
 (S1,  52)
 (S2, 52.5)
 (S4, 51.6)
 (S8,  51)
 (H, 52.3)
    };
    \addplot+ [ybar] coordinates {
        (G,  0)  (S1, 0) (S2, 0) (S4, 0) (S8, 0) (H, 0)
    };
    \addplot+ [ybar, nodes near coords, point meta=121 / y, nodes near coords align={vertical}, /pgf/number format/.cd,fixed,precision=1] coordinates {
        (G,  0.00001)
        (S1, 0.00001)
        (S2, 0.00001)
        (S4, 0.00001)
        (S8, 0.00001)
        (H, 0.00001)
    };
  \end{axis}
  \end{tikzpicture}
  \begin{tikzpicture}
  \begin{axis}[
    ybar stacked, ymin=0,
    width=5cm,height=4cm,
    title={$\epsilon=10^{-6}$},
    symbolic x coords={G,S1,S2,S4,S8,H}, xtick=data, ymajorgrids=true, grid style=dashed]
    \addplot+ [ybar] coordinates {
(G, 95.9)
 (S1, 9.02)
 (S2, 15.4)
 (S4, 28.5)
 (S8, 54.3)
 (H, 16.5)
    };
    \addplot+ [ybar] coordinates {
(G, 105)
 (S1, 107)
 (S2, 108)
 (S4, 109)
 (S8, 108)
 (H, 116)
    };
    \addplot+ [ybar] coordinates {
        (G,  0)  (S1, 0) (S2, 0) (S4, 0) (S8, 0) (H, 0)
    };
    %\addplot+ [ybar, nodes near coords, point meta=115.034 / y] coordinates {
    \addplot+ [ybar, nodes near coords, point meta=201 / y, nodes near coords align={vertical}, /pgf/number format/.cd,fixed,precision=1] coordinates {
        (G,  0.00001)
        (S1, 0.00001)
        (S2, 0.00001)
        (S4, 0.00001)
        (S8, 0.00001)
        (H,  0.00001)
    };
  \end{axis}
  \end{tikzpicture}}

  \resizebox{\textwidth}{!}{
  \begin{tikzpicture}
  \begin{axis}[
    width=5cm,height=4cm,
    ybar stacked, ymin=0,
    ylabel={\shortstack[c]{QChem $40$K\\ time (s)}}, 
    symbolic x coords={G,S1,S2,S4,S8,H}, xtick=data, 
    ymajorgrids=true, grid style=dashed]
    \addplot+ [ybar] coordinates {
(G, 12.7)
 (S1, 0.769)
 (S2, 1.49)
 (S4, 2.61)
 (S8, 4.7)
 (H, 77.6)
    };
    \addplot+ [ybar] coordinates {
(G, 0.544)
 (S1, 0.332)
 (S2, 0.332)
 (S4, 0.314)
 (S8, 0.32)
 (H, 0.274)
    };
    \addplot+ [ybar] coordinates {
        (G,  0)  (S1, 0) (S2, 0) (S4, 0) (S8, 0) (H, 0)
    };
    \addplot+ [ybar, nodes near coords, point meta=13.2 / y, nodes near coords align={vertical}, /pgf/number format/.cd,fixed,precision=1] coordinates {
        (G,  0.0000001)
        (S1, 0.0000001)
        (S2, 0.0000001)
        (S4, 0.0000001)
        (S8, 0.0000001)
        (H, 0.0000001)
   };
  \end{axis}
  \end{tikzpicture}
  \begin{tikzpicture}
  \begin{axis}[
    ybar stacked,
    width=5cm,height=4cm, ymin=0,
    symbolic x coords={G,S1,S2,S4,S8,H}, xtick=data, ymajorgrids=true, grid style=dashed]
    \addplot+ [ybar] coordinates {
(G, 12.6)
 (S1, 0.774)
 (S2, 1.5)
 (S4, 2.57)
 (S8, 4.72)
 (H, 77.8)
    };
    \addplot+ [ybar] coordinates {
(G, 0.5)
 (S1, 0.378)
 (S2, 0.379)
 (S4, 0.383)
 (S8, 0.347)
 (H, 0.393)
    };
    \addplot+ [ybar] coordinates {
        (G,  0)  (S1, 0) (S2, 0) (S4, 0) (S8, 0) (H, 0)
    };
    \addplot+ [ybar, nodes near coords, point meta=13.1 / y, nodes near coords align={vertical}, /pgf/number format/.cd,fixed,precision=1] coordinates {
        (G,  0.00001)
        (S1, 0.00001)
        (S2, 0.00001)
        (S4, 0.00001)
        (S8, 0.00001)
        (H,  0.00001)
    };
  \end{axis}
  \end{tikzpicture}
  \begin{tikzpicture}
  \begin{axis}[
    ybar stacked,
    width=5cm,height=4cm, ymin=0,
    symbolic x coords={G,S1,S2,S4,S8,H}, xtick=data, ymajorgrids=true, grid style=dashed]
    \addplot+ [ybar] coordinates {
(G, 12.6)
 (S1, 0.77)
 (S2, 1.49)
 (S4, 2.61)
 (S8, 4.71)
 (H, 77.6)
    };
    \addplot+ [ybar] coordinates {
(G, 0.649)
 (S1, 0.443)
 (S2, 0.448)
 (S4, 0.445)
 (S8, 0.454)
 (H, 0.363)
    };
    \addplot+ [ybar] coordinates {
        (G,  0)  (S1, 0) (S2, 0) (S4, 0) (S8, 0) (H, 0)
    };
    \addplot+ [ybar, nodes near coords, point meta= 13.2/ y, nodes near coords align={vertical}, /pgf/number format/.cd,fixed,precision=1] coordinates {
        (G,  0.00001)
        (S1, 0.00001)
        (S2, 0.00001)
        (S4, 0.00001)
        (S8, 0.00001)
        (H, 0.00001)
    };
  \end{axis}
  \end{tikzpicture}}

  \resizebox{\textwidth}{!}{
  \begin{tikzpicture}
  \begin{axis}[
    width=5cm,height=4cm,
    ybar stacked, ymin=0,
    ylabel={\shortstack[c]{Scat. wave $20$K\\ time (s)}}, 
    symbolic x coords={G,S1,S2,S4,S8,H}, xtick=data, 
    ymajorgrids=true, grid style=dashed]
    \addplot+ [ybar] coordinates {
    (G, 81.7)
 (S1, 15.7)
 (S2, 18.6)
 (S4, 27.7)
 (S8, 47.6)
 (H, 44.2)
    };
    \addplot+ [ybar] coordinates {
    (G, 3.47)
 (S1, 2.82)
 (S2, 3.38)
 (S4, 3.44)
 (S8, 1.57)
 (H, 3.31)
    };
    \addplot+ [ybar] coordinates {
        (G,  0)  (S1, 0) (S2, 0) (S4, 0) (S8, 0) (H, 0)
    };
    %\addplot+ [ybar, nodes near coords, point meta=14.855 / y] coordinates {
    \addplot+ [ybar, nodes near coords, point meta=85.1 / y, nodes near coords align={vertical}, /pgf/number format/.cd,fixed,precision=1] coordinates {
        (G,  0.00001)
        (S1, 0.00001)
        (S2, 0.00001)
        (S4, 0.00001)
        (S8, 0.00001)
        (H,  0.00001)
    };
  \end{axis}
  \end{tikzpicture}
  \begin{tikzpicture}
  \begin{axis}[
    ybar stacked,
    width=5cm,height=4cm, ymin=0,
    symbolic x coords={G,S1,S2,S4,S8,H}, xtick=data, ymajorgrids=true, grid style=dashed]
    \addplot+ [ybar] coordinates {
(G, 89.9)
 (S1,  17)
 (S2, 20.9)
 (S4, 31.1)
 (S8, 51.6)
 (H, 44.2)
    };
    \addplot+ [ybar] coordinates {
(G, 4.07)
 (S1, 3.9)
 (S2, 3.9)
 (S4, 3.9)
 (S8, 3.88)
 (H, 3.61)
    };
    \addplot+ [ybar] coordinates {
        (G,  0)  (S1, 0) (S2, 0) (S4, 0) (S8, 0) (H, 0)
    };
    %\addplot+ [ybar, nodes near coords, point meta=16.612 / y] coordinates {
    \addplot+ [ybar, nodes near coords, point meta=94 / y, nodes near coords align={vertical}, /pgf/number format/.cd,fixed,precision=1] coordinates {
        (G,  0.00001)
        (S1, 0.00001)
        (S2, 0.00001)
        (S4, 0.00001)
        (S8, 0.00001)
        (H,  0.00001)
    };
  \end{axis}
  \end{tikzpicture}
  \begin{tikzpicture}
  \begin{axis}[
    ybar stacked,
    width=5cm,height=4cm, ymin=0,
    symbolic x coords={G,S1,S2,S4,S8,H}, xtick=data, ymajorgrids=true, grid style=dashed]
    \addplot+ [ybar] coordinates {
(G, 81.6)
 (S1,  17)
 (S2, 20.9)
 (S4, 31.3)
 (S8, 51.3)
 (H, 44.3)
    };
    \addplot+ [ybar] coordinates {
(G, 4.09)
 (S1, 4.06)
 (S2, 4.09)
 (S4, 4.18)
 (S8, 4.32)
 (H, 3.62)
    };
    \addplot+ [ybar] coordinates {
        (G,  0)  (S1, 0) (S2, 0) (S4, 0) (S8, 0) (H, 0)
    };
    %\addplot+ [ybar, nodes near coords, point meta=14.762 / y] coordinates {
    \addplot+ [ybar, nodes near coords, point meta=85.7 / y, nodes near coords align={vertical}, /pgf/number format/.cd,fixed,precision=1] coordinates {
        (G,  0.00001)
        (S1, 0.00001)
        (S2, 0.00001)
        (S4, 0.00001)
        (S8, 0.00001)
        (H, 0.00001)
    };
  \end{axis}
  \end{tikzpicture}}

  \resizebox{\textwidth}{!}{
  \begin{tikzpicture}
  \begin{axis}[
    width=5cm,height=4cm,
    ybar stacked, ymin=0,
    ylabel={\shortstack[c]{3D Poisson $200^2$\\ time (s)}}, 
    symbolic x coords={G,S1,S2,S4,S8,H}, xtick=data, 
    ymajorgrids=true, grid style=dashed]
    \addplot+ [ybar] coordinates {
    (G, 35.9)
 (S1, 10.2)
 (S2, 9.58)
 (S4, 11.6)
 (S8, 20.5)
 (H, 79.7)
    };
    \addplot+ [ybar] coordinates {
    (G, 3.66)
 (S1, 14.7)
 (S2, 4.17)
 (S4, 3.6)
 (S8, 3.64)
 (H, 14.2)
    };
    \addplot+ [ybar] coordinates {
        (G,  0)  (S1, 0) (S2, 0) (S4, 0) (S8, 0) (H, 0)
    };
    \addplot+ [ybar, nodes near coords, point meta=39.6 / y, nodes near coords align={vertical}, /pgf/number format/.cd,fixed,precision=1] coordinates {
        (G,  0.00001)
        (S1, 0.00001)
        (S2, 0.00001)
        (S4, 0.00001)
        (S8, 0.00001)
        (H,  0.00001)
    };
  \end{axis}
  \end{tikzpicture}
  \begin{tikzpicture}
  \begin{axis}[
    ybar stacked,
    width=5cm,height=4cm, ymin=0,
    symbolic x coords={G,S1,S2,S4,S8,H}, xtick=data, ymajorgrids=true, grid style=dashed]
    \addplot+ [ybar] coordinates {
(G, 82.4)
 (S1, 16.8)
 (S2,  19)
 (S4, 28.2)
 (S8, 49.7)
 (H, 79.9)
    };
    \addplot+ [ybar] coordinates {
(G, 14.5)
 (S1, 50.2)
 (S2, 16.3)
 (S4, 15.4)
 (S8, 15.6)
 (H, 22.3)
    };
    \addplot+ [ybar] coordinates {
       (G,  0)  (S1, 0) (S2, 0) (S4, 0) (S8, 0) (H, 0)
    };
    \addplot+ [ybar, nodes near coords, point meta=96.9 / y, nodes near coords align={vertical}, /pgf/number format/.cd,fixed,precision=1] coordinates {
        (G,  0.00001)
        (S1, 0.00001)
        (S2, 0.00001)
        (S4, 0.00001)
        (S8, 0.00001)
        (H,  0.00001)
    };
  \end{axis}
  \end{tikzpicture}
  \begin{tikzpicture}
  \begin{axis}[
    ybar stacked,
    width=5cm,height=4cm, ymin=0,
    symbolic x coords={G,S1,S2,S4,S8,H}, xtick=data, ymajorgrids=true, grid style=dashed]
    \addplot+ [ybar] coordinates {
(G, 115)
 (S1, 18.7)
 (S2, 26.1)
 (S4, 38.7)
 (S8, 70.2)
 (H, 80.1)
    };
    \addplot+ [ybar] coordinates {
(G, 30.6)
 (S1, 78.1)
 (S2,  34)
 (S4, 32.4)
 (S8, 33.3)
 (H, 32.6)
    };
    \addplot+ [ybar] coordinates {
(G,  0)  (S1, 0) (S2, 0) (S4, 0) (S8, 0) (H, 0)
    };
    \addplot+ [ybar, nodes near coords, point meta=146 / y, nodes near coords align={vertical}, /pgf/number format/.cd,fixed,precision=1] coordinates {
        (G,  0.00001)
        (S1, 0.00001)
        (S2, 0.00001)
        (S4, 0.00001)
        (S8, 0.00001)
        (H,  0.00001)
    };
  \end{axis}
  \end{tikzpicture}}
  \\
  
  \ref{times_legend}
  \caption{Serial HSS construction time and sketching time. Overall speedup compared to Gaussian sketching is shown at the top of each bar.}
  \label{fig:times}
\end{figure}

\Cref{tab:experiments} shows timing results for the four test problems, with varying dimensions and compression tolerances. In this table, the HSS construction time includes the sketching time. The final column shows the memory usage for the HSS matrix as a percentage of the storage requirements for the corresponding dense matrix. This means that if $n\%$ is listed in the table, $n\%$ of the space required to store a dense matrix $A$ is required to store an HSS compressed version. As expected, when we increase the problem size, memory usage goes down when using HSS format relative to dense format.

The timings for the largest matrices of each test case are also shown in \cref{fig:times} where blue represents the sketching step for each run and red represents the remaining HSS construction time. We observe that the sketching step does in fact dominate the computation. Additionally, we list the ratio of total time to run the compression algorithm in relation to the Gaussian case. 
Frequently, we observe that with SJLT($\alpha = 1$) we achieve an up to $12 \times$ speedup and when we use $\alpha = 2\text{ or }4$ we achieve up to $7 \times$ speedup. We observed that the random matrix construction time is negligible in both the Gaussian and SJLT cases. 

For a matrix $A \in \mathbb{C}^{m \times n}$, the computational cost for a sketch in $d$-dimensions is $O(mnd)$ for the Gaussian sketch and $O(mn \log n)$ for SRHT. As such, SRHT  is more efficient compared to the Gaussian matrix in the regime where $d > \log n$ and less efficient when $d < \log n$. The QChem matrix has small rank and requires small $d$; hence SRHT is inefficient in this regime. For the other test cases, where $d$ is large, SRHT is competitive with the Gaussian and SJLT, and in some cases the most efficient.

\begin{figure}
  \centering
  \begin{tikzpicture}
  \begin{axis}[
    width=4.5cm,height=4cm,
    title={$\epsr = 10^{-2}$},
    ylabel={$d/r$},
    ymin=1, ymax=1.8, %ytick distance=10,
    symbolic x coords={Cov,Imp,Front},
    ymajorgrids=true, grid style=dashed,
    cycle list name=exotic, xtick=data,
    legend entries={$G$, $S(1)$, $S(2)$, $S(4)$, $S(8)$},
    legend to name=d_rank_ratio_legend,legend columns=5]
    \addplot[color=red, mark=*] %, only marks, mark size=2pt, error bars/.cd, y dir=both,y explicit]
    coordinates { % G
        (Cov,   1.55465587044534)
        %(Toep,  10.6666666666667)
        (Imp,   1.10133843212237)
        (Front, 1.70919881305638)
    };
    \addplot[color=green, mark=square*] %, only marks, mark size=2pt, error bars/.cd, y dir=both,y explicit]
    coordinates { % S1
        (Cov,   1.51778656126482)
        %(Toep,  9.84615384615385)
        (Imp,   1.10133843212237)
        (Front, 1.16201117318436)
    };
    \addplot[color=blue, mark=triangle*] %, only marks, mark size=2pt, error bars/.cd, y dir=both,y explicit]
    coordinates { % S2
        (Cov,   1.44796380090498)
        %(Toep, 10.6666666666667)
        (Imp,   1.09923664122137)
        (Front, 1.71940298507463)
    };
    \addplot[color=black, mark=star] %, only marks, mark size=2pt, error bars/.cd, y dir=both,y explicit]
    coordinates { % S4
        (Cov,   1.54838709677419)
        %(Toep, 10.6666666666667)
        (Imp,   1.10133843212237)
        (Front, 1.71428571428571)
    };
    \addplot[color=gray, mark=pentagon] %, only marks, mark size=2pt, error bars/.cd, y dir=both,y explicit]
    coordinates { % S8
        (Cov,   1.63404255319149)
        %(Toep, 9.84615384615385)
        (Imp,   1.09923664122137)
        (Front, 1.71940298507463)
    };
  \end{axis}
  \end{tikzpicture}
  \begin{tikzpicture}
  \begin{axis}[
    width=4.5cm,height=4cm,
    title={$\epsr = 10^{-4}$},
    ymin=1, ymax=1.8, %ytick distance=10,
    symbolic x coords={Cov,Imp,Front},
    ymajorgrids=true, grid style=dashed,
    cycle list name=exotic, xtick=data]
    \addplot[color=red, mark=*] %, only marks, mark size=2pt, error bars/.cd, y dir=both,y explicit]
    coordinates { % G
        (Cov,   1.34782608695652)
        %(Toep,  6.0952380952381)
        (Imp,   1.18959107806691)
        (Front, 1.4125636672326)
    };
    \addplot[color=green, mark=square*] %, only marks, mark size=2pt, error bars/.cd, y dir=both,y explicit]
    coordinates { % S1
        (Cov,   1.42222222222222)
        %(Toep,  4.57142857142857)
        (Imp,   1.18959107806691)
        (Front, 1.4006734006734)
    };
    \addplot[color=blue, mark=triangle*] %, only marks, mark size=2pt, error bars/.cd, y dir=both,y explicit]
    coordinates { % S2
        (Cov,   1.38947368421053)
        %(Toep, 5.56521739130435)
        (Imp,   1.19626168224299)
        (Front, 1.48837209302326)
    };
    \addplot[color=black, mark=star] %, only marks, mark size=2pt, error bars/.cd, y dir=both,y explicit]
    coordinates { % S4
        (Cov,   1.39774983454666)
        %(Toep, 5.56521739130435)
        (Imp,   1.19402985074627)
        (Front, 1.48837209302326)
    };
    \addplot[color=gray, mark=pentagon] %, only marks, mark size=2pt, error bars/.cd, y dir=both,y explicit]
    coordinates { % S8
        (Cov,   1.43673469387755)
        %(Toep, 6.0952380952381)
        (Imp,   1.18959107806691)
        (Front, 1.49832775919732)
    };
  \end{axis}
  \end{tikzpicture}
  \begin{tikzpicture}
  \begin{axis}[
    width=4.5cm,height=4cm,
    title={$\epsr = 10^{-6}$},
    ymin=1, ymax=1.8, %ytick distance=10,
    symbolic x coords={Cov,Imp,Front},
    ymajorgrids=true, grid style=dashed,
    cycle list name=exotic, xtick=data]
    \addplot[color=red, mark=*] %, only marks, mark size=2pt, error bars/.cd, y dir=both,y explicit]
    coordinates { % G
        (Cov,   1.32330827067669)
        %(Toep,  3.55555555555556)
        (Imp,   1.04727272727273)
        (Front, 1.51243781094527)
    };
    \addplot[color=green, mark=square*] %, only marks, mark size=2pt, error bars/.cd, y dir=both,y explicit]
    coordinates { % S1
        (Cov,   1.38996138996139)
        %(Toep,  3.2)
        (Imp,   1.13676731793961)
        (Front, 1.50123456790123)
    };
    \addplot[color=blue, mark=triangle*] %, only marks, mark size=2pt, error bars/.cd, y dir=both,y explicit]
    coordinates { % S2
        (Cov,   1.38528138528139)
        %(Toep,  3.45945945945946)
        (Imp,   1.14695340501792)
        (Front, 1.56862745098039)
    };
    \addplot[color=black, mark=star] %, only marks, mark size=2pt, error bars/.cd, y dir=both,y explicit]
    coordinates { % S4
        (Cov,   1.40693698094773)
        %(Toep,  3.65714285714286)
        (Imp,   1.14490161001789)
        (Front, 1.60401002506266)
    };
    \addplot[color=gray, mark=pentagon] %, only marks, mark size=2pt, error bars/.cd, y dir=both,y explicit]
    coordinates { % S8
        (Cov,   1.38929088277858)
        %(Toep,  3.65714285714286)
        (Imp,   1.13676731793961)
        (Front, 1.60200250312891)
    };
  \end{axis}
  \end{tikzpicture}
  \\

  \ref{d_rank_ratio_legend}
  \caption{Oversampling ratios, the final $d$ over the HSS rank, for the largest test cases, covariance, impedance matrix (scattering wave), and frontal matrix. The quantum chemistry Toeplitz problem is omitted, since for this problem the rank are so small that it does not require any adaptation.}
  \label{fig:d_rank_ratio}
\end{figure}
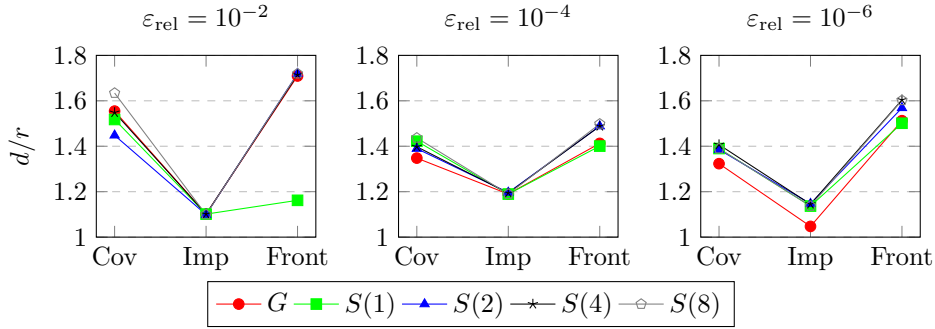

\Cref{fig:d_rank_ratio} shows the oversampling ratio, i.e., the ratio of the final $d$ over the HSS rank $r$, for the largest test problems. The quantum chemistry Toeplitz problem is omitted, since the ranks are so small that no adaptation is required. The oversampling ratio is similar for the different sketching methods. 

\input{PaperSections/figs/serialfigures/HSS_error_rank.tex}

Finally, \crefrange{fig:err_rank_cov}{fig:err_rank_front} show the relative errors and the HSS ranks for these problems. For these results, the experiments are run $5$ times and the figures show error bars with the minimum, median and maximum values. We observe that the HSS ranks and errors are comparable between all of the sketching operators except for SJLT with $\alpha = 1$, in which performance in terms of rank and error are worse than Gaussian sketching operators. 
From \Cref{fig:err_rank_cov} we observe that the errors and the ranks are approximately the same across all methods except S(1) which has worse error and H which has larger ranks for the largest problems. S(1) is often not sufficient to obtain good accuracy and H has some performance degradation for larger HSS ranks. From \cref{fig:times} we observe that SJLT is the most efficient method, yielding a time improvement ranging from 1.2--4$\times$ over the Gaussian sketches. 

For the QChem Toeplitz matrix, we observe that the HSS ranks and errors are the same across all of the methods except S(1) in some cases (\cref{fig:err_rank_qc}). Again, this is likely due to S(1) not being sufficiently dense to capture the matrix information. For timing, since this problem has the smallest ranks, SJLT is able to outperform all of the methods because it is the fastest sketch to apply, while SRHT performs worse due to the large overhead of computing the Hadamard transform.

For the Scattering wave problem, we observe that the HSS ranks and errors are the same except in the strictest tolerances, blue triangles in \cref{fig:err_rank_wave}, where the error is worse for SJLT and SRHT. In this case Gaussian sketches yield the most accurate results but S(8) has comparable errors and can be computed between 1.5--1.7$\times$ faster (see \cref{fig:times}) which highlights that this performance improvement may come at a slight loss of accuracy.

Finally, for the 3D Poisson frontal matrix in \cref{fig:err_rank_front} we observe that the errors and HSS ranks degrade for S(1) and S(2) relative to the other methods. This is likely due to this problem having larger HSS ranks but, as shown in \cref{fig:times}, S(8) can be applied 1.4--1.6$\times$ faster than Gaussian sketches and yields similar accuracy.

We recommend that users of STRUMPACK use the default values of $d_0 = 128, \Delta d = 64$ when running the HSS compression algorithm. Additionally, if using SJLT matrices we recommend setting $\alpha = 4$, the default value. We have found that this is usually the correct balance between performance improvement over Gaussian sketching operators while having similar accuracy.

\subsubsection{Distributed Memory Results with SJLT}\label{subsec:hss_mpi}

Next we experiment with using the distributed memory SJLT sketching operators and distributed memory Gaussian sketching operators. We did not implement a parallel distributed version of SRHT because SRHT was less competitive compared to SJLT.
We conduct all distributed experiments in the symmetric dense matrix $A$ case and only calculate a sketch of $S = AR$.  We run all experiments for three trials with the following fixed settings: relative tolerance $\epsr = 10^{-4}$, absolute tolerance: $\epsa = 10^{-7}$, HSS leaf size: 512, initial sketch size $d_0 = 512$ and adaptive sketch size $\Delta d = 256$. We vary the sketching operator settings using SJLT with 1, 2, 4 and 8 nonzeros in addition to the Gaussian sketching operators. Additionally, we vary the number of MPI ranks: 8, 16 and 32 requiring 1, 2 and 4 cpu nodes on Perlmutter respectively. Since our distributed parallel implementation is only compatible with symmetric matrices we test the HSS construction algorithm on the covariance matrix, Toeplitz matrix and 3d Poisson frontal matrix described in \cref{subsec:setup}. We test on problem sizes that are larger than the sequential case showing that the distributed parallel implementation is more scalable.

\begin{table}[h]
    \centering
    \resizebox{\textwidth}{!}{
    \begin{tabular}{|l|r|r|rrrrr|rrrrr|}
          \hline
                &     &            & \multicolumn{5}{c|}{HSS sketching time (sec)} & \multicolumn{5}{c|}{Total HSS construction time (sec)} \\ 
         Matrix & MPI size & $n$ & G & S(1) & S(2) & S(4) & S(8) & G & S(1) & S(2) & S(4) & S(8) \\
         \hline
         \hline
         \multirow{9}{*}{\shortstack[l]{Cov.}}
    & \multirow{3}{*}{8} & $20^3$ & 0.402 & 0.009 & 0.013 & 0.028 & 0.056 & 1.084 & 0.705 & 0.649 & 0.854 & 0.815 \\
    &   & $30^3$ & 9.062 & 0.202 & 0.328 & 0.598 & 1.001 & 16.549 & 7.99 & 7.903 & 7.981 & 7.507 \\
    &   & $35^3$ & 38.519 & 0.877 & 1.241 & 2.475 & 4.408 & 77.163 & 49.63 & 42.006 & 41.34 & 37.212 \\
        \cline{2-13}
    & \multirow{3}{*}{16} & $20^3$ & 0.219 & 0.003 & 0.006 & 0.013 & 0.028 & 0.742 & 0.544 & 0.503 & 0.646 & 0.608 \\
    &   & $30^3$ & 4.712 & 0.106 & 0.172 & 0.318 & 0.542 & 9.591 & 5.508 & 5.401 & 5.341 & 5.0 \\
    &   & $35^3$ & 19.494 & 0.448 & 0.634 & 1.273 & 2.269 & 43.014 & 29.432 & 25.644 & 25.569 & 22.905 \\
        \cline{2-13}
    & \multirow{3}{*}{32} & $20^3$ & 0.128 & 0.0001 & 0.003 & 0.004 & 0.012 & 0.605 & 0.676 & 0.463 & 0.525 & 0.533 \\
    &   & $30^3$ & 2.585 & 0.057 & 0.09 & 0.168 & 0.295 & 6.097 & 4.053 & 3.952 & 3.952 & 3.524 \\
    &   & $35^3$ & 10.076 & 0.24 & 0.331 & 0.67 & 1.204 & 23.362 & 18.446 & 16.138 & 15.951 & 14.174 \\
        %\cline{2-13}
        \hline
        \hline
         \multirow{9}{*}{\shortstack[l]{QChem\\ Toeplitz}}
    & \multirow{3}{*}{8} & 25K & 3.087 & 0.027 & 0.038 & 0.067 & 0.125 & 3.454 & 0.204 & 0.23 & 0.28 & 0.316 \\
    &   & 50K & 12.433 & 0.132 & 0.212 & 0.375 & 0.686 & 13.339 & 0.635 & 0.722 & 0.744 & 1.191 \\
    &   & 100K & 50.064 & 0.666 & 1.097 & 1.986 & 3.762 & 54.088 & 1.565 & 1.884 & 6.913 & 5.437 \\
        \cline{2-13}
    & \multirow{3}{*}{16} & 25K & 1.61 & 0.012 & 0.019 & 0.034 & 0.064 & 1.864 & 0.155 & 0.144 & 0.189 & 0.191 \\
    &   & 50K & 6.284 & 0.049 & 0.073 & 0.132 & 0.25 & 6.659 & 0.569 & 0.352 & 0.852 & 0.67 \\
    &   & 100K & 25.2 & 0.258 & 0.422 & 0.749 & 1.376 & 26.537 & 2.221 & 1.114 & 3.355 & 2.0 \\
        \cline{2-13}
    & \multirow{3}{*}{32} & 25K & 0.889 & 0.007 & 0.011 & 0.018 & 0.036 & 1.009 & 0.119 & 0.108 & 0.126 & 0.147 \\
    &   & 50K & 3.404 & 0.023 & 0.037 & 0.068 & 0.131 & 3.642 & 0.213 & 0.228 & 0.274 & 0.31 \\
    &   & 100K & 13.678 & 0.093 & 0.147 & 0.263 & 0.496 & 14.113 & 0.76 & 0.599 & 0.733 & 1.17 \\
        %\cline{2-13}
        \hline
        \hline
         \multirow{9}{*}{\shortstack[l]{3D\\Poisson\\front}}
    & \multirow{3}{*}{8} & $100^2$ & 0.516 & 0.004 & 0.006 & 0.011 & 0.021 & 0.853 & 0.381 & 0.311 & 0.301 & 0.31 \\
    &   & $150^2$ & 2.622 & 0.038 & 0.031 & 0.056 & 0.104 & 3.551 & 1.291 & 0.865 & 0.818 & 0.937 \\
    &   & $200^2$ & 10.692 & 0.241 & 0.121 & 0.38 & 0.713 & 12.934 & 3.596 & 1.769 & 2.175 & 2.471 \\
        \cline{2-13}
    & \multirow{3}{*}{16} & $100^2$ & 0.271 & 0.002 & 0.003 & 0.006 & 0.011 & 0.507 & 0.262 & 0.246 & 0.244 & 0.248 \\
    &   & $150^2$ & 1.365 & 0.019 & 0.015 & 0.027 & 0.052 & 1.998 & 0.855 & 0.617 & 0.576 & 0.612 \\
    &   & $200^2$ & 5.426 & 0.115 & 0.048 & 0.174 & 0.334 & 6.798 & 2.783 & 1.498 & 1.426 & 1.55 \\
        \cline{2-13}
    & \multirow{3}{*}{32} & $100^2$ & 0.158 & 0.001 & 0.002 & 0.003 & 0.006 & 0.365 & 0.23 & 0.212 & 0.205 & 0.208 \\
    &   & $150^2$ & 0.707 & 0.011 & 0.009 & 0.015 & 0.029 & 1.129 & 0.683 & 0.431 & 0.46 & 0.478 \\
    &   & $200^2$ & 3.086 & 0.059 & 0.024 & 0.081 & 0.157 & 4.087 & 1.812 & 0.878 & 0.995 & 1.092 \\
        \cline{2-13}
         \hline
    \end{tabular}
    }
    \caption{Parallel runtimes for HSS sketching and construction, excluding redistribution times.
    $G$ refers to sketching with a Gaussian sketching operator, $S(\alpha)$ refers to sketching with an SJLT matrix (block construction) with $\alpha$ nonzeros per row.}
    \label{tab:mpi_timing_table}
\end{table}

\begin{figure}
  \centering
  \resizebox{\textwidth}{!}{
  \begin{tikzpicture}
  \begin{axis}[
    width=5cm,height=4cm,
    ybar stacked, ymin=0,
    title={$P=8$},
    ylabel={\shortstack[c]{Covariance $35^3$\\ time (s)}}, 
    symbolic x coords={G,S1,S2,S4,S8}, xtick=data, 
    ymajorgrids=true, grid style=dashed,
    legend entries={sketching, HSS construction},
    legend to name=times_legend_mpi,legend columns=3]
    \addplot+ [ybar] coordinates {
(G, 38.519)
 (S1, 0.877)
 (S2, 1.241)
 (S4, 2.475)
 (S8, 4.408)
    };
    \addplot+ [ybar] coordinates {
(G, 38.644)
 (S1, 48.753)
 (S2, 40.765)
 (S4, 38.865)
 (S8, 32.804)
    };
    \addplot+ [ybar] coordinates {
        (G,  0)  (S1, 0) (S2, 0) (S4, 0) (S8, 0)
    };
    \addplot+ [ybar, nodes near coords, point meta=77.163 / y, nodes near coords align={vertical}, /pgf/number format/.cd,fixed,precision=1] coordinates {
        (G,  0.00001)
        (S1, 0.00001)
        (S2, 0.00001)
        (S4, 0.00001)
        (S8, 0.00001)
    };
  \end{axis}
  \end{tikzpicture}
  \begin{tikzpicture}
  \begin{axis}[
    width=5cm,height=4cm,
    ybar stacked, ymin=0,
    title={$P=16$},
    symbolic x coords={G,S1,S2,S4,S8}, xtick=data, 
    ymajorgrids=true, grid style=dashed]
    \addplot+ [ybar] coordinates {
(G, 19.494)
 (S1, 0.448)
 (S2, 0.634)
 (S4, 1.273)
 (S8, 2.269)
    };
    \addplot+ [ybar] coordinates {
(G, 23.52)
 (S1, 28.984)
 (S2, 25.01)
 (S4, 24.296)
 (S8, 20.636)
    };
    \addplot+ [ybar] coordinates {
        (G,  0)  (S1, 0) (S2, 0) (S4, 0) (S8, 0)
    };
    \addplot+ [ybar, nodes near coords, point meta=43.014 / y, nodes near coords align={vertical}, /pgf/number format/.cd,fixed,precision=1] coordinates {
        (G,  0.00001)
        (S1, 0.00001)
        (S2, 0.00001)
        (S4, 0.00001)
        (S8, 0.00001)
    };
  \end{axis}
  \end{tikzpicture}
  \begin{tikzpicture}
  \begin{axis}[
    width=5cm,height=4cm,
    ybar stacked, ymin=0,
    title={$P=32$},
    symbolic x coords={G,S1,S2,S4,S8}, xtick=data, 
    ymajorgrids=true, grid style=dashed]
    \addplot+ [ybar] coordinates {
(G, 10.076)
 (S1, 0.24)
 (S2, 0.331)
 (S4, 0.67)
 (S8, 1.204)
    };
    \addplot+ [ybar] coordinates {
(G, 13.286)
 (S1, 18.206)
 (S2, 15.807)
 (S4, 15.281)
 (S8, 12.97)
    };
    \addplot+ [ybar] coordinates {
        (G,  0)  (S1, 0) (S2, 0) (S4, 0) (S8, 0)
    };
    \addplot+ [ybar, nodes near coords, point meta=23.362 / y, nodes near coords align={vertical}, /pgf/number format/.cd,fixed,precision=1] coordinates {
        (G,  0.00001)
        (S1, 0.00001)
        (S2, 0.00001)
        (S4, 0.00001)
        (S8, 0.00001)
    };
  \end{axis}
  \end{tikzpicture}  }

    \resizebox{\textwidth}{!}{
  \begin{tikzpicture}
  \begin{axis}[
    width=5cm,height=4cm,
    ybar stacked, ymin=0,
    title={$P=8$},
    ylabel={\shortstack[c]{QChem $100$K\\ time (s)}}, 
    symbolic x coords={G,S1,S2,S4,S8}, xtick=data, 
    ymajorgrids=true, grid style=dashed]
    \addplot+ [ybar] coordinates {
(G, 50.064)
 (S1, 0.666)
 (S2, 1.097)
 (S4, 1.986)
 (S8, 3.762)
    };
    \addplot+ [ybar] coordinates {
(G, 4.024)
 (S1, 0.899)
 (S2, 0.787)
 (S4, 4.927)
 (S8, 1.675)
    };
    \addplot+ [ybar] coordinates {
        (G,  0)  (S1, 0) (S2, 0) (S4, 0) (S8, 0)
    };
    \addplot+ [ybar, nodes near coords, point meta=54.088 / y, nodes near coords align={vertical}, /pgf/number format/.cd,fixed,precision=1] coordinates {
        (G,  0.00001)
        (S1, 0.00001)
        (S2, 0.00001)
        (S4, 0.00001)
        (S8, 0.00001)
    };
  \end{axis}
  \end{tikzpicture}
  \begin{tikzpicture}
  \begin{axis}[
    width=5cm,height=4cm,
    ybar stacked, ymin=0,
    title={$P=16$},
    symbolic x coords={G,S1,S2,S4,S8}, xtick=data, 
    ymajorgrids=true, grid style=dashed]
    \addplot+ [ybar] coordinates {
(G, 25.2)
 (S1, 0.258)
 (S2, 0.422)
 (S4, 0.749)
 (S8, 1.376)
    };
    \addplot+ [ybar] coordinates {
(G, 1.337)
 (S1, 1.963)
 (S2, 0.692)
 (S4, 2.606)
 (S8, 0.624)
    };
    \addplot+ [ybar] coordinates {
        (G,  0)  (S1, 0) (S2, 0) (S4, 0) (S8, 0)
    };
    \addplot+ [ybar, nodes near coords, point meta=26.537 / y, nodes near coords align={vertical}, /pgf/number format/.cd,fixed,precision=1] coordinates {
        (G,  0.00001)
        (S1, 0.00001)
        (S2, 0.00001)
        (S4, 0.00001)
        (S8, 0.00001)
    };
  \end{axis}
  \end{tikzpicture}
  \begin{tikzpicture}
  \begin{axis}[
    width=5cm,height=4cm,
    ybar stacked, ymin=0,
    title={$P=32$},
    symbolic x coords={G,S1,S2,S4,S8}, xtick=data, 
    ymajorgrids=true, grid style=dashed]
    \addplot+ [ybar] coordinates {
(G, 13.678)
 (S1,  0.093)
 (S2, 0.147)
 (S4, 0.263)
 (S8, 0.496)
    };
    \addplot+ [ybar] coordinates {
(G, 0.435)
 (S1, 0.667)
 (S2, 0.452)
 (S4, 0.47)
 (S8, 0.674)
    };
    \addplot+ [ybar] coordinates {
        (G,  0)  (S1, 0) (S2, 0) (S4, 0) (S8, 0)
    };
    \addplot+ [ybar, nodes near coords, point meta=14.113 / y, nodes near coords align={vertical}, /pgf/number format/.cd,fixed,precision=1] coordinates {
        (G,  0.00001)
        (S1, 0.00001)
        (S2, 0.00001)
        (S4, 0.00001)
        (S8, 0.00001)
    };
  \end{axis}
  \end{tikzpicture}  }

 \resizebox{\textwidth}{!}{
  \begin{tikzpicture}
  \begin{axis}[
    width=5cm,height=4cm,
    ybar stacked, ymin=0,
    title={$P=8$},
    ylabel={\shortstack[c]{3D Poisson $200^2$\\ time (s)}}, 
    symbolic x coords={G,S1,S2,S4,S8}, xtick=data, 
    ymajorgrids=true, grid style=dashed]
    \addplot+ [ybar] coordinates {
(G, 10.692)
 (S1, 0.241)
 (S2, 0.121)
 (S4, 0.38)
 (S8, 0.713 )
    };
    \addplot+ [ybar] coordinates {
(G, 2.242)
 (S1, 3.355)
 (S2, 1.648)
 (S4, 1.795)
 (S8, 1.758)
    };
    \addplot+ [ybar] coordinates {
        (G,  0)  (S1, 0) (S2, 0) (S4, 0) (S8, 0)
    };
    \addplot+ [ybar, nodes near coords, point meta=12.934 / y, nodes near coords align={vertical}, /pgf/number format/.cd,fixed,precision=1] coordinates {
        (G,  0.00001)
        (S1, 0.00001)
        (S2, 0.00001)
        (S4, 0.00001)
        (S8, 0.00001)
    };
  \end{axis}
  \end{tikzpicture}
  \begin{tikzpicture}
  \begin{axis}[
    width=5cm,height=4cm,
    ybar stacked, ymin=0,
    title={$P=16$},
    symbolic x coords={G,S1,S2,S4,S8}, xtick=data, 
    ymajorgrids=true, grid style=dashed]
    \addplot+ [ybar] coordinates {
(G, 5.426)
 (S1, 0.115)
 (S2, 0.048)
 (S4, 0.174)
 (S8, 0.334)
    };
    \addplot+ [ybar] coordinates {
(G, 1.372)
 (S1, 2.668)
 (S2, 1.45)
 (S4, 1.252)
 (S8, 1.216)
    };
    \addplot+ [ybar] coordinates {
        (G,  0)  (S1, 0) (S2, 0) (S4, 0) (S8, 0)
    };
    \addplot+ [ybar, nodes near coords, point meta=6.798 / y, nodes near coords align={vertical}, /pgf/number format/.cd,fixed,precision=1] coordinates {
        (G,  0.00001)
        (S1, 0.00001)
        (S2, 0.00001)
        (S4, 0.00001)
        (S8, 0.00001)
    };
  \end{axis}
  \end{tikzpicture}
  \begin{tikzpicture}
  \begin{axis}[
    width=5cm,height=4cm,
    ybar stacked, ymin=0,
    title={$P=32$},
    symbolic x coords={G,S1,S2,S4,S8}, xtick=data, 
    ymajorgrids=true, grid style=dashed]
    \addplot+ [ybar] coordinates {
(G, 3.086)
 (S1, 0.059)
 (S2, 0.024)
 (S4, 0.081)
 (S8, 0.157)
    };
    \addplot+ [ybar] coordinates {
(G, 1.001)
 (S1, 1.753)
 (S2, 0.854)
 (S4, 0.914)
 (S8, 0.935)
    };
    \addplot+ [ybar] coordinates {
        (G,  0)  (S1, 0) (S2, 0) (S4, 0) (S8, 0)
    };
    \addplot+ [ybar, nodes near coords, point meta=4.087 / y, nodes near coords align={vertical}, /pgf/number format/.cd,fixed,precision=1] coordinates {
        (G,  0.00001)
        (S1, 0.00001)
        (S2, 0.00001)
        (S4, 0.00001)
        (S8, 0.00001)
    };
  \end{axis}
  \end{tikzpicture}  }
  \\
  
  \ref{times_legend_mpi}
  \caption{HSS construction time and sketching time for the distributed memory experiments ($\epsilon=10^{-4}$). Overall speedup compared to Gaussian sketching is shown at the top of each bar.}
  \label{fig:times_MPI}
\end{figure}
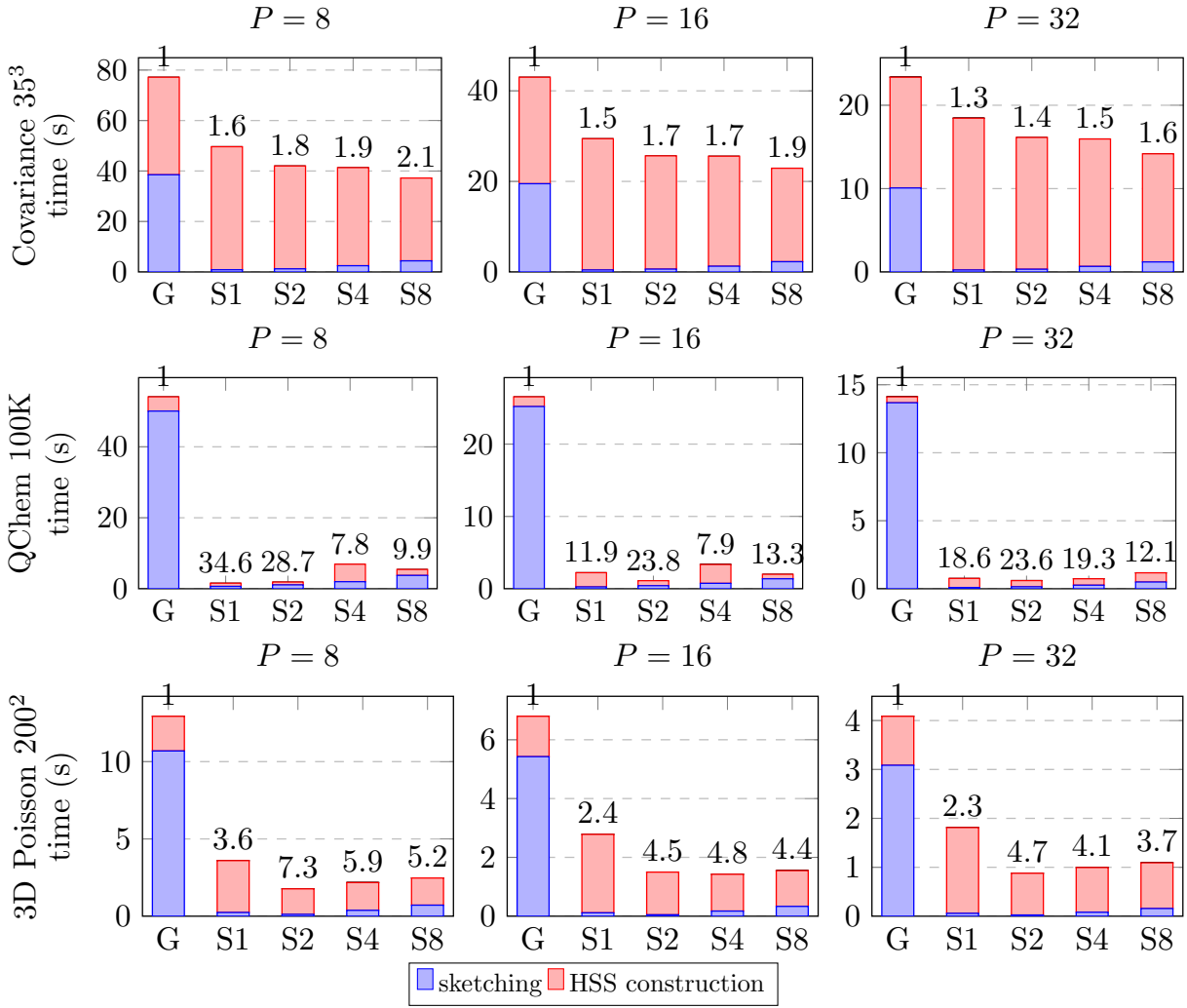

In \cref{tab:mpi_timing_table} we show the parallel sketching time and the total HSS construction time. We observe that the Sketching time for SJLT versus Gaussian sketching operators across all of the test matrices yields between an 8-40x improvement in sketching time. We hypothesize that this improvement is attributed to the reduced communication cost of computing the sketch. In the Gaussian case, since the Gaussian sketching operator is dense we store it in a 2d block cyclic form across the MPI ranks and the same for dense matrix $A$, which requires additional communication time to compute $AR$. Whereas for SJLT, since it is a sparse matrix with low memory cost to store we can duplicate the sketching operator $R$ and use a 1d block row distribution of $A$, and multiplication routine across all MPI ranks. This yields no communication when computing the sketch which yields a 8--40$\times$ improvement in sketching time. Similarly to sketching time, the overall HSS construction time yields a 1.3--35$\times$ improvement depending on the problem which can be observed in \cref{fig:times_MPI}. Additionally, we observe that when we double the MPI ranks from 8 to 16 to 32 the timing is halved and then halved again across all problems, as expected. The total HSS construction time improvement is problem and parameter dependent. 

For the Covariance matrix, which has the largest HSS rank, we see a large speedup in the sketching time of up to 40x speedup. This speedup is not reflected in the overall time which is between 1.2--1.7$\times$ faster. This is likely due to the larger HSS rank, requiring more adaptive steps be taken, increasing the computation on other parts of the algorithm. 
The final $d$ and the HSS ranks for all experiments can be found in the appendix in \cref{tab:experiments_ranks_mpi}.
For the Toeplitz matrix, which has the smallest HSS rank among test problems, there is the largest improvement when using SJLT sketching operators over Gaussian on overall HSS construction of between approximately 8--35$\times$.  
Finally, the 3d Poisson frontal matrix has an up to 100$\times$ speedup when computing the sketch but the overall time is improved by a factor of 2.3--7.3$\times$. By using this parallel distributed implementation the global sketch is no longer the bottleneck for the HSS construction algorithm.

\section{Conclusions}
\label{sec:conclusion}
In this paper we extend the adaptive HSS compression algorithm from \cite{gorman2019robust} which required a Gaussian sketching operator to use any Johnson--Lindenstrauss sketching operator. We provide theoretical guarantees that the adaptive stopping criterion holds for all JL sketching operators including a concentration bound in terms of Frobenius norm. We implement the Sparse Johnson--Lindenstrauss Transform from \cite{KaneNelson14} as a use case for the more general HSS compression algorithm and examine when such a transform outperforms the Gaussian sketching operator. We provide the code in the STRUMPACK C++ library \footnote{\url{https://github.com/pghysels/STRUMPACK/}}.
We demonstrate experimentally that using SJLT or SRHT instead of Gaussian sketching operators leads to up to 2.5$\times$ speedups of the serial HSS construction implementation and up to 35$\times$ speedup over Gaussian in the parallel STRUMPACK C++ implementation using up to 32 processes.

\section*{Acknowledgments}
This research was supported by the Exascale Computing Project (17-SC-20-SC), a collaborative effort of the U.S. Department of Energy Office of Science and the National Nuclear Security Administration. This research used resources of the National Energy Research Scientific Computing Center (NERSC), a U.S. Department of Energy Office of Science User Facility located at Lawrence Berkeley National Laboratory, operated under Contract No. DE-AC02-05CH11231 using NERSC award ASCR-ERCAP0017690. YY was partially supported by the NSF MSGI summer internship program. OAM was supported by the U.S.\ Department of Energy, Office of Science, Office of Advanced Scientific Computing Research under Award Number DE-AC02-05CH11231. All opinions expressed in this paper are the author's and do not necessarily reflect the policies and views of NSF, ORAU/ORISE, or DOE. We acknowledge the Scalable Solvers Group in the Applied Math and Computational Research Division of LBNL, Kenzaburo Nagahama, and Kristen Dawson for insightful conversations.

\bibliographystyle{plain}
\bibliography{ref}

\newpage
\renewcommand{\appendixname}{Supplementary Material}
\appendix

\section{Frobenius Norm Bounds Additional Notes and Proofs}

\subsection{Notes on \cref{thm:gaussian-estimator}} \label{sec:comment-on-avron-results}

The results in \cite{avron2011RandomizedAlgorithms} are concerned with stochastic trace estimation.
When $A$ is real, \cref{thm:gaussian-estimator} follows directly from Theorem~5.2 in \cite{avron2011RandomizedAlgorithms} since
\begin{equation}
    \| A R \|_F^2 = \trace(R^T A^T A R) = \sum_{i=1}^d R^T_{i:} A^T A R_{i:},
\end{equation}
where $R_{i:}$ is the $i$th row of $R$.

When $A$ is complex, we may write it as $A = B + \hat{\imath} C$ where $B, C \in \R^{m \times n}$.
Since the equations in \cref{eq:complex-as-real-matrix} then hold and since there is no $m$-dependence in \cref{thm:gaussian-estimator}, the result for the complex case follows immediately with no modification to the theorem statements.

\subsection{\texorpdfstring{Proof of \Cref{thm:SJLT-F-norm}}{Proof of Theorem 
4.5}}\label{appendix:sjlt-fro-norm}

The proof follows the proof of Theorem~5 in \cite{cohen2018simple} with adaptions made for the matrix case.
We first consider the case when $A$ is real.
For notational simplicity, let $X = A^T$ and note that $\|A R\|_F = \| R^T X \|_F$.
Following the notation in \cite{cohen2018simple}, let $\eta_{ij}$ for $(i,j) \in [d] \times [n]$ be Bernoulli random variables which indicate if the element on position $(i,j)$ of $R^T$ is nonzero.
Moreover, let $\sigma_{ij}$ for $(i,j) \in [d] \times [n]$ be independent Rademacher random variables taking values in $\{-1,1\}$ which indicate the sign of the nonzero entries in $R^T$.
Then, the random matrix $R^T$ defined elementwise via
\begin{equation} \label{eq:R-elementwise}
    R^T_{ij} = \eta_{ij} \sigma_{ij} / \sqrt{\alpha}
\end{equation}
is either a graph or block constructed SJLT depending on how the $\eta_{ij}$ are drawn.
In particular, note that $\eta_{ij}$ and $\eta_{i' j'}$ are independent for all $i,i' \in [d]$ if $j \neq j'$, but the random variables $\eta_{i j}$ and $\eta_{i' j}$ are not independent in general.

It is straightforward to show that
\begin{equation} \label{eq:random-minus-deterministic-F-norm}
    \|R^T X\|_F^2 - \|X\|_F^2 = \frac{1}{\alpha} \sum_{\ell=1}^m \sum_{i=1}^d \sum_{\substack{j,j'=1 \\ j \neq j'}}^n \eta_{ij} \eta_{ij'} \sigma_{ij} \sigma_{ij'} x_{j \ell} x_{j' \ell}.
\end{equation}
Define the matrices $\tilde{X}^{(i)} \in \R^{n \times m}$ for $i \in [d]$ elementwise via
\begin{equation}
    \tilde{x}^{(i)}_{j\ell} = \eta_{ij} x_{j \ell}.
\end{equation}
Let $A_{X,\eta} \in \R^{dn \times dn}$ be block diagonal with the $i$th $n \times n$ block defined by $\frac{1}{\alpha} (\tilde{X}^{(i)} \tilde{X}^{(i)T})^\circ$, where the function $(\cdot)^\circ$ takes a square matrix as input and returns the same matrix but with the diagonal elements set to zero.
Moreover, with a slight overloading of notation, let $\sigma \in \R^{dn}$ denote the vector whose $(j+(i-1)d)$th entry is $\sigma_{ij}$, i.e.,
\begin{equation}
    \sigma = 
    \begin{bmatrix}
        \sigma_{11} & \cdots & \sigma_{1n} & \sigma_{21} & \cdots & \sigma_{2n} & \cdots & \sigma_{k1} & \cdots & \sigma_{kn}
    \end{bmatrix}^T.
\end{equation}
The expression in \eqref{eq:random-minus-deterministic-F-norm} can now be written as the quadratic form
\begin{equation} \label{eq:quadratic-form}
    \|R^T X\|_F^2 - \|X\|_F^2 = \sigma^T A_{X,\eta} \sigma.
\end{equation}

For some random variable $Y$, recall the definition of the $\mathcal{L}^q$-norm for $1 \leq q \leq \infty$:
\begin{equation}
    \| Y \|_q = (\mathbb{E} |Y|^q)^{1/q}.
\end{equation}
We will additionally add superscripts $\eta$ and $\sigma$ to denote $\mathcal{L}^q$-norms and expectations with respect to the variables $(\eta_{ij})$ and $(\sigma_{ij})$ only, for example
\begin{equation}
   \| Y \|_{q, \eta} = (\mathbb{E}_{\eta} | Y |^q )^{1/q},
\end{equation}
where $q := \lceil 2 \log(1/\delta) \rceil > 1$. 
Due to independence between the two sets of variables $(\eta_{ij})$ and $(\sigma_{ij})$, we have $\mathbb{E} Y = \mathbb{E}_\eta \mathbb{E}_\sigma Y$, and consequently
\begin{equation}
    \| \sigma^T A_{X,\eta} \sigma \|_q = \| \| \sigma^T A_{X,\eta} \sigma \|_{q,\sigma} \|_{q, \eta}.
\end{equation}
Applying the Hanson-Wright inequality (Theorem~3 in \cite{cohen2018simple}) to the innermost norm in the expression above followed by the triangle inequality yields
\begin{equation} \label{eq:Hanson-Wright-bound}
    \| \sigma^T A_{X,\eta} \sigma \|_q \leq C_1 (\sqrt{q} \, \| \|A_{X,\eta}\|_F \|_{q,\eta} + q \, \| \|A_{X,\eta}\| \|_{q,\eta}),
\end{equation}
where $C_1$ is an absolute constant.

Now, we bound $\| \| A_{X,\eta} \|_F \|_{q,\eta}$.
To that end, note that
\begin{equation}
\begin{aligned}
    \| \| A_{X,\eta} \|_F \|_{q,\eta} 
    &= \| \| A_{X,\eta} \|_F^2 \|_{q/2,\eta}^{1/2} \\
    &\leq \| \| A_{X,\eta} \|_F^2 \|_{q,\eta}^{1/2} \\
    &= \frac{1}{\alpha} \Big\| \sum_{\substack{j, j'=1 \\ j \neq j'}}^n (X X^T)_{jj'}^2 \sum_{i=1}^d \eta_{ij} \eta_{ij'} \Big\|_{q,\eta}^{1/2},
\end{aligned}
\end{equation}
where the inequality follows from an application of Jensen's inequality, and the last equality uses the fact that $\eta_{ij}^2 = \eta_{ij}$.
Applying the triangle inequality gives
\begin{equation} \label{eq:A_F-bound}
    \| \| A_{X,\eta} \|_F \|_{q,\eta} 
    = \frac{1}{\alpha} \Big( \sum_{\substack{j, j'=1 \\ j \neq j'}}^n (X X^T)_{j j'}^2 \Big\| \sum_{i=1}^d \eta_{ij} \eta_{ij'} \Big\|_{q, \eta} \Big)^{1/2}.
\end{equation}
Since $q \geq 1$ is an integer, and since $\eta_{ij}^2=\eta_{ij}$, we can write
\begin{equation} \label{eq:eta-sum}
    \Big(\sum_{i=1}^d \eta_{ij} \eta_{ij'}\Big)^q = \sum_{S \in \mathcal{S}} \prod_{(i,j) \in S} \eta_{ij}
\end{equation}
for some appropriate set $\mathcal{S}$ of subsets of $[k] \times [n]$ (i.e., each $S \in \mathcal{S}$ satisfies $S \subset [d] \times [n]$).
One property of both the graph and block constructions of SJLT is that 
\begin{equation}
    \mathbb{E} \prod_{(i,j) \in S} \eta_{ij} \leq \prod_{(i,j) \in S} \mathbb{E}\eta_{ij} = (\alpha/d)^{|S|}
\end{equation}
for any $S \subset [d] \times [n]$; see the discussion in Section~2 of \cite{cohen2018simple} for details.
For $(i,j) \in [d] \times [n]$, let $\tilde{\eta}_{ij}$ be \emph{independent} Bernoulli random variables with $\mathbb{E}\tilde{\eta}_{ij} = \mathbb{E} \eta_{ij} = \alpha/d$.
Then, since $\mathbb{E} \prod_{(i,j) \in S} \tilde{\eta_{ij}} = (\alpha/d)^{|S|}$, it follows that 
\begin{equation}
     \mathbb{E} \prod_{(i,j) \in S} \eta_{ij} \leq \mathbb{E} \prod_{(i,j) \in S} \tilde{\eta_{ij}}.
\end{equation}
Combining this with \eqref{eq:eta-sum} gives
\begin{equation} \label{eq:eta-bound}
    \Big\| \sum_{i=1}^d \eta_{ij} \eta_{ij'} \Big\|_{q, \eta} \leq \Big\| \sum_{i=1}^d \tilde{\eta}_{ij} \tilde{\eta}_{ij'} \Big\|_{q, \eta}.
\end{equation}
Note that for $j \neq j'$ it holds that $\mathbb{P}(\tilde{\eta}_{ij} \tilde{\eta}_{ij'} = 1) = (\alpha/d)^2$ due to independence.
Therefore, $\sum_{i=1}^d \tilde{\eta}_{ij} \tilde{\eta}_{ij'}$ follows a $\Binomial(d, (\alpha/d)^2)$ distribution.
It follows from Lemma~2
\footnote{
    In the notation of \cite{cohen2018simple}, the condition $B<e$ in the lemma is satisfied if $C_k > 4/e$.
    Our absolute constant $C_d$ is chosen so that it satisfies this.
} 
in \cite{cohen2018simple} that 
\begin{equation} \label{eq:binomial-bound}
    \Big\| \sum_{i=1}^d \tilde{\eta}_{ij} \tilde{\eta}_{ij'} \Big\|_q 
    \leq C_2 \frac{\alpha^2}{k},
\end{equation}
where $C_2$ is an absolute constant.
Combining \eqref{eq:A_F-bound}, \eqref{eq:eta-bound} and \eqref{eq:binomial-bound} now gives
\begin{equation} \label{eq:bound-1st-term}
    \| \| A_{X,\eta} \|_F \|_{q,\eta} 
    \leq \sqrt{\frac{C_2}{k}} \| X \|_F^2.
\end{equation}

Next, we bound $\|A_{X,\eta}\|$.
Since $A_{X,\eta}$ is block-diagonal, its two norm is equal to the maximum two norm of its sub-blocks: $\|A_{X,\eta}\| = \max_{i \in [d]} \| \frac{1}{\alpha} (\tilde{X}^{(i)} \tilde{X}^{(i) T})^\circ\|$.
We have
\begin{equation}
\begin{aligned}
    \| (\tilde{X}^{(i)} \tilde{X}^{(i) T})^\circ \|
    &= \Big\| \tilde{X}^{(i)} \tilde{X}^{(i) T} - \diagmat\Big( \Big(\sum_{\ell=1}^m \eta_{ij} x_{j\ell}^2 \Big)_j \Big) \Big\| \\
    &\leq \max \Big\{ \|\tilde{X}^{(i)} \tilde{X}^{(i) T}\|, \; \Big\| \diagmat\Big( \Big(\sum_{\ell=1}^m \eta_{ij} x_{j\ell}^2 \Big)_j \Big) \Big\| \Big\} \\
    &\leq \|X\|_F^2,
\end{aligned}
\end{equation}
where the first inequality is due to the fact that both $\tilde{X}^{(i)} \tilde{X}^{(i)T}$ and $\diagmat((\sum_\ell \eta_{ij} x_{jl}^2)_j)$ are positive semi-definite.
It follows that
\begin{equation} \label{eq:bound-2nd-term}
    \|A_{X,\eta}\| \leq \frac{1}{\alpha} \|X\|_F^2.
\end{equation}

Inserting \eqref{eq:bound-1st-term} and \eqref{eq:bound-2nd-term} into \eqref{eq:Hanson-Wright-bound}, and inserting the values of $q$, $d$ and $\alpha$ gives
\begin{equation}
    \| \sigma^T A_{X,\eta} \sigma \|_q \leq \eps C_1 \Big(2 \sqrt{\frac{C_2}{C_d}} + \frac{4}{C_d} \Big) \|X\|_F^2.
\end{equation}

Finally, note that
\begin{equation} \label{eq:final-line}
\begin{aligned}
    \mathbb{P}(| \|R^T X\|_F^2 - \|X\|_F^2 | > \eps \|X\|_F^2) 
    &= \mathbb{P}(| \sigma^T A_{X,\eta} \sigma | > \eps \|X\|_F^2) \\
    &\leq \eps^{-q} \| X \|_F^{-2q} \| \sigma^T A_{X,\eta} \sigma \|_q^q \\
    &\leq \delta,
\end{aligned}
\end{equation}
where the first equality follows from \eqref{eq:quadratic-form}, the first inequality is Markov's inequality, and the second inequality holds with an appropriate choice
\footnote{
    If $C_d$ is chosen so that $C_1(2 \sqrt{C_2/C_d} + 4/C_d) < 1/\sqrt{e}$ is satisfied, then second line in \eqref{eq:final-line} is less than $1/e^{\log(1/\delta)} = \delta$.
    Since $C_1$ and $C_2$ are absolute constants, the absolute constant $C_d$ can be chosen so that it satisfies this requirement.
}
of $C_d$.

This completes the proof for the case when $A$ is real. 
Since there is no $m$-dependence in \cref{thm:SJLT-F-norm}, the case when $A$ is complex follows directly using the argument in \cref{sec:comment-on-avron-results}.

\section{Rangefinder Bounds Additional Notes and Proofs}

\subsection{Lemmas for Proof of \Cref{thm:rangefinderJL}} \label{appendix:jl-rangefinder}
In this section, we recall a theorem from \cite{halko2011} and prove two lemmas which we leverage in the proof of \Cref{thm:rangefinderJL}.  

\begin{theorem}[Theorem 9.1 from \cite{halko2011}, deterministic bound]\label{thm:halkobd}
    Let $A \in \C^{m \times n}$ have SVD $A = U\Sigma V^*$, and fix $r \geq 0$ and oversampling parameter $p \geq 0$. Choose a test matrix $R \in \R^{n \times d}$ and construct $Y = A R = Q \Omega$ with $P_Y = Q Q^*$. Partition $\Sigma$ as in \cref{def:svd}, and define $R_1, R_2$ as in \cref{eq:orthTimesR}. Assuming that $R_1$ has full row rank, the approximation error satisfies
    \begin{equation}
    \normsq{(I - P_Y)A} \leq  \|\Sigma_2\|^2 + \|\Sigma_2R_2 R_1^\dagger\|^2.
    \end{equation}
\end{theorem}

Next, we state and prove two additional lemmas that we apply to prove \Cref{thm:rangefinderJL}.

The first lemma, \Cref{lem:sketchnorm} provides an upper bound for the 2-norm of any JL sketching operator.

\begin{lemma}[2-norm of sketch matrix]\label{lem:sketchnorm}
    Let $R \in \R^{d \times n}$ be a distributional JL sketching operator drawn from a $(n,d,\frac{\delta}{n},\eps)$-JL distribution such that $\eps,\delta \in (0,1)$ and $d < n$.
    Then, with probability $1-\delta$, we have $\norm{R} \leq \sqrt{n(1+\eps)}$.
\end{lemma}

\begin{proof}
    Let $e_1, \ldots, e_n \in \R^n$ denote the canonical basis vectors.
    Note that
    \begin{equation}
        \| R \| 
        = \max_{\substack{y \in \R^n \\ \|y\| = 1}} \|R y\| 
        = \max_{\substack{\beta \in \R^n \\ \|\beta\|=1}} \Big\|R \sum_{i=1}^n \beta_i e_i\Big\| 
        \leq \max_{\substack{\beta \in \R^n \\ \|\beta\|=1}} \sum_{i=1}^n |\beta_i| \, \|R e_i\|.
    \end{equation}
    Since $\Pr[\|R e_i\| \leq \sqrt{1+\eps}] \geq \delta/n$, a union bound therefore gives that the following holds with probability at least $1-\delta$:
    \begin{equation}
        \|R\| 
        \leq \max_{\substack{\beta \in \R^n \\ \|\beta\|=1}} \sum_{i=1}^n |\beta_i| \, \|R e_i\|
        \leq \max_{\substack{\beta \in \R^n \\ \|\beta\|=1}} \sum_{i=1}^n |\beta_i| \sqrt{1+\eps}
        \leq \sqrt{n(1+\eps)},
    \end{equation}
    where the last equality follows from the Cauchy--Schwarz inequality.
\end{proof}

The second lemma provides a lower bound on the smallest singular value of a JL sketching operator times a tall-and-skinny matrix $V$. This bound is required when applying \cref{thm:halkobd}.

\begin{lemma}[JL implies subspace embedding, Theorem 2.3 from \cite{woodruff2014sketching}] \label{lem:subspace-embed}
Let $R \in \R^{d \times n}$ be a distributional JL sketching operator drawn from a $(n,d,\frac{\delta}{5^{2r}},\frac{\eps}{12})$--JL distribution with $\frac{\eps}{12}, \delta \in (0,1)$. Let $V \in \C^{n \times r}$ where $r < d < n$ be a full rank matrix. Then with probability at least $1-\delta$ the following holds:
\begin{equation} \label{eq:subspace-embedding}
    \abs{\normsq{RVx} - \normsq{Vx}} < \eps\normsq{Vx} \qquad \text{for all } x \in \R^r. 
\end{equation}
\end{lemma}

We first state the following intermediate lemma to prove \Cref{lem:subspace-embed}, following the steps of \cite{woodruff2014sketching}.

\begin{lemma}[See page~12 of \cite{woodruff2014sketching}] \label{lemma:angle-preservation}
    Let $x, y \in \R^n$.
    If $|\| Rz \|^2 - \|z\|^2| \leq \eps$ for all $z \in \{x, y, x+y\}$, then 
    \begin{equation}
        |\langle Rx, Ry \rangle - \langle x, y \rangle| \leq 3 \eps \langle x, y \rangle.
    \end{equation}
\end{lemma}

\begin{proof}
    The proof follows the argument on page~12 of \cite{woodruff2014sketching}.
    Without loss of generality we assume $\|x\| = \|y\| = 1$.
    Note that 
    \begin{equation}
    \begin{aligned}
        \langle Rx, Ry \rangle 
        &= \frac{1}{2} \big(\|R(x+y)\|^2 - \|Rx\|^2 - \|Ry\|^2\big) \\ 
        &= \frac{1}{2}\big((1+\alpha_1) \|x+y\|^2 - (1+\alpha_2)\|x\|^2 - (1+\alpha_3)\|y\|^2\big)\\
        &= \frac{1}{2} (2\alpha_1 - \alpha_2 - \alpha_3) + \alpha_1 \langle x, y \rangle.
    \end{aligned}
    \end{equation}
    Since each $|\alpha_i| \leq \eps$, it follows that
    \begin{equation}
        |\langle Rx, Ry \rangle - \langle x, y \rangle| \leq \frac{1}{2} 4 \eps + \eps = 3 \eps.
    \end{equation}
\end{proof}

\begin{proof}[Proof of Lemma~\ref{lem:subspace-embed}]
    The proof follows the discussion on pages 12--14 in \cite{woodruff2014sketching}.
    It is sufficient to show that the claim holds for $y=Vx$ when $y$ is unit length.
    Let $\mathcal{S} = \{y \in \range(V) \, : \, \|y\| = 1 \}$.
    Furthermore, let $\mathcal{N}$ be a 1/2-net for $\mathcal{S}$.
    It is possible to choose $\mathcal{N}$ such that $N := |\mathcal{N}| \leq 5^r$ (see Corollary~4.2.13 in \cite{vershynin2018high}).
    There are $N^2-N$ sums $x+y$ with distinct $x,y \in \mathcal{N}$. 
    Consequently, the following holds with probability at least $1-\delta$:
    \begin{equation}
    	\big| \|Rx\|^2 - \|x\|^2 \big| \leq \frac{\eps}{12} \qquad \text{for all } x \in \mathcal{N} \cup \{y + y' \, : \, y, y' \in \mathcal{N}\}.
    \end{equation}
    Due to Lemma~\ref{lemma:angle-preservation}, the following therefore holds with probability at least $1-\delta$:
    \begin{equation} \label{eq:angle-preservation}
    	\big| \langle Rx, Ry \rangle - \langle x, y \rangle\big| \leq \frac{\eps}{4} \qquad \text{for all } x, y \in \mathcal{N}.
    \end{equation}

    Any $y \in \mathcal{S}$ may be represented as
    \begin{equation}
        y = \sum_{i=0}^\infty \beta_i y^{(i)},
    \end{equation}
    where $|\beta_i| \leq 1/2^{i}$ and each $y^{(i)} \in \mathcal{N}$.
    Consequently, 
    \begin{equation}
    \begin{aligned}
        \|Ry\|^2 
        &= \sum_{i=0}^\infty \sum_{j=0}^\infty \beta_i \beta_j \langle Ry^{(i)}, Ry^{(j)} \rangle 
        = \sum_{i=0}^\infty \sum_{j=0}^\infty \beta_i \beta_j (\langle y^{(i)}, y^{(j)} \rangle + \alpha_{i,j}) \\
        &= \|y\|^2 + \sum_{i=0}^\infty \sum_{j=0}^\infty \beta_i \beta_j \alpha_{i,j},
    \end{aligned}
    \end{equation}
	where each $|\alpha_{i,j}| \leq \eps/4$ due to \eqref{eq:angle-preservation}.
	Consequently, we have
	\begin{equation}
		\big| \|Ry\|^2 - \|y\|^2 \big| \leq \sum_{i=0}^{\infty} \sum_{j=0}^\infty \frac{1}{2^{i+j}} \frac{\eps}{4} = \eps.
	\end{equation}
	
\end{proof}

\begin{remark} \label{rem:subspace-embedding-smallest-sv}
    The smallest singular value of any matrix $B$ satisfies (see, e.g., Theorem~8.6.1 in \cite{golub2013MatrixComputations})
    \begin{equation}
        \sigma_{\min}^2(B) = \min_{\|x\|=1} \|Bx\|^2.
    \end{equation}
    The statement in \eqref{eq:subspace-embedding} therefore implies
    \begin{equation} \label{eq:sigma-min-bound}
        \sigma^2_{\min}(RV) \geq (1-\eps) \sigma^2_{\min}(V),
    \end{equation}
    and consequently that $RV$ is of full rank since $\sigma^2_{\min}(V) > 0$ and $(1-\eps) > 0$.
\end{remark}

\begin{remark}    
    The exponential dependence on $r$ in the $(n,d,\frac{\delta}{5^{2r}}, \frac{\varepsilon}{12})$ in \cref{lem:subspace-embed} may seem alarming.
    However, for many JL sketching operator distributions the embedding dimension has a logarithmic dependence on $1/\delta$, which translates to a linear dependence on $r$.
    This is true for the Gaussian sketching operators, as well as for the SRHT and SJLT we consider in this paper.
\end{remark}
\subsection{Lemmas for Proof of \Cref{range:sjlt}}\label{appendix:sjlt-rangefinder}

We state \Cref{lem:sjlt-norm-bound,thm:sjlt-subspace-embedding} which are akin to \cref{lem:sketchnorm,lem:subspace-embed} but with stronger guarantees since they are restricted to SJLT matrices.

\begin{lemma} \label{lem:sjlt-norm-bound}
    Suppose $R \sim \SJLT(n,d,\alpha)$ with $n > d > \alpha$, and define $\mu = n \alpha / d$. 
    For any $t > 1$, it then holds that 
    \begin{equation} \label{eq:SJLT-spectral-norm-bound}
        \Pr[\| R \|_2^2 \geq t \mu] \leq d e^{-\mu} \Big(\frac{e}{t}\Big)^{t \mu}. 
    \end{equation}
    In particular, if $t > \max(e^2, \mu^{-1} \log(d/\delta)-1)$, then
    \begin{equation} \label{eq:SJLT-spectral-norm-bound-simplified}
        \Pr[\|R\|^2 \geq t \mu] < \delta.
    \end{equation}
\end{lemma}

\begin{proof}
    Recall that we may write $R$ elementwise as in \cref{eq:R-elementwise} where $\eta_{ij}$ for $(i,j) \in [d] \times [n]$ is Bernoulli random variables which indicate if the element on position $(i,j)$ of $R$ is nonzero.
    Our starting point is the following bound on the two norm:
    \begin{equation}
        \|R\|_2^2 \leq \| R \|_1 \| R \|_\infty = \max_{i \in [d]} \sum_{j=1}^n \eta_{ij},
    \end{equation}
    where the inequality is Corollary~2.3.2 in \cite{golub2013MatrixComputations}, and the equality follows from the standard definitions of the $1$- and $\infty$-norms (see Section~2.3.2 in \cite{golub2013MatrixComputations}).
    Consequently, 
    \begin{equation} \label{eq:SJLT-spectral-norm-bound-1}
    \begin{aligned}
        \Pr[\| R \|_2^2 \geq t \mu] 
        &\leq \Pr \Big[ \max_{i \in [d]} \sum_{j=1}^n \eta_{ij} \geq t \mu \Big] 
        = \Pr\Big[ \bigcup_{i \in [d]} \Big\{ \sum_{j=1}^{n} \eta_{ij \geq t \mu} \Big\} \Big] \\
        &\leq \sum_{i=1}^d \Pr\Big[\sum_{j=1}^{n} \eta_{ij} \geq t \mu \Big] 
        = d \Pr\Big[\sum_{j=1}^{n} \eta_{1j} \geq t \mu \Big],
    \end{aligned}
    \end{equation}
    where the second inequality follows from subadditivity of measure.
    Chernoff's inequality (see Theorem~2.3.1 in \cite{vershynin2018high}) gives that
    \begin{equation} \label{eq:chernoff}
        \Pr\Big[\sum_{j=1}^{n} \eta_{1j} \geq t \mu \Big] \leq e^{-\mu} \Big(\frac{e}{t}\Big)^{t \mu}.
    \end{equation}
    Combining \cref{eq:SJLT-spectral-norm-bound-1} and \cref{eq:chernoff} gives the result in \cref{eq:SJLT-spectral-norm-bound}.

    If additionally $t > \max(e^2, \mu^{-1} \log(d/\delta)-1)$, then the bound in \cref{eq:SJLT-spectral-norm-bound} simplifies to
    \begin{equation} 
        \Pr[\| R \|_2^2 \geq t \mu] 
        \leq d e^{-\mu} \Big(\frac{e}{t}\Big)^{t \mu}
        \leq d e^{-\mu} e^{- t \mu} < \delta.
    \end{equation}
    
\end{proof}

The following lemma appeared as Theorem 5 in \cite{nelson2013osnap}. 

\begin{lemma}[SJLT satisfies subspace embedding property, Theorem 5 from \cite{nelson2013osnap}] \label{thm:sjlt-subspace-embedding}
    Given $R \sim \SJLT(n,d,\alpha)$, $V \in \C^{n \times r}$ and $\eps,\delta \in (0,1)$. If $\alpha = \Theta(\log^3(r/\delta)/\eps)$ and $d = \Omega(r\log^6(r/\delta)/\eps^2)$ then the following holds with probability at least $1-\delta$: 
    \begin{equation}
        |\normsq{R V x} - \normsq{V x} | < \eps \normsq{V x} \qquad \text{for all $x \in \R^r$.}
    \end{equation}
    
\end{lemma}

\section{Additional Experimental Results}

\Cref{tab:experiments_d_rank} shows the final $d$ selected for each method after adaptivity and the HSS rank, the rank of the largest off diagonal block as computed by the interpolative decomposition in the construction. Ideally, the difference between $d$ and the HSS rank should be less than $\Delta d = 64$ in our case meaning that the perfect amount of adaptive steps was taken. We observe that using Gaussian sketching operators and SJLT matrices with $\alpha = 2, 4 \text{ or } 8$ results in similar adaptive $d$ and HSS rank. When using SJLT matrices with $\alpha = 1$ the number of adaptive steps may be higher because the SJLT matrix is too sparse so new data about the original matrix is learned very slowly, requiring many more adaptive steps.

\begin{table}[h]
    \centering
    \resizebox{\textwidth}{!}{
    \begin{tabular}{|l|r|r|rrrrrr|rrrrrr|}
          \hline
                &     &    & \multicolumn{6}{c|}{Final $d$} & \multicolumn{6}{c|}{HSS rank}  \\ 
         Matrix & $\epsr$ & $n$ & G & S(1) & S(2) & S(4) & S(8) & H & G & S(1) & S(2) & S(4) & S(8) & H \\
         \hline
         \hline
         \multirow{9}{*}{Cov.}
         & \multirow{3}{*}{$10^{-2}$}
& $10^3$ &   128 &   128 &   128 &   128 &   128 & 128  &  97 &   102 &    96 &    97 &    97 & 97\\ 
& & $20^3$ &   256 &   256 &   256 &   256 &   256 & 256 & 180 &   179 &   175 &   167 &   159 & 179\\ 
& & $30^3$ &   384 &   384 &   320 &   384 &   384 & 384 & 247 &   253 &   221 &   248 &   235 & 239\\ 
        \cline{2-15}
            & \multirow{3}{*}{$10^{-4}$} 
& $10^3$ &   192 &   128 &   192 &   192 &   192 & 192 & 152 &   154 &   152 &   151 &   152 & 154\\ 
& & $20^3$ &   832 &   896 &   832 &   896 &   832 & 832 & 597 &   617 &   586 &   604 &   589 & 608\\ 
& & $30^3$ &  1984 &  2176 &  2112 &  2112 &  2112 & 2176 & 1472 &  1530 &  1520 &  1511 &  1470 & 1709\\ 
          \cline{2-15}
            & \multirow{3}{*}{$10^{-6}$}
& $10^3$ &   320 &   192 &   256 &   320 &   320 & 256 &  226 &   213 &   218 &   222 &   225 & 224\\ 
& & $20^3$ &  1088 &  1216 &  1216 &  1216 &  1280 & 1152 &  835 &   875 &   858 &   863 &   864 & 879\\ 
& & $30^3$ &  2816 &  2880 &  2880 &  2880 &  2880 &  3008 &  2128 &  2072 &  2079 &  2047 &  2073 & 2426\\ 
         \hline
         \hline
         \multirow{9}{*}{\shortstack[l]{QChem\\ Toeplitz}}
                 & \multirow{3}{*}{$10^{-2}$} 
& 10K &   128 &   128 &   128 &   128 &   128 & 128 &  11 &    10 &    10 &    10 &    11 & 10 \\ 
& & 20K &   128 &   128 &   128 &   128 &   128 & 128 &   13 &    13 &    12 &    12 &    11 & 12 \\ 
& & 40K &   128 &   128 &   128 &   128 &   128 &  128 &  12 &    13 &    12 &    12 &    13 & 11\\ 
        \cline{2-15}
            & \multirow{3}{*}{$10^{-4}$} 
& 10K &   128 &   128 &   128 &   128 &   128 &  128 &  18 &    20 &    17 &    17 &    16 & 17 \\ 
& & 20K &   128 &   128 &   128 &   128 &   128 & 128 &   18 &    19 &    20 &    18 &    20 & 19\\ 
& & 40K &   128 &   128 &   128 &   128 &   128 & 128 &   21 &    28 &    23 &    23 &    21 & 22\\ 
        \cline{2-15}
            & \multirow{3}{*}{$10^{-6}$} 
& 10K &   128 &   128 &   128 &   128 &   128 &  128 & 25 &    27 &    24 &    25 &    24 & 25 \\ 
& & 20K &   128 &   128 &   128 &   128 &   128 &  128 &  29 &    31 &    29 &    29 &    29 & 30\\ 
& & 40K &   128 &   128 &   128 &   128 &   128 & 128 &  36 &    40 &    37 &    35 &    35 & 34\\ 
         \hline
         \hline
         \multirow{9}{*}{\shortstack[l]{Scatt.\\ wave}}
         & \multirow{3}{*}{$10^{-2}$} 
& 5K &   192 &   192 &   192 &   192 &   192 &  576&  137 &   137 &   137 &   137 &   137 & 138 \\ 
& & 10K &   320 &   320 &   320 &   320 &   320 & 576 &  266 &   266 &   266 &   266 &   265 & 266\\ 
& & 20K &   576 &   576 &   576 &   576 &   576 & 576 &  523 &   523 &   524 &   523 &   524 & 522\\ 
        \cline{2-15}
       & \multirow{3}{*}{$10^{-4}$}  
& 5K &   192 &   192 &   192 &   192 &   192 &  576 &  146 &   147 &   145 &   145 &   144 & 146\\ 
& & 10K &   320 &   320 &   320 &   320 &   320 & 576 &  275 &   275 &   274 &   275 &   275 & 275\\ 
& & 20K &   640 &   640 &   640 &   640 &   640 & 576 &  538 &   538 &   535 &   536 &   538 & 529\\ 
        \cline{2-15}
       & \multirow{3}{*}{$10^{-6}$}
& 5K &   192 &   192 &   192 &   192 &   192 & 576 &  153 &   151 &   149 &   151 &   151 &  147\\ 
& & 10K &   320 &   320 &   320 &   320 &   320 & 576 &  284 &   284 &   281 &   282 &   284 & 275\\ 
& & 20K &   576 &   640 &   640 &   640 &   640 & 576 &  550 &   563 &   558 &   559 &   563 & 529\\ 
         \hline
         \hline
         \multirow{9}{*}{\shortstack[l]{3D\\ Poisson\\ front}}
         & \multirow{3}{*}{$10^{-2}$} 
& $100^2$ &   192 &  448  &  192 &   192 & 192   & 1856 &  158 &   350 &   159 &   156 &   156 & 168\\ 
& & $150^2$ &   384 &  1088 &   448 &   384 &   384 & 1856 &  245 & 916 & 295 & 247 & 241 & 253 \\
& & $200^2$ &   448 &  1536 &   704 &   512 &  512 & 1856 &  317 & 1333 & 414 & 320 & 318 & 335 \\
                 \cline{2-15}
       & \multirow{3}{*}{$10^{-4}$}  
& $100^2$ &   384 &   768 &   448 &  448  &  448 & 1856 &  282 &  601 &  294 &  278 &   276 & 279 \\ 
& & $150^2$ &   768 &   1536 &   832 &   832 &   832 & 1856 &  460 & 1188 & 526 & 505 & 496 & 430 \\
& & $200^2$ &   1088 &   2496 &   1280 &  1216 &  1216 & 1856 &  662 & 1936 & 800 & 766 & 762 & 569 \\

                 \cline{2-15}
       & \multirow{3}{*}{$10^{-6}$}
& $100^2$ &   576 &  896  &   640 &  576  &   640 & 1856 &  365 & 644 & 392 & 364 & 367 & 374 \\
& & $150^2$ &   1088 &  1856 &  1216 & 1152 &  1152 & 1856 &  645 & 1381 & 764 & 711 & 702 & 574\\
& & $200^2$ &  1536 &  2816 &  1728 & 1664 &  1728 & 1856 &   946 & 2093 & 1070 & 1019 & 1018 & 765 \\
         \hline
    \end{tabular}
    }
    \caption{Final $d$ and HSS rank for problems in \cref{tab:experiments}.}
    \label{tab:experiments_d_rank}
\end{table}

\Cref{tab:experiments_ranks_mpi} shows the final $d$ selected for each method after adaptivity and the HSS rank, the rank of the largest off diagonal block as computed by the interpolative decomposition in the construction for the parallel distributed experiments. Similarly to the above table, the difference between $d$ and the HSS rank should be less than $\Delta d = 256$ in our case meaning that the perfect amount of adaptive steps was taken. We observe that using Gaussian sketching operators and SJLT matrices with with $\alpha = 1$ the number of adaptive steps may be higher because the SJLT matrix is too sparse, requiring many more adaptive steps. While using SJLT with  $\alpha = 2, 4, \text{or}, 8$ yields similar results to the Gaussian matrices.

\begin{table}[ht]
    \centering
    \resizebox{\textwidth}{!}{
    \begin{tabular}{|l|r|r|rrrrr|rrrrr|}
          \hline
                &     &            & \multicolumn{5}{c|}{Final $d$} & \multicolumn{5}{c|}{HSS rank} \\ 
         Matrix & MPI size & $n$ & G & S(1) & S(2) & S(4) & S(8) & G & S(1) & S(2) & S(4) & S(8) \\
         \hline
         \hline
         \multirow{9}{*}{\shortstack[l]{Cov.}}
    & \multirow{3}{*}{8} & $20^3$ & 640 & 640 & 640 & 896 & 896 & 523 & 585 & 527 & 528 & 534 \\
    &   & $30^3$ & 1664 & 1920 & 1920 & 1920 & 1664 & 1204 & 1325 & 1212 & 1214 & 1166 \\
    &   & $35^3$ & 2944 & 3456 & 2944 & 3200 & 2944 & 2091 & 2275 & 2090 & 2013 & 2000 \\
        \cline{2-13}
    & \multirow{3}{*}{16} & $20^3$ & 640 & 640 & 640 & 896 & 896 & 523 & 585 & 527 & 528 & 534 \\
    &   & $30^3$ & 1664 & 1920 & 1920 & 1920 & 1664 & 1182 & 1325 & 1212 & 1214 & 1166 \\
    &   & $35^3$ & 2944 & 3456 & 2944 & 3200 & 2944 & 2077 & 2275 & 2090 & 2013 & 2000 \\
        \cline{2-13}
        
    & \multirow{3}{*}{32} & $20^3$ & 640 & 896 & 640 & 896 & 896 & 527 & 656 & 538 & 541 & 540 \\
    &   & $30^3$ & 1664 & 1920 & 1920 & 1920 & 1664 & 1196 & 1325 & 1212 & 1214 & 1166 \\
    &   & $35^3$ & 2688 & 3456 & 2944 & 3200 & 2944 & 1979 & 2275 & 2090 & 2013 & 2000 \\
        \cline{2-13}
         \hline
         \hline
         \multirow{9}{*}{\shortstack[l]{QChem\\ Toeplitz}}
    & \multirow{3}{*}{8} & 25000 & 512 & 512 & 512 & 512 & 512 & 24 & 23 & 20 & 19 & 21 \\
    &   & 50000 & 512 & 512 & 512 & 512 & 512 & 20 & 21 & 20 & 20 & 20 \\
    &   & 100000 & 512 & 512 & 512 & 512 & 512 & 25 & 27 & 25 & 26 & 25 \\
        \cline{2-13}
    & \multirow{3}{*}{16} & 25000 & 512 & 512 & 512 & 512 & 512 & 22 & 23 & 20 & 19 & 21 \\
    &   & 50000 & 512 & 512 & 512 & 512 & 512 & 20 & 21 & 20 & 20 & 20 \\
    &   & 100000 & 512 & 512 & 512 & 512 & 512 & 24 & 27 & 25 & 26 & 25 \\
        \cline{2-13}
    & \multirow{3}{*}{32} & 25000 & 512 & 512 & 512 & 512 & 512 & 22 & 23 & 20 & 19 & 21 \\
    &   & 50000 & 512 & 512 & 512 & 512 & 512 & 20 & 21 & 20 & 20 & 20 \\
    &   & 100000 & 512 & 512 & 512 & 512 & 512 & 24 & 27 & 25 & 26 & 25 \\
        \cline{2-13}
         \hline
          \hline
         \multirow{9}{*}{\shortstack[l]{3D\\Poisson\\front}}
    & \multirow{3}{*}{8} & $100^2$ & 512 & 512 & 512 & 512 & 512 & 278 & 349 & 279 & 279 & 279 \\
    &   & $150^2$ & 512 & 768 & 512 & 512 & 512 & 424 & 709 & 426 & 422 & 423 \\
    &   & $200^2$ & 768 & 1280 & 512 & 768 & 768 & 573 & 1119 & 564 & 569 & 569 \\
        \cline{2-13}
    & \multirow{3}{*}{16} & $100^2$ & 512 & 512 & 512 & 512 & 512 & 277 & 349 & 279 & 279 & 279 \\
    &   & $150^2$ & 512 & 768 & 512 & 512 & 512 & 425 & 709 & 426 & 422 & 423 \\
    &   & $200^2$ & 768 & 1280 & 512 & 768 & 768 & 576 & 1119 & 564 & 569 & 569 \\
        \cline{2-13}
    & \multirow{3}{*}{32} & $100^2$ & 512 & 512 & 512 & 512 & 512 & 278 & 349 & 279 & 279 & 279 \\
    &   & $150^2$ & 512 & 768 & 512 & 512 & 512 & 424 & 709 & 426 & 422 & 423 \\
    &   & $200^2$ & 768 & 1280 & 512 & 768 & 768 & 573 & 1119 & 564 & 569 & 569 \\
        \cline{2-13}
         \hline
    \end{tabular}
    }
    \caption{Final $d$ and HSS rank for problems in \cref{subsec:hss_mpi}. $G$ refers to sketching with a Gaussian sketching operator, $S(\alpha)$ to sketching with an SJLT matrix (block construction) with $\alpha$ nonzeros per row.}
    \label{tab:experiments_ranks_mpi}
\end{table}
\FloatBarrier
\section{HSS Algorithm Detailed Description}\label{appendix:algodescription}

\begin{algorithm2e}[H]
  \DontPrintSemicolon
  \SetAlgoLined
  \SetKwProg{Fn}{function}{}{}
  \SetKwFunction{compressadaptive}{HSSCompressAdaptive}
  \SetKwFunction{compressnodeadaptive}{CompressNodeAdaptive}
  \SetKwFunction{localsamples}{ComputeLocalSamples}
  \SetKwFunction{reducelocalsamples}{ReduceLocalSamples}
  \SetKwFunction{isleaf}{isleaf}
  \SetKwFunction{root}{root}
  \SetKwFunction{isroot}{isroot}
  \SetKwFunction{jlsketch}{JL-Operator}
  \SetKwFunction{child}{children}
  \SetKwFunction{level}{level}
  \SetKwFunction{cols}{cols}
  \SetKwFunction{ID}{ID}
  \SetKwFunction{local}{local}
  \SetKwFunction{rsinc}{RS-\inc}
  \SetKwBlock{Try}{try}{}
  \SetKwBlock{Fail}{catch}{}
  \SetKwData{UNTOUCHED}{\texttt{UNTOUCHED}}
  \SetKwData{COMPRESSED}{\texttt{COMPRESSED}}
  \SetKwData{PARTIALLYCOMPRESSED}{\texttt{PARTIALLY\_COMPRESSED}}
  \SetKwData{state}{state}
  \SetKwData{start}{start}
  \SetKw{goto}{goto}
  \SetKw{continue}{continue}
  \SetKwFunction{QR}{QR}
  \SetKw{break}{break}
  \SetKwFunction{minv}{min}
  \SetKwFunction{diag}{diag}

  \Fn{$H =$ \compressadaptive{$A$, $\mathcal{T}$, $d_0$, $\Delta d$}}{
    $d \gets d_0$; \quad $n \gets \cols(A)$ \\
    $R \gets $ \jlsketch{$d + \Delta d, n$} \\
    $S \gets AR$ \\
    \lForEach{$\tau \in \mathcal{T}$}{$\tau.\state \gets \UNTOUCHED$}
    \While{\root{$\mathcal{T}$}$.\state \neq \COMPRESSED$ {\bf and} $d < d_{\textup{max}}$}{
    \ForEach{$\tau \in \mathcal{T}$ {\bf in topological order}}{
      \uIf{$\tau.\state = \UNTOUCHED$}{
        \lIf{\isleaf{$\tau$}}{
           $D_\tau \gets A(I_\tau, I_\tau)$
        }\uElse{
          $\nu_1, \nu_2 \gets$ \child{$\tau$} \\
          $B_\tau \gets A(\widetilde{I}_{\nu_1}, \widetilde{I}_{\nu_2})$
        }
        $\iota \gets 1:d+\Delta d$
      }\lElse{
        $\iota \gets d+1:d+\Delta d$
      }
      \uIf{\isroot{$\tau$}}{
        $\tau.\state \gets \COMPRESSED$ \\
        \break
      }
      \lIf{\isleaf{$\tau$}}{
        $S_\tau(\,:\,, \iota) \gets S(I_\tau, \iota) - D_\tau \,\, R(I_\tau, \iota)$
      }\uElse{
        $S_\tau(\,:\,, \iota) \gets \begin{bmatrix} S_{\nu_1}(J_{\nu_1}, \iota) - B_{\tau} \,\, R_{\nu_2}( \, :\,, \iota) \\
        S_{\nu_2}(J_{\nu_2}, \iota) - B_{\tau}^* \,\, R_{\nu_1}( \, :\,, \iota) \end{bmatrix}$
      }
      \uIf{$\tau.\state \neq \COMPRESSED$}{
        % compute S and Q \\
        \uIf{$\tau.\state = \UNTOUCHED$}{
          $\{Q_{\tau}, \Omega_{\tau} \} \gets$ \QR{$S_{\tau}(\, :\,, 1:d)$}
        }
        $\widetilde{S} \gets S_{\tau}(\, :\,, d + 1 :d+\Delta d)$ \tcp*{last $\Delta d$ columns}
        $\widehat{S} \gets (I - Q_\tau Q_\tau^*) \widetilde{S}$  \label{line:block_Gram_Schmidt} \\
        $\epsa^\tau \gets \epsa/\level(\tau)$; \quad
        $\epsr^\tau \gets \epsr/\level(\tau)$ \\
        \uIf(\tcp*[f]{Eq.~\ref{eq:stopping_criteria-1}}){$\|\widehat{S}\|_{F} < \epsa^\tau$ {\bf or} $\|\widehat{S}\|_{F} < \epsr^\tau\|\widetilde{S}\|_{F}$}{
        $\goto$ line~\ref{line::incrementing_RRQR}  
        }
        % \tcp*{ortho.~against previous $Q$}
        $\{\widehat{Q}, \, \widehat{\Omega}\} \gets$ \QR{$\widehat{S}$} \label{line:inc_QR} \\
        $Q_\tau \gets \begin{bmatrix} Q_\tau & \widehat{Q} \end{bmatrix}$ \\
        \uIf(\tcp*[f]{}){\minv{\diag{$|\widehat{\Omega}|$}}$ < \epsa^\tau$ {\bf or} \minv{\diag{$|\widehat{\Omega}|$}}$ < \epsr^\tau |(\Omega_\tau)_{11}|$} {
          $\{ U_\tau^*,\,\, J_\tau \} \gets$ \ID{$S_\tau^*$, $\epsr^{\tau}$, $\epsa^{\tau}$} \label{line::incrementing_RRQR}\\
          $\tau.\state \gets \COMPRESSED$
          %\reducelocalsamples{$R$, $\tau$, $1:d+\Delta d$} 
        }\uElse{ 
           $\bar{R} \gets$ \jlsketch{$\Delta d, n$}  \tcp*{extending sketch}
           $d \gets d + \Delta d$; \quad $S \gets \begin{bmatrix} S & A \bar{R} \end{bmatrix}$; \quad $R \gets \begin{bmatrix} R & \bar{R} \end{bmatrix}$ \\
           $\tau.\state \gets \PARTIALLYCOMPRESSED$ \\
           \break
           %\goto \start \\
        }
      }
      \uIf{\isleaf{$\tau$}}{
        $R_\tau( \, :\,, \iota) \gets U_\tau^* \,\, R(I_\tau\,,\iota)$; \quad
        $\widetilde{I}_\tau \gets I_\tau(J_\tau)$
      }\uElse{
        $R_\tau( \, :\,, \iota) \gets U_\tau^* \begin{bmatrix} R_{\nu_1}( \, :\,, \iota) \\ R_{\nu_2}( \, :\,, \iota) \end{bmatrix}$; \quad
        $\widetilde{I}_\tau \gets \begin{bmatrix} I_{\nu_1} & I_{\nu_2} \end{bmatrix}(J_\tau)$
      }
    } % end forall postorder
    } %end while loop
    \Return $\mathcal{T}$
  }
  \caption{Adaptive HSS compression of $A \in \mathbb{C}^{n \times n}$ using cluster tree $\mathcal{T}$ with relative and absolute tolerances $\epsr$ and $\epsa$ respectively, see \Cref{tab::addfunctions} for helper function details.}
  \label{algo::HSScompressNodeAdaptive}
\end{algorithm2e}

%\yy{If the line numbers change remember to update them in section 2.1 and 2.2}

\begin{table}[!h]
  \begin{center}
    \begin{tabular}[c]{r|l}
      \hline
      \texttt{cols($A$)}            & number of columns in matrix $A$ \\
      \texttt{JL-Operator($d, n$)}    & a $d \times N$ matrix drawn from a JL Distribution \\
      \texttt{isleaf($\tau$)}   & \texttt{true} if $\tau$ is a leaf node, \texttt{false} otherwise \\
      \texttt{children($\tau$)} & a list with the children of node $\tau$, always zero or two \\
       \texttt{isroot($\tau$)}   & \texttt{true} if $\tau$ is a root node, \texttt{false} otherwise \\
       \texttt{$\{Q, \Omega\} \gets$ QR($S$)} & $S = Q\Omega$ where $Q$ is orthogonal, $\Omega$ is upper triangular \\
      \texttt{level($\tau$)} & level of node $\tau$, starting from $0$ at the root \\
      \texttt{$\{Y, J\} \gets$ ID($S, \varepsilon_r, \varepsilon_a$)} & interpolative decomposition: $S \approx S(:, J) Y$ \\
      \hline
    \end{tabular}
  \end{center}
  \caption{List of helper functions for \Cref{algo::HSScompressNodeAdaptive}.}
  \label{tab::addfunctions}
\end{table}
\FloatBarrier

Here we describe the steps to compress a symmetric HSS matrix $A$ with dimensions $4k \times 4k$ and HSS rank $r \ll k$ represented by a three level HSS tree shown in \cref{fig:hssthreetree} using \cref{algo::HSScompressNodeAdaptive}. Assume that $R$ has dimensions $4k \times l_1$. Initially, we compute $S = AR$ which has dimensions $4k \times l_1$.

We begin at the leaf level of the HSS tree where we can compress nodes one through four in parallel. We will compress the first node, corresponding to the first Hankel row block, whose rows we have highlighted in \cref{fig:hssleaf}. By symmetry this also corresponds to the columns of the first Hankel column block.   

\begin{figure}
    \centering
    \includegraphics[clip, trim=3cm 4cm 3cm 4cm, width=1.00\textwidth]{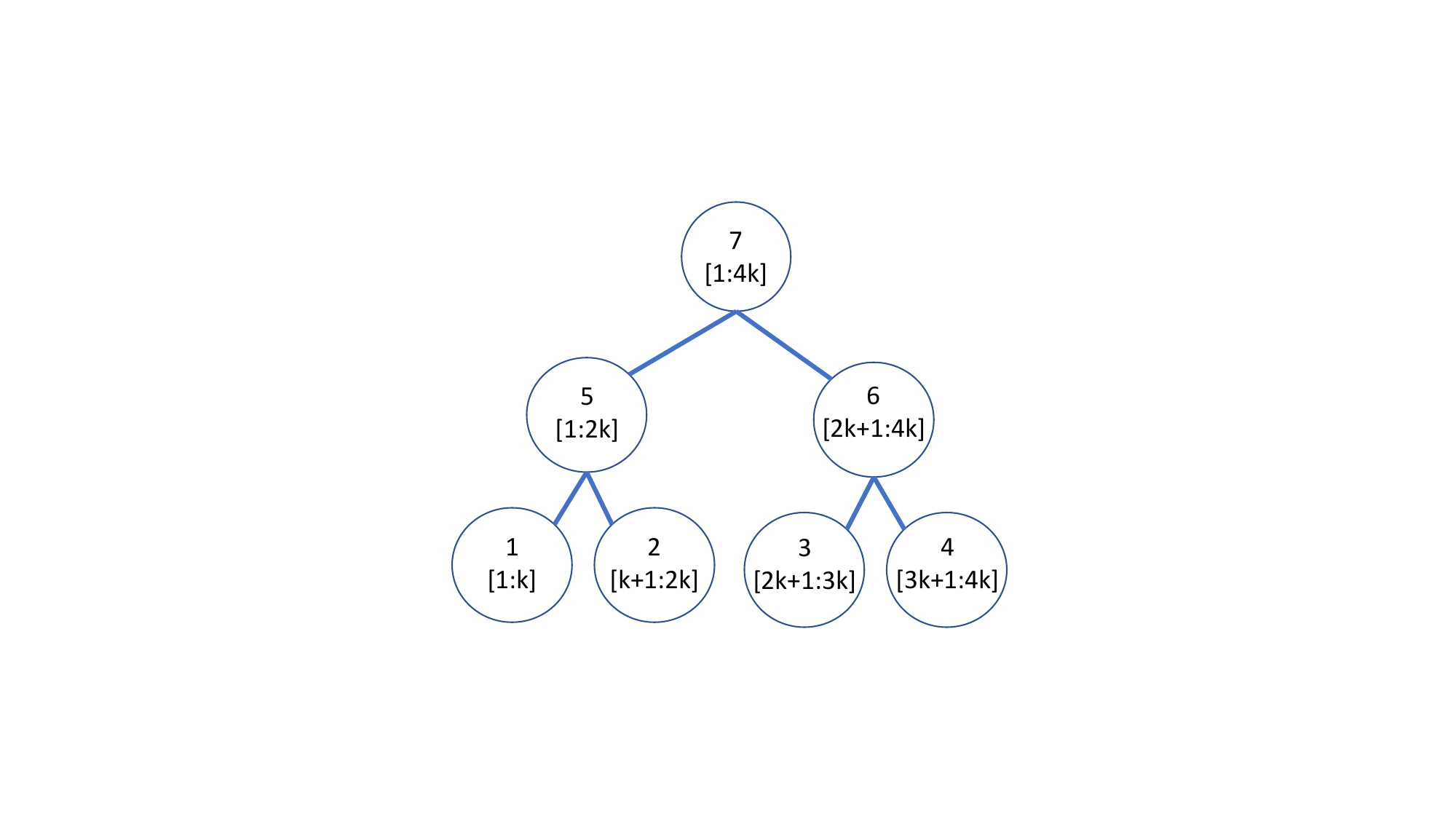}
    \caption{Three level HSS tree for our compression example with the nodes labeled and the corresponding indices in brackets.}
    \label{fig:hssthreetree}
\end{figure}

\subsection{Compression of a Leaf Node}

\begin{figure}
    \centering
    \includegraphics[clip, trim=3cm .5cm 3cm 1cm, width=1.00\textwidth]{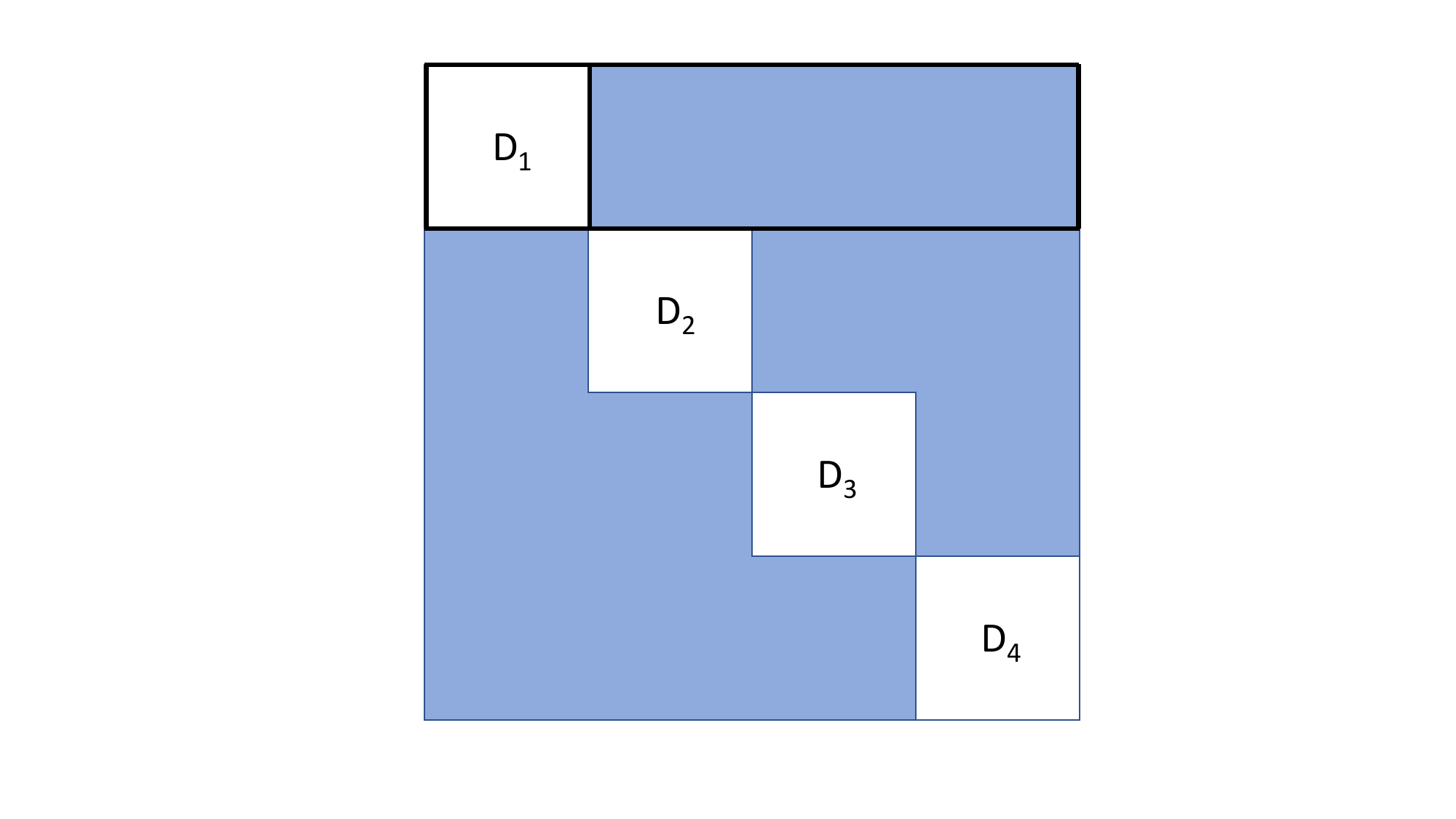}
    \caption{Leaf level of HSS tree with the first node rows in a box.}
    \label{fig:hssleaf}
\end{figure}

First, we store the dense diagonal matrix $D_1$ in our leaf node $1$ this is line 9 of the algorithm. Next, since we do not have the matrix $A$ but instead just the sketch $S = AR$ we must figure out what the local sketch of the Hankel row block $H_1 = A(1:k,1:4k \setminus 1:k ) = A(1:k,k+1:4k)$ is (the first $k$ rows excluding the dense diagonal). We compute a sketch of our Hankel row block $S^1_{\text{loc}} = [0, H_1] R$ by writing $[0, H_1] R = ([D_1, H_1] - [D_1, 0]) R = (A(1:k,:) - [D_1, 0]) R = S(1:k,:) - D_1 R(1:k,:)$ which is line 18 of \cref{algo::HSScompressNodeAdaptive}. 

Next, to compress our approximation of $H_1$ which is $S_1^{\text{loc}}$ with dimensions $k \times l_1$ lines 21-31 of \cref{algo::HSScompressNodeAdaptive} verify that $S_1^{\text{loc}}$ is a good enough approximation of $H_1$. For now, we will assume that it is and skip these lines. Later we will see how if the sketch is not accurate enough, we extend the sketching operator $R$ (lines 35-38) by appending columns to it which will require a small modification to the local sketches. We compute an interpolative decomposition of  $S_1^{\text{loc}}$ on line 32 of \cref{algo::HSScompressNodeAdaptive} such that  $S_1^{\text{loc}} \approx U_1 S_1^{\text{loc}}(J_1,:)$ where $U_1$ has dimensions $k \times r$ and $J_1$ is a subset of $r$ distinct indices in $[1:k]$. Then we set the state of node one to compressed (line 33). The interpolative decomposition cleverly gives us a low rank factorization for all of $H_1$ where $U_1$ could be thought of as a basis for the Hankel block and $J_1$ is an index set of rows which define the block. Since $S_1^{\text{loc}} = [0, H_1] R \approx U_1 S_1^{\text{loc}}(J_1,:) = U_1 [0, H_1](J_1,:) R$ and $R$ is full column rank with high probability we have that $[0, H_1] \approx  U_1 [0, H_1](J_1,:)$. So we have found a low rank factorization for the Hankel row block which we display in \cref{fig:hsscompressedH1}. 

\begin{figure}
    \centering
    \includegraphics[clip, trim=3cm 6cm 3cm 1cm, width=1.00\textwidth]{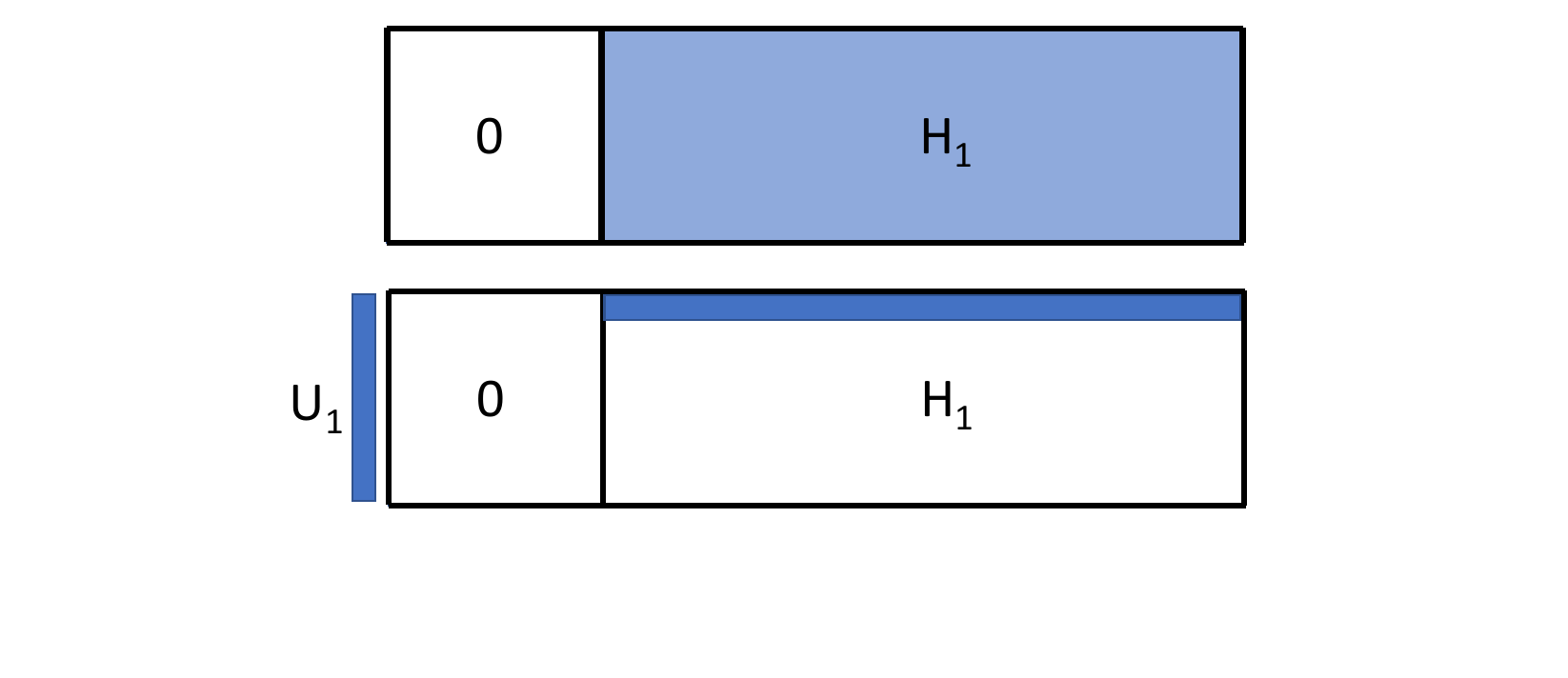}
    \caption{Compression of the first Hankel block $H_1$ into $U_1$, a basis matrix, and $r$ rows of the original Hankel block, denoted by the thin horizontal stripe (not necessarily the first $r$ rows) and indexed by index set $J_1$.}
    \label{fig:hsscompressedH1}
\end{figure}

We can now repeat this process for the rest of the leaf nodes which would result in matrices $U_2,U_3,U_4$ (dimensions $k \times r$) and index sets $J_2,J_3,J_4$ (of size $r$) being computed and stored. For the non-symmetric case we would also compress all of the leaf nodes for the column Hankel blocks as well. We display the result in \cref{fig:hssleavescompressed} where we additionally denote the low rank blocks $L_1$--$L_4$ which we would like to have compressed. 

\begin{remark}
The Hankel block does not need to be a contiguous nonzero block, for example $H_2 = [A(k+1:2k,1:k), 0, A(k+1:2k,2k+1:4k)]$ because $D_2$ is subtracted to compute $H_2$.
\end{remark}

\begin{figure}
    \centering
    \includegraphics[clip, trim=3cm .5cm 3cm 1cm, width=1.00\textwidth]{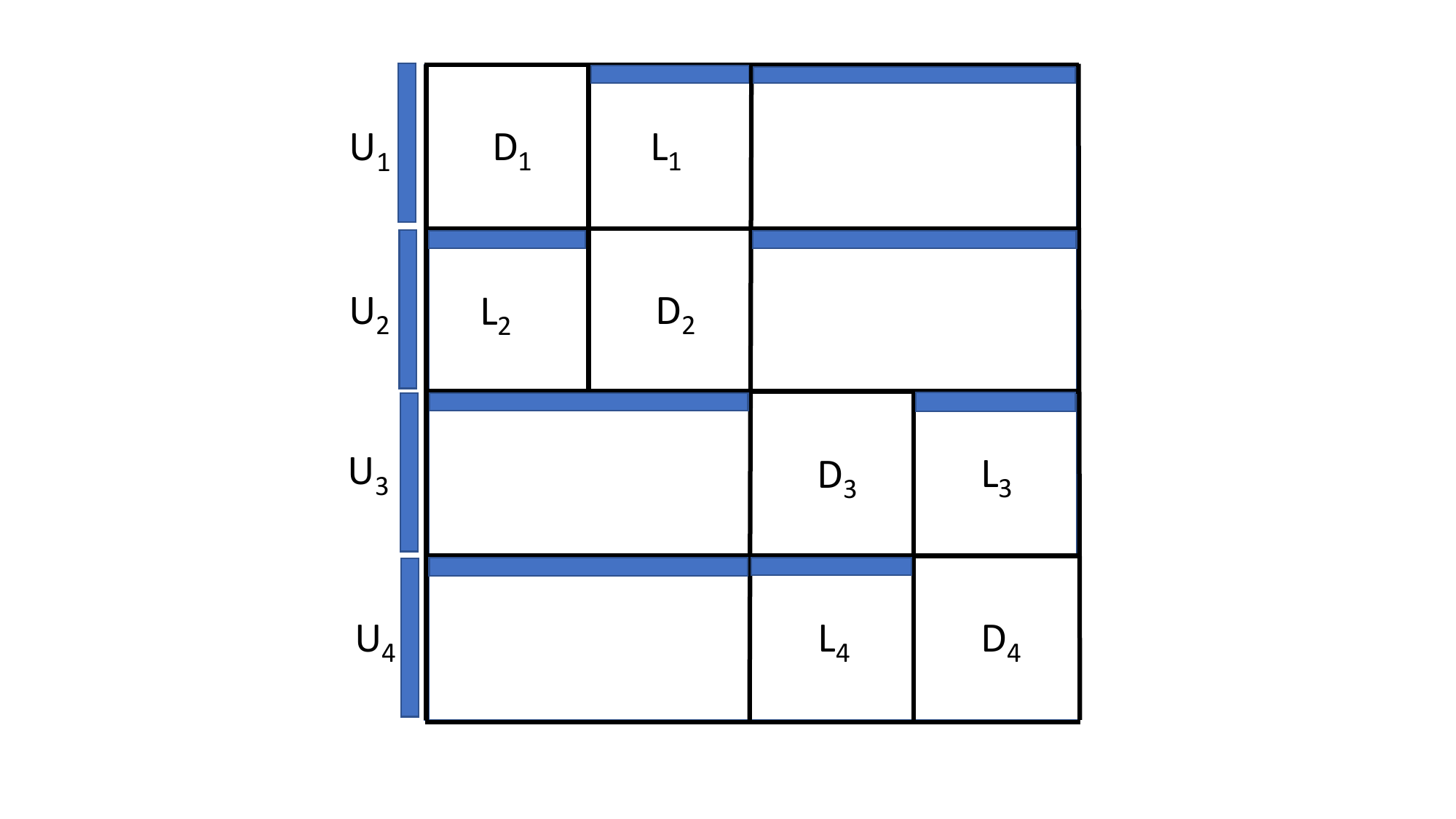}
    \caption{HSS matrix after all four row leaves have been compressed with the low rank blocks, $L_1$--$L_4$ sections listed .}
    \label{fig:hssleavescompressed}
\end{figure}
 
Next, We show that we have already computed a low rank factorization for $L_1$--$L_4$ based on the interpolative decompositions of both the row and column of the two Hankel blocks that intersect at the low rank block. We detail how to compress $L_1$ in \cref{fig:hsslowrank1compressed}. Since we have a row basis for $H_1$ we can just take the indices of the rows that intersect with $L_1$. So we have the factorization $L_1 \approx U_1 A(J_1,k+1:2k)$. Similarly, we have basis for the column Hankel block $H_2^T$ which intersected with $L_1$ because we assumed that our matrix $A$ was symmetric. So the column factorization for $L_1$ is the conjugate transpose of the row factorization for $L_2$ which we have already computed. Thus $L_1 \approx A(1:k,J_2) U^*_2$ we can rename $U_2^*$ as $V_2$ for clarity in the non-symmetric case where the second column Hankel block does not correspond to the conjugate transpose of the second row Hankel block. Combining the row and column factorizations, we have the low rank factorization $L_1 \approx  U_1 A(J_1,J_2) U_2^* = U_1 A(J_1,J_2) V_2$. Notice that we currently do not have $A(J_1,J_2)$, the small $r \times r$ matrix of entries of $A$. This will be queried and stored in the parent node in the next level of the algorithm (line 12 in \cref{algo::HSScompressNodeAdaptive}). For completeness we can factorize $L_2 \approx U_2 A(J_2,J_1) U_1^*,\; L_3 \approx U_3 A(J_3,J_4) U_4^*$ and $L_4 \approx U_4 A(J_4,J_3) U_3^*$.    

\begin{figure}
    \centering
    \includegraphics[clip, trim=3cm .5cm 3cm .5cm, width=1.00\textwidth]{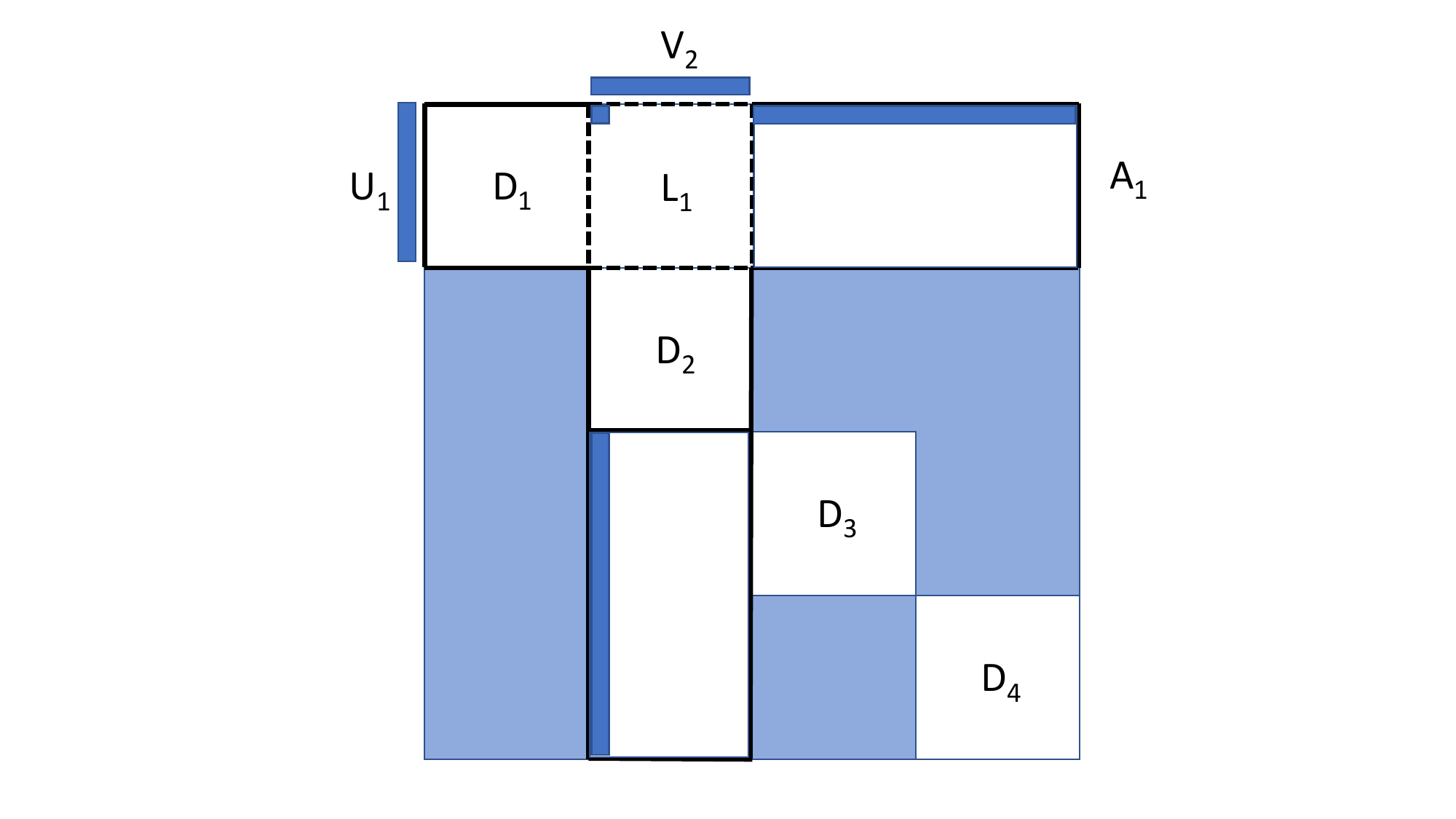}
    \caption{HSS matrix illustration of how the off diagonal low rank block $L_1$ is computed and stored.}
    \label{fig:hsslowrank1compressed}
\end{figure}

The final step that occurs at each leaf node is to compute $R^{\text{loc}}_i$ which corresponds to the sketching operator $R$ in the local column basis for the low rank block we have compressed. This will allow us to re-use the computation from our leaf nodes and subtract off the already compressed low rank blocks when trying to compress the parent nodes. Additionally, this allows us to leverage the nested basis property. So for the first leaf node, we compute and store $R_1^{\text{loc}} = U_1^* R(1:k,:)$. 

We have completed our compression for the first node, we store five variables: $1.\; D_1, \; 2.\; U_1, 3. \; J_1$ which is the dense diagonal block and what we use to represent the Hankel row block for rows $[1:k]$ and part of the low rank factorization for $L_1$ and we store $4. \; S_1^{\text{loc}}, 5. \; R_1^{\text{loc}}$ which we use to represent the sketch for the Hankel row block and the sketching operator for the Hankel row block in the column basis of $L_1$ which we use for the computation of the parent node.

\subsection{Compression of Internal Node}
We move on to compressing the second level of the HSS tree whose Hankel row blocks are shown in \cref{fig:parentHankel}. Before we describe the compression of $H_5$, we explain the \textbf{nested basis property} which all internal (non-leaf, non-root) nodes in the HSS tree use. This property explains the hierarchical in HSS matrices.

The nested basis property states that for a non-leaf Hankel block, $H_5$ with children nodes $H_1,\; H_2$ we can write a row (or column) basis $U^{\text{big}}_5$ of dimension $2k \times r$ as a product of the bases of $U^{\text{big}}_1,\; U^\text{big}_2$ (dimensions $k \times r$) of  $H_1,\; H_2$ respectively and a small matrix $U_5$ of dimension $2r \times r$: 

\begin{equation*}
    U^{\text{big}}_5 = \left[ \begin{array}{cc}
        U_1 & 0 \\
        0 & U_2
    \end{array}\right] U_5.
\end{equation*}

\begin{remark}
    For leaf node $i$, $U_i = U_i^{\text{big}}$.
\end{remark}

The intuition behind this property is that by constructing a basis $U^{\text{big}}_1$ for the first $k$ rows and  $U^{\text{big}}_2$ for the next $k$ rows, when we want to construct a basis $U^{\text{big}}_5$ for the $2k$ rows we should be able to use the basis information from our earlier constructions. When constructing HSS matrices we assume that this property holds.

Now that we have the nested basis property we can explain how this reduces the computation for the compression for node 5 (and any internal node) in \cref{algo::HSScompressNodeAdaptive}. We would like to have a sketch of $H_5$ depicted in \cref{fig:parentHankel} and compute $U_5$, of dimension $2r \times r$ . If we consider the matrix  $\left[ \begin{array}{c}
        S_1^{\text{loc}}\\
        S_2^{\text{loc}} 
    \end{array}\right]$ 
then we have an approximation for the block depicted in the top of \cref{fig:parent1precompress} because when we computed $S_1^{\text{loc}}$ and $S_2^{\text{loc}}$ we subtracted the diagonal blocks $D_1$ and $D_2$ respectively.  

\begin{figure}
    \centering
    \includegraphics[clip, trim=3cm .5cm 3cm .5cm, width=1.00\textwidth]{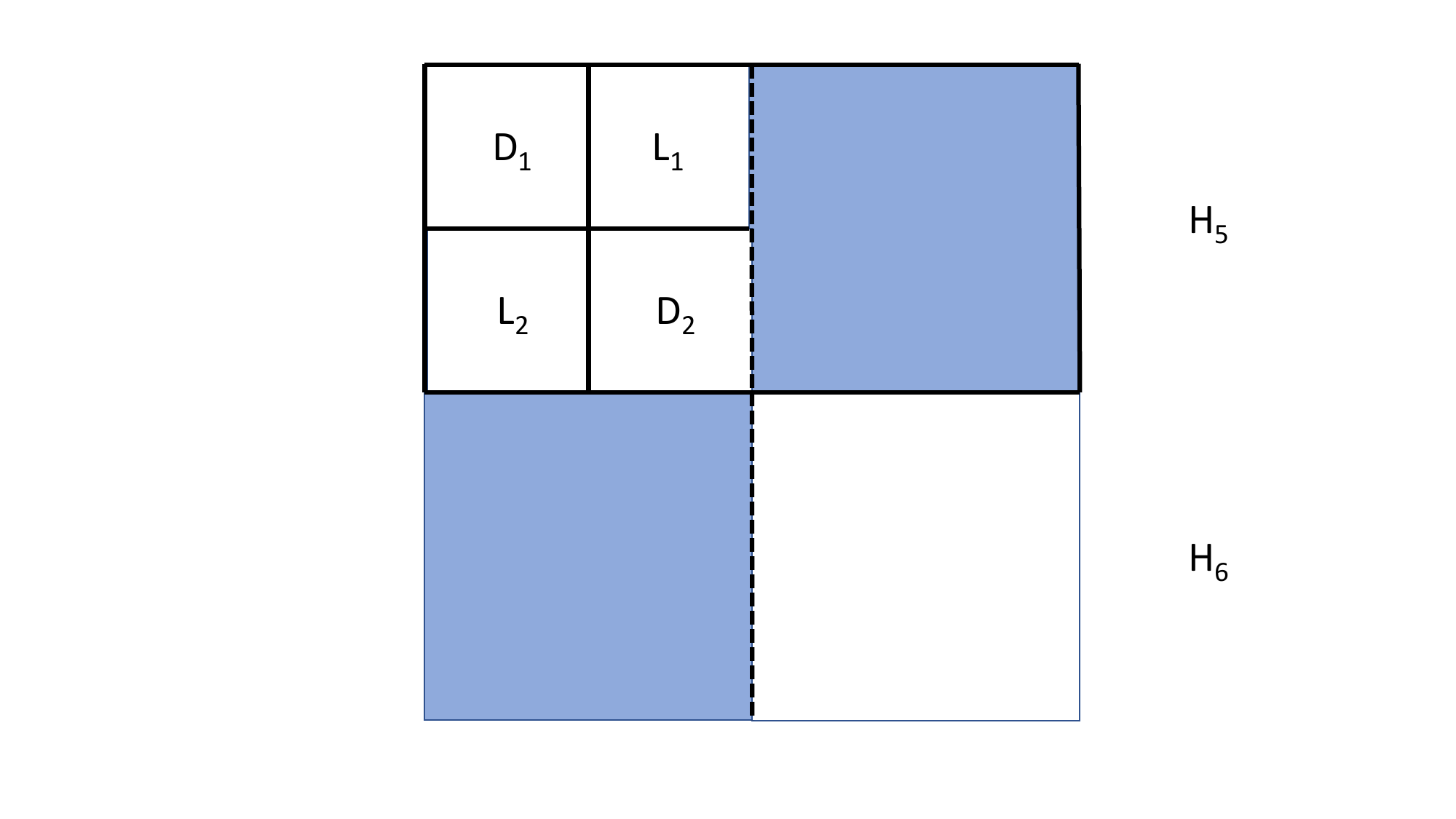}
    \caption{HSS matrix with the second level of row Hankel blocks highlighted in blue.}
    \label{fig:parentHankel}
\end{figure}

\begin{figure}
    \centering
    \includegraphics[clip, trim=3cm .5cm 3cm .5cm, width=1.00\textwidth]{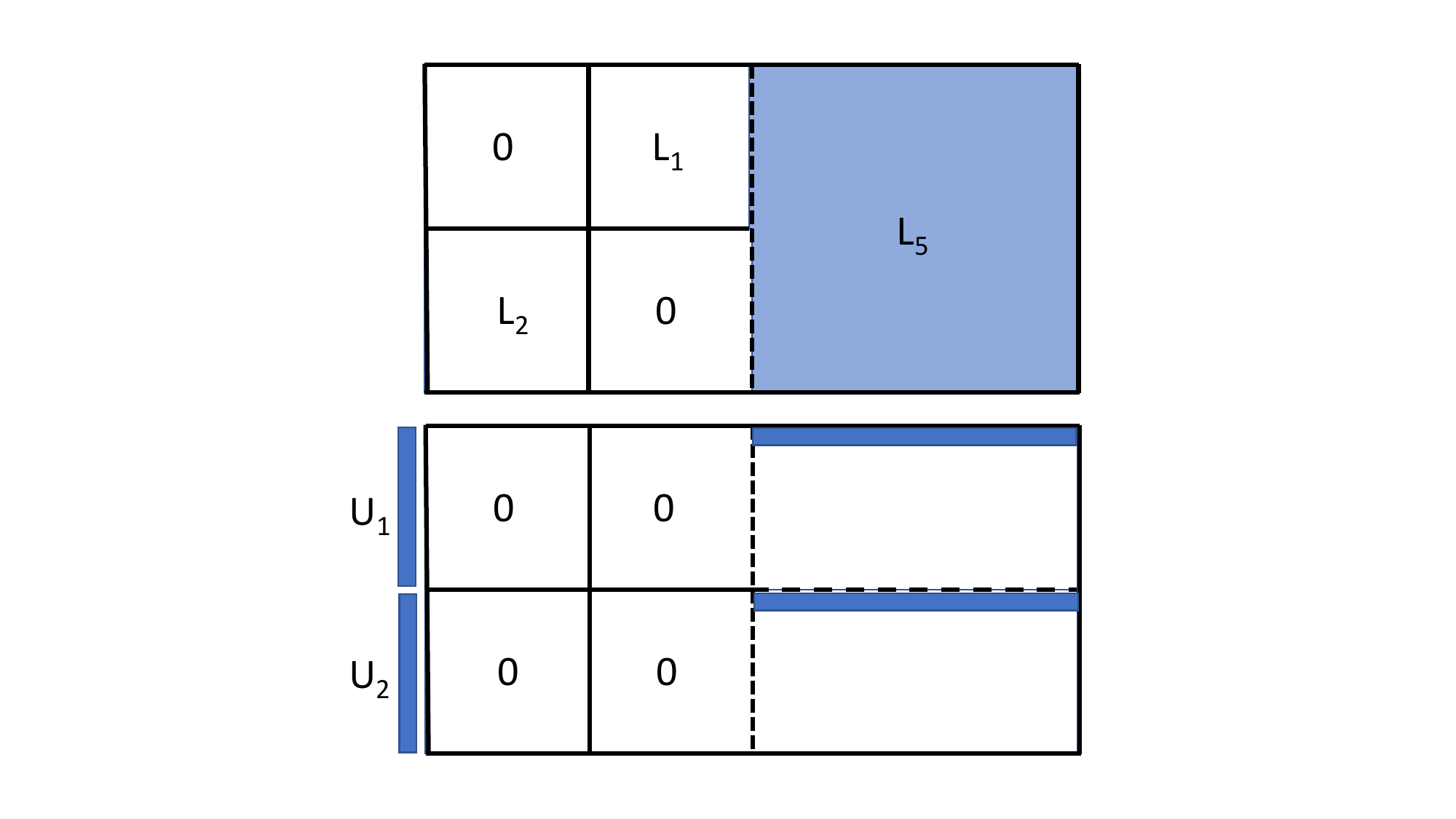}
    \caption{Node 5 row Hankel block being prepared for compression.}
    \label{fig:parent1precompress}
\end{figure}

We show how we use the nested basis property and information from the children nodes to compute a local sketch of $H_5$. We can subtract our compression of the low dimension blocks $L_1, \; L_2$ which we computed in the children nodes.

\begin{align}
    S_5 \nonumber & = \left( \left[ \begin{array}{ccc}
       0 & 0 & H_5(1:k,:)\\
       0 & 0 & H_5(k+1:2k,:)
    \end{array}\right] \right) R = 
    \left( A(1:2k,:) - \left[ \begin{array}{ccc}
        D_1 & L_1 & 0\\
        L_2 & D_2 & 0
    \end{array}\right] \right) R \\ &=  A(1:2k,:) R - \left[ \begin{array}{cc}
        D_1 & L_1\\
        L_2 & D_2 
    \end{array}\right] \left[ \begin{array}{c}
        R(1:k, :) \\
        R(k+1:2k, :)
    \end{array}\right] \nonumber \\ &= \nonumber
     \left[ \begin{array}{c}
       S_1^{\text{loc}}\\
        S_2^{\text{loc}} 
    \end{array}\right]  - \left[ \begin{array}{cc}
       0 & L_1\\
        L_2 & 0
    \end{array}\right] \left[ \begin{array}{c}
        R(1:k, :) \\
        R(k+1:2k, :)
    \end{array}\right] \\ & \approx \nonumber
    \left[ \begin{array}{cc}
       U_1^{\text{big}} & 0\\
        0 & U_2^{\text{big}} 
    \end{array}\right]   \left[ \begin{array}{c}
       S_1^{\text{loc}}(J_1,:)\\
        S_2^{\text{loc}}(J_2,:)
    \end{array}\right] - \left[ \begin{array}{c}
       U_1^{\text{big}} A(J_1,J_2)V_2^{\text{big}}R(k+1:2k, :)\\
       U_2^{\text{big}} A(J_2,J_1) V_1^{\text{big}} R(1:k, :)
    \end{array}\right] \\ &= \label{eq:localSketch}  
     \left[ \begin{array}{cc}
       U_1^{\text{big}} & 0\\
        0 & U_2^{\text{big}} 
    \end{array}\right]  \left( \left[ \begin{array}{c}
       S_1^{\text{loc}}(J_1,:)\\
        S_2^{\text{loc}}(J_2,:)
    \end{array}\right] - \left[ \begin{array}{c}
       A(J_1,J_2) R_2^{\text{loc}}\\
        A(J_2,J_1) R_1^{\text{loc}}
    \end{array}\right]\right) \\ & := \nonumber
    \left[ \begin{array}{cc}
       U_1^{\text{big}} & 0\\
        0 & U_2^{\text{big}}  
    \end{array}\right] S_5^{\text{loc}}
\end{align}
Since HSS matrices satisfy the nested basis property to compute a row basis for node 5 we use $S_5^{\text{loc}}$ which has dimensions $2r \times l_1$ and contains the nested basis prefactor seen in the second to last row of the above computation which generalizes to any internal HSS tree node. $S_5^{\text{loc}}$ corresponds to a sketch of the two dark blue horizontal strips in the bottom of \cref{fig:parent1precompress} and only requires information already computed in the children nodes. 

We go through the steps of compressing $H_5$ using \cref{algo::HSScompressNodeAdaptive}.  First, since node 5 is the parent node of nodes 1 and 2, it stores the small sub-blocks of $A$ used to compute $L_1$ and $L_2$ which in this case is $A(J_1,J_2)$ and $A(J_2,J_1)$, by symmetry only storing the $r \times r$ matrix $A(J_1,J_2)$ is required, line 12 of \cref{algo::HSScompressNodeAdaptive}. Then on line 20 of \cref{algo::HSScompressNodeAdaptive} a local sketch $S_5^{\text{loc}}$ as in \cref{eq:localSketch} is computed using the sub-blocks of $A$ that we just stored and the information in the children nodes. We then check if the local sketch, $S_5^{\text{loc}}$, is sufficient to approximate $H_5$ and adaptively increase the size of the sketching operator in lines 21-31 and lines 35-38. We discuss how this adaptation is done in the following section. Assuming that  $S_5^{\text{loc}}$ is sufficiently accurate, on line 32 of \cref{algo::HSScompressNodeAdaptive} we compute our basis $U_5$ and row indices $J_5$ in the nested basis defined by $U_1$ and $U_2$. Finally, on line 42 of \cref{algo::HSScompressNodeAdaptive} we compute a local sketching operator, $R_5^{\text{loc}}$, in the basis of $U_5$ which we will use to subtract the block which we have compressed in higher levels of the tree. So we have computed and stored: $1.\; A(J_1,J_2), \; 2. \; S_5^{\text{loc}},\; 3. \; U_5,\; 4. \; J_5,$ and $ 5. \; R_5^{\text{loc}}$ which are the five components that define an internal node. 

We can similarly compress $H_6$ which would now give us all the information to compress $L_5$ and $L_6$ by symmetry then move up to the root node. 
\begin{remark}
    When compressing the root node we do not do any compression but instead store the two $r \times r$ blocks of $A$ ($A(J_5,J_6)$ and $A(J_6,J_5)$ here) that are required to compute the low rank factorization for the two largest low rank off diagonal blocks ($L_5$ and $L_6$ here).
\end{remark}

\subsection{Adaptation}
At each non-root node of the HSS tree we verify that the sketch of our current node, $S_i^{\text{loc}}$, is sufficiently accurate before we compress it. If $S_i^{\text{loc}}$ is sufficiently accurate, which is checked by the computation and stopping criteria on lines 21-31 of \cref{algo::HSScompressNodeAdaptive} then we can compress node $i$, otherwise we increase the size of our global sketching operator and global sketch on lines 35 and 36 (from $l_1$ to $l_1+\Delta d$ in our example). We then mark the state of the current node, $i$, as partially compressed and restart our compression loop for all of the nodes. 

For the compressed nodes we will update their local sketches and sketching operators to have $l_1+\Delta d$ instead of just $l_1$ columns. This operation is computed in \cref{algo::HSScompressNodeAdaptive} as follows. First on line 14 we set the columns we will be modifying as the final $\Delta d$ that we added to the global sketch and sketching operator in line 36. Then on lines 18-20 we update the local sketch information, finally on lines 39-42 the local sketching operators are updated.   

For the one partially compressed node we will update the sketching operator as for the compressed nodes but we will also check the stopping criteria on lines 27 and 31. If either is met then node $i$ can now be compressed and the algorithm can continue. Otherwise, lines 35-37 will trigger again, expanding the global sketch and sketching operator then marking node $j$ as partially compressed again. Finally, for uncompressed nodes we do not need to update anything, we will use the updated sketching operator and sketches. For a detailed discussion of why we use the stopping criteria on lines 27 and 31 we refer the reader to \cref{sec:adaptivestoppingcrit}.

\end{document}